\documentclass[11pt]{article}

\makeatletter
\DeclareRobustCommand{\qed}{%
  \ifmmode % if math mode, assume display: omit penalty etc.
  \else \leavevmode\unskip\penalty9999 \hbox{}\nobreak\hfill
  \fi
  \quad\hbox{\qedsymbol}}
\newcommand{\openbox}{\leavevmode
  \hbox to.77778em{%
  \hfil\vrule
  \vbox to.675em{\hrule width.6em\vfil\hrule}%
  \vrule\hfil}}
\newcommand{\qedsymbol}{\openbox}
\newenvironment{proof}[1][\proofname]{\par
  \normalfont
  \topsep6\p@\@plus6\p@ \trivlist
  \item[\hskip\labelsep\itshape
    #1.]\ignorespaces
}{%
  \qed\endtrivlist
}
\newcommand{\proofname}{Proof}
\makeatother

\usepackage[utf8]{inputenc}

\usepackage{color}
\usepackage{latexsym}
\usepackage{dsfont}
\usepackage{amssymb}
\usepackage{comment}
\usepackage{graphicx}
\usepackage{amsmath,amsfonts,amssymb,theorem,euscript,array,enumerate,amsfonts,mathrsfs}
\usepackage{hyperref}
\usepackage{appendix}
\usepackage[T1]{fontenc}
\usepackage{babel}
\numberwithin{equation}{section}
\usepackage{bbm}
\usepackage{subfigure}
\usepackage{color}

\usepackage{stmaryrd}

\usepackage[algo2e,ruled,vlined]{algorithm2e} 
 \usepackage{algorithm}
 \usepackage{algorithmicx}
\usepackage{algpseudocode}
\usepackage{ marvosym }
\usepackage{hyperref}

\usepackage{mathtools}
\mathtoolsset{showonlyrefs}

\usepackage{comment}

\usepackage{tikz}
\usetikzlibrary{fit,matrix,chains,positioning,decorations.pathreplacing,arrows}

\usepackage{mathtools}
\mathtoolsset{showonlyrefs}

\usepackage[a4paper,margin=1in,heightrounded]{geometry}

\usepackage{algpseudocode}
\usepackage{algorithm}

\usepackage[normalem]{ulem}

\def \trans{^{\scriptscriptstyle{\intercal}}}

\def \Inf{\displaystyle\inf}

\def \Max{\displaystyle\max}

\def \b1{\bf{1}}

\def \I{\mathbb{I}}

\def \R{\mathbb{R}}

\def \E{\mathbb{E}}
\def \F{\mathbb{F}}

\def \P{\mathbb{P}}

\def \bd{\mathbb{d}}

\def \bd{\boldsymbol{d}}

\def \mra{\mathrm{a}}

\def \d{\mathrm{d}}

\def\esssup_#1{\underset{#1}{\mathrm{ess\,sup\, }}}

\def\argmin_#1{\underset{#1}{\mathrm{argmin\, }}}
\def\argmax_#1{\underset{#1}{\mathrm{argmax\, }}}

\def\dmu{\frac{\delta}{\delta m}}
\def\dm#1{\frac{\delta}{\delta m}}

\def \Ac{{\cal A}}
\def \Bc{{\cal B}}

\def \Fc{{\cal F}}

\def \Hc{{\cal H}}
\def \Ic{{\cal I}}

\def \Lc{{\cal L}}
\def \Pc{{\cal P}}

\def \Nc{{\cal N}}

\def \Sc{{\cal S}}

\def \Uc{{\cal U}}

\def \Wc{{\cal W}}

\def \eps{\varepsilon}

\def \ep{\hbox{ }\hfill$\Box$}

\def\reff#1{{\rm(\ref{#1})}}

\def\bd{{\bf d}}

\def \mra{\mathrm{a}}

\def \d{\mathrm{d}}

\def\beqs{\begin{eqnarray*}}
\def\enqs{\end{eqnarray*}}
\def\beq{\begin{eqnarray}}
\def\enq{\end{eqnarray}}

\newcommand{\red}[1]{\textcolor{red}{#1}}

\def\red#1{{\color{red}#1}}

\newtheorem{Theorem}{Theorem}[section]
\newtheorem{Definition}{Definition}[section]
\newtheorem{Proposition}{Proposition}[section]
\newtheorem{Assumption}{Assumption}[section]
\newtheorem{Lemma}{Lemma}[section]
\newtheorem{Corollary}{Corollary}[section]
\newtheorem{Remark}{Remark}[section]
\newtheorem{Example}{Example}[section]
\numberwithin{equation}{section}

\title{
Mean-field control of non exchangeable systems}

\author{Anna De Crescenzo\footnote{LPSM, Universit\'e Paris Cit\'e and Sorbonne University, \sf decrescenzo at lpsm.paris} 
\and Marco Fuhrman \footnote{Dipartimento di Matematica, Università degli Studi di Milano, {\sf marco.fuhrman at unimi.it} This author was a member of  INdAM-GNAMPA when this work was conceived.} 
\and  Idris Kharroubi \footnote{LPSM, Sorbonne University and  Universit\'e Paris Cit\'e, \sf kharroubii at lpsm.paris. 
The work of this author is  partially supported by Agence Nationale de la Recherche (ReLISCoP grant ANR-21-CE40-0001)  } 
\and 
Huy\^en Pham\footnote{LPSM, Universit\'e Paris Cit\'e and Sorbonne University, \sf pham at lpsm.paris. 
The work of this author is  partially supported by the BNP-PAR Chair ``Futures of Quantitative Finance", and the 
Chair Finance \& Sustainable Development / the FiME Lab (Institut Europlace de Finance)}}

\date{}

\begin{document}

\maketitle

\begin{abstract}
%To be completed
We study the optimal control of mean-field systems with heterogeneous and asymmetric interactions. This leads to considering a family of controlled Brownian diffusion processes with dynamics depending on the whole collection of marginal probability laws. We prove the well-posedness of such systems and define the control problem together with its related value function. We next prove a law invariance property for the value function which allows us  to work on the set of collections of probability laws. We show that the value function satisfies a dynamic programming principle (DPP) on the flow of collections of probability measures. We also derive a chain rule for a class of regular functions along the flows  of collections of marginal laws of diffusion processes. Combining the DPP and the chain rule, we prove that the value function is a viscosity solution of a Bellman dynamic programming equation in a 
$L^2$-set of Wasserstein space-valued functions.  
\end{abstract}

%\begin{comment}

%\tableofcontents

\vspace{5mm}


\vspace{5mm}


\newpage

\section{Introduction}

The study of large population and complex systems is a major question in mathematical modelling  with various applications in our society like e.g.  social networks, power grid networks,  financial markets, lightning networks. 
Classical models consider mean-field systems with symmetric particles and homogenous interaction: denoting by $X^{i,N}$ the state of the $i$-th particle (agent/player) in the $N$-population,  it interacts with the other 
particles via the empirical measure: $\mu^N$ $=$ $\frac{1}{N} \sum_{j=1}^N \delta_{X^{j,N}}$, and therefore the system of $N$-particle $X^N$ $=$  $(X^{1,N},\ldots,X^{N,N})$ is exchangeable.  The macroscopic behavior of the  limiting mean-field system when the number 
$N$ of agents goes to infinity leads to an equation with a representative state depending on its probability distribution, called McKean-Vlasov (MKV) equation, and has been extensively studied in the context of mean-field game  (MFG where agents are in strategic interaction), 
of mean-field control (MFC with cooperative interaction among agents following a center of decision). We refer 
to the lectures of P.L. Lions \cite{liocolfra} at Coll\`ege de France, and to the monographs by  Bensoussan, Frehse and Yam \cite{benetal13} and Carmona and Delarue \cite{cardel18a}, \cite{cardel18b} (and the references therein) for 
a comprehensive treatment of the mathematical tools (It\^o's formula along flow of probability measures, maximum principle, forward backward SDE of MKV type, dynamic programming, Master Bellman equation in the Wasserstein space) used in the optimization 
problems  from  MFG and MFC.

In this paper, motivated by more realistic applications in complex networks,  we deal  with large systems of agents whose interactions are not necessarily symmetric and possibly heterogeneous,   
hence leading to non exchangeable systems as in \cite{jabinetal21}.  For example, the theory of graphons (see \cite{lovasz_large_2012})  provides a framework for such modelling, and has been used  in  \cite{caihua20}, \cite{auretal22}, \cite{lacker_label-state_2023},  \cite{berrak24}  for the MFG problem with heterogeneous agents. 
In the context of cooperative interaction,  the controlled dynamics of the particle $i$ $\in$ $\llbracket 1,N\rrbracket$  in  a graphon system  with heterogeneous drift and volatility coefficients  is driven by
\begin{align}
\d X_t^{i,N} &=  b\big(u_i,X_t^{i,N}, \alpha_t^{i,N},\frac{1}{N_i}\sum_{j=1}^N  G(u_i,u_j) \delta_{X_t^{j,N}} \big) \d t \\
& \quad \quad + \;   \sigma\big(u_i,X_t^{i,N}, \alpha_t^{i,N},\frac{1}{N_i}\sum_{j=1}^N  G(u_i,u_j) \delta_{X_t^{j,N}} \big) \d W_t^{u_i}, 
\end{align}
where $u_i$ $=$ $i/N$ $\in$ $U$ $:=$ $[0,1]$  is the label of particle $i$ in the $N$-system, $G$ is a graphon, i.e., a measurable function from $U\times U$ into $U$, measuring the weight of interaction between particles, $N_i$ $=$ $\sum_{j=1}^N G(u_i,u_j)$ is the degree of interaction of particle $i$, 
$W$ $=$ $(W^u)_{u\in U}$,  is a family of i.i.d. Brownian motion on some probability space $(\Omega,\Fc,\P)$,  and $\alpha^{i,N}$ is  a control process valued in some action space $A$, 
%an  $\F^{W^{u_i}}$-progressively measurable control process valued in some action space $A$, and  
followed by agent $i$.   The aim of a center of decision in such a framework  is to 
minimize over  $(\alpha^{1,N},\ldots,\alpha^{N,N})$ a cost functional in the form 
\begin{align}
\frac{1}{R_N} \sum_{i=1}^N  r(u_i)  \E \Big[ \int_0^T f\big(u_i,X_t^{i,N},\alpha_t^{i,N},\frac{1}{N_i}\sum_{j=1}^N  G(u_i,u_j) \delta_{X_t^{j,N}} \big) \d t \\
\quad + \;  g\big(u_i,X_T^{i,N},\frac{1}{N_i}\sum_{j=1}^N  G(u_i,u_j) \delta_{X_T^{j,N}} \big) \Big],  
\end{align} 
where $r(u_i)$ is the weight of particle $i$ in the social cost criterion, $R_N$ $:=$ $\sum_{i=1}^N r(u_i)$, and 
$f,g$ are running/terminal costs, possibly heterogeneous (i.e. depending on the label of particles). 
%in the population). 

When the number of agents $N$ goes to infinity, and in line with the convergence results in \cite{lac17}, \cite{dje22},  \cite{caretal23}, \cite{dauetal23} for MFC with homogeneous interaction corresponding to  $G$ $\equiv$ $1$, and propagation of chaos for graphon mean-field systems in \cite{bayetal23},  \cite{copdecpha24}, we formally expect to obtain the formulation of a graphon MFC as the problem of minimizing over 
a collection of control process $\alpha$ $=$ $(\alpha^u)_{u \in U}$ valued in $A^U$ a cost functional in the form 
 \begin{align} \label{gain-Graphon}
\int_U  \E \Big[& \int_0^T f\big(u,X_t^u,\alpha_t^u,    \int_0^1 \frac{G(u,v)}{\|G(u,.)\|_{_1}}  \P_{ X_t^{v}} \d v \big) \d t \\
& \qquad  + \;  g\big(u,X_T^u,  \int_0^1 \frac{G(u,v)}{\|G(u,.)\|_{_1}}  \P_{ X_T^{v}} \d v \big) \Big] r(u) \d u, 
 \end{align} 
where $\|G(u,.)\|_{_1}$ $:=$ $\int_0^1G(u,v) \d v$, $\P_Y$ denotes the probability law of a random variable $Y$ on $(\Omega,\Fc,\P)$, and $X$ $=$ $(X^u)_{u\in U}$  is a collection of controlled state process  in $\R^d$ governed by 
\begin{align} \label{X-graphon}
\d X_t^u &= \;  b\big(u,X_t^u,\alpha_t^u,  \int_0^1 \frac{G(u,v)}{\|G(u,.)\|_{_1}}  \P_{ X_t^{v}} \d v \big) \d t  \\
& \quad \quad + \;  \sigma\big(u,X_t^u,\alpha_t^u,  \int_0^1 \frac{G(u,v)}{\|G(u,.)\|_{_1}}  \P_{ X_t^{v}} \d v \big) \d W_t^u.  \end{align}

\paragraph{Our work and contributions.} Inspired by the above discussion, we extend the graphon MFC formulation in \eqref{gain-Graphon}-\eqref{X-graphon}, and introduce a class of mean-field control for non exchangeable systems 
by considering a collection of controlled state process $X$ $=$ $(X^u)_{u\in U}$ governed  by 
\begin{align} \label{dynXu}
\d X_t^u  &= \; b\big(u,X_t^u,\alpha_t^u, (\P_{X_t^v})_{_v},(\P_{\alpha_t^v})_{_v} \big) \d t +  \sigma\big(u,X_t^u,\alpha_t^u, (\P_{X_t^v})_{_v},(\P_{\alpha_t^v})_{_v} \big) \d W_t^u.
\end{align}
Then,  the MFC problem is to minimize over the collection of control process  $(\alpha^u)_{u\in U}$ the cost functional over a finite horizon 
\begin{align} \label{defJ}
J(\alpha) &= \; \int_U \E \Big[ \int_0^T f\big(u,X_t^u,\alpha_t^u, (\P_{X_t^v})_{_v},(\P_{\alpha_t^v})_{_v} \big) \d t  + g\big(u,X_T^u, (\P_{X_T^v})_{_v}\big)  \Big]     \lambda(\d u),
\end{align}
where $\lambda$ is a positive finite measure on $U$ that specifies the weight of the agents/particles  in the social cost criterion. In this general modeling, we see from \eqref{dynXu} that the state processes of the agents $u$ $\in$ $U$ are independent, but not identically distributed, and they interact through the whole collection of their probability laws on the state $(\P_{X_t^v})_{_{v\in U}}$ and control $(\P_{\alpha_t^v})_{_{v\in U}}$. Notice also the measurability issues $(u,\omega)$ $\in$ $U\times\Omega$ $\mapsto$ $X^u(\omega)$ due to the independence of the continuum of state process, see \cite{sun06}. In particular, this probabilistic independence property prevents from proving a measurability property with respect to the product of $\sigma$-algebrae on $U$ and $\Omega$. As a consequence, we are not able to do computations and derive estimates involving the variables $u\in U$ and $\omega\in \Omega$ on the state process anymore. This issue leads us to focus on the collection of laws $(\P_{X_t^v})_{_{v\in U}}$ as a state variable.

Our first task is a rigorous formulation of the MFC for coupled SDEs $(X^u)_{u\in U}$ with the admissible set of controls $\alpha$ $=$ $(\alpha^u)_{u\in U}$, and we precise in the next section the assumptions on the coefficients $b,\sigma$, $f,g$  for ensuring the existence and uniqueness of a solution to \eqref{dynXu} given initial conditions, and the well-posedness of the cost functional in \eqref{defJ}.  
Our main goal is to provide an analytic characterization of the solution to this novel class of control problems by adopting a dynamic programming approach. This will be achieved through the following steps. 

We define the value function associated to our MFC problem, and by a law invariance property, it is a 
function defined on $[0,T]\times L_\lambda^2(\Pc_2(\R^d))$, where $L_\lambda^2(\Pc_2(\R^d))$ is the collection 
$(\mu^u)_{u\in U}$ of square integrable probability measures on $\R^d$ s.t. 
$\int_U \int_{\R^d} |x|^2 \mu^u(\d x) \lambda(\d u)$ $<$ $\infty$. From the flow property of the solution to \eqref{dynXu}, we then state directly the dynamic programming principle (DPP) for the value function. 

Next, we aim to derive an It\^o's formula for the collection of flow of probability measures 
$(\P_{X_t^u})_{t\in [0,T],u\in U}$ that extends the It\^o formula for flow of measures  in 
\cite{buckdahn_mean-field_2017}, see also Chapter 5 in \cite{cardel18a}. This is obtained with the notion of 
linear (or flat)  derivative of a function defined on $L_\lambda^2(\Pc_2(\R^d))$, and by standard method of time discretization.  

Once we have the DPP and the It\^o's formula, we can derive as usual the associated Bellman equation that we express  in a suitable and unified form that takes into consideration both dependence of the dynamics and reward on the collection of probability laws on state/control. A verification theorem is shown for classical solutions to the Bellman equation, and in general, we prove the (discontinuous) viscosity property of the value function to the Bellman equation. Uniqueness of viscosity solutions is beyond the scope of this paper and postponed for future research.

\paragraph{Outline.} The plan of this paper is organized as follows.  We formulate the mean-field control problem for non exchangeable systems in Section \ref{secMFC}. In Section \ref{secDPP}, we show the law invariance property of the value function and the dynamic programming principle. Section \ref{secIto} is devoted to It\^o formula along a collection of flow of probability measures with the notion of linear derivative on  $L_\lambda^2(\Pc_2(\R^d))$. We then derive in Section \ref{secHJB} the dynamic programming Bellman equation for the value function, and its viscosity property. Finally, we postpone some proofs to  Appendix \ref{secappenthmexuuniq}, and  collect in Appendix \ref{secappen} 
some auxiliary results dealing with measurability questions that are needed in the proofs of some results. 
 
\paragraph{Notations.}
\begin{itemize}
\item We denote by $x\cdot y$ the scalar product between vectors $x,y$, and by $A:B$ $=$ ${\rm tr}(A B\trans)$ the inner product of two matrices $A,B$ with compatible dimensions, where $B\trans$ is the transpose matrix of $B$.
\item  Throughout the paper, $T>0$ denotes a fixed time horizon. For and integer $d\geq1$ and $[a,b]\subset [0,T]$ the space $C([a,b];\R^d)$ of continuous functions $[a,b]\to \R^d$ will be denoted simply $C_{[a,b]}^d$. It will be given the supremum norm and the corresponding distance and Borel sets. For $w\in C^d_{[a,b]}$ and $[a',b']\subset [a,b]$,  $w_{[a',b']}\in C^d_{[a',b']}$ stands for the restriction of $w$ to $[a',b']$.    
We also denote by
$\mathbb{W}_T$ the Wiener measure on $C^d_{[0,T]}$.
\item  For $a,b,c\in[0,T]$ such that $a\leq b\leq c$, $\hat{w}\in C^d_{[a,b]}$ and $\check{w}\in C^d_{[b,c]}$  we define the concatenation $\hat{w}\oplus \check{w} \in C^\ell_{[a,c]}$ by the formula
\begin{equation}\label{DefConc}
(\hat{w}\oplus \check{w})(s)=\left\{ \begin{array}{ll}
\hat{w}(s)& \text{ if } s\in [a,b],
\\
\hat{w}(b)+\check{w}(s)-\check{w}(b)& \text{ if } s\in [b,c].
\end{array}\right.
\end{equation}

\item  For any generic Polish space $E$ with a complete metric $d$,
we denote by $\Pc_2(E)$ the Wasserstein space of Borel probability measures $\rho$ on $E$ satisfying 
$\int_E d(x,0)^2\,\rho(\d x)<\infty$, 
where $0$ denotes an arbitrary fixed element in $E$ (the origin when $E$ is a vector space).  $\Pc_2(E)$ is
endowed with the $2$-Wasserstein distance $\Wc_2$ corresponding to the quadratic transport cost $(x,y)\mapsto d(x,y)^2$. As a measurable space, $\Pc_2(E)$ is
endowed with the corresponding Borel $\sigma$-algebra.  
\item Given a random variable $Y$ on a probability space $(\Omega,\Fc,\P)$, we denote by $\P_{Y}$ the law of $Y$ under $\P$. We shall also use the notation  $\Lc(Y)$ for the law of $Y$ (under $\P$) when there is no ambiguity.

\end{itemize}

%\red{We often integrate w.r.t $x$, i.e. $\mu^u(\d x)$, so maybe we should use the notation $\bolx$ $=$ $(x^u)_u$ instead of $x$ $=$ $(x^u)_u$}

\section{Controlled mean field non exchangeable system and the optimi\-zation problem
%the mean field control of non exchangeable systems
} \label{secMFC}

\subsection{Preliminaries}

Let $U$ be a Polish space   endowed with its   Borel $\sigma$-algebra $\Uc$, representing the continuum of heterogenous agents/particles, and $\lambda$ be  a positive  finite  measure on $U$. 
In particular, when $\lambda$ is a discrete measure, this models the case of a finite number of class of heterogenous interacting agents. 
%\red{Put a remark on the generality of $\lambda$, eventually a discrete measure}

We will often consider maps $u\mapsto \mu^u$ from $U$ to $\Pc_2(E)$. When we need 
to check measurability of such maps we will use the fact that the Borel $\sigma$-algebra   in $\Pc_2(E)$ coincides with the trace of the Borel $\sigma$-algebra corresponding to the weak topology, and that measurability holds if and only if the maps of the form 
\begin{align} \label{meascriterion}
    u \in U & \mapsto \;  \int_E\Phi(x)\,\mu^u(\d x) \in \R, 
\end{align}
are measurable for every choice of bounded continuous function $\Phi:E\to \R$.

The space $L^2(U,\Uc,\lambda;\Pc_2(E))$, denoted $L^2_\lambda(\Pc_2(E))$ for short, consists of  elements $\mu=(\mu^u)_{u\in U}$ that are  measurable functions $U\to \Pc_2(E)$, $u\mapsto \mu^u$, satisfying 
\begin{align} 
\int_U\int_{\R^d}d(x,0)^2\,\mu^u(\d x) \,\lambda(\d u) &= \; \int_U\Wc_2(\mu^u,\delta_0)^2\,\lambda(\d u)
<\infty,
\end{align} 
where $\delta_0$ denotes the Dirac mass at the fixed element $0$. 
$L^2_\lambda(\Pc_2(E))$ is  endowed with the (complete) metric
\begin{align} 
\bd (\mu,\nu) & = \; \left( \int_U\Wc_2(\mu^u,\nu^u)^2\,\lambda(\d u)\right)^{1/2},
\qquad \mu,\nu\in L^2_\lambda(\Pc_2(E)),
\end{align} 
and the corresponding Borel $\sigma$-algebra (when we deal with measurability issues).

%A basic role will be played by the space $L^2_\lambda(\Pc_2(\R^d))$, that we will simply denote   $L^2_\lambda(\Pc_2(\R^d))$. 
%\red{HP: I would prefer to keep $L^2_\lambda(\Pc_2(\R^d))$ because it is important to have in mind  that it is w.r.t. collection of measures, and the notation is not really shortened.OK}

We will also often deal with the case when $E$ is a space of continuous functions. For instance when $E=C_{[a,b]}^d$, for any  $\mu=(\mu^u)_u\in L^2_\lambda(\Pc_2(C_{[a,b]}^d))$ and $s\in[a,b]$, we may consider $\mu_s:=(\mu_s^u)_u$, where each $\mu^u_s$ is the image  of the measure $\mu^u$ under the coordinate mapping $C_{[a,b]}\to \R^d$, $w\mapsto w(s)$. It is then easy to see that $\mu_s^u\in\Pc_2(\R^d)$, $\mu_s\in L^2_\lambda(\Pc_2(\R^d))$,   $s\mapsto \mu_s$  is continuous $[a,b]\to L^2_\lambda(\Pc_2(\R^d))$, $u\mapsto \mu_s^u$ is measurable $U\to \Pc_2(\R^d)$,
$(u,s)\mapsto \mu_s^u$ is also measurable $U\times [a,b]\to \Pc_2(\R^d)$ and finally, for any other  
$\nu=(\nu^u)_u\in L^2_\lambda(\Pc_2(C_{[a,b]}^d))$,
$$
\Wc_2 (\mu^u_s,\nu^u_s)\le 
\Wc_2 (\mu^u,\nu^u),
\qquad
\bd (\mu_s,\nu_s)\le 
\bd (\mu ,\nu),
$$
for every $u\in U$, $s\in [a,b]$.

For a collection $\xi=(\xi^u)_{u\in U}$ of random variables defined on a probability space $(\Omega,\Fc,\P)$, we denote by $\P_{\xi^{\cdot}}$ the collection of the probability laws $(\P_{\xi^u})_{u\in U}$.%\red{PLutot: $\P_{\xi^{cdot}}$} 

We denote by $A$ the set of control actions and we assume that it is a Polish space. As written above, we also denote by $d$ and $\bd$ the metrics on $A$ and $L^2_\lambda(\Pc_2(A))$    %a complete metric $d_A$. 
%The corresponding distance in the space $L^2_\lambda(\Pc_2(A))$ is denoted by $\bd_A$ to stress the dependence on 
%$A$. \red{Is it really necessary to keep the subscript in $A$?}

\subsection{Coupled controlled mean field SDEs}

Let $(\Omega,\Fc, \P)$ be a complete probability space. For every $u\in U$ we are given an $\R^\ell$-valued standard Brownian motion $W^u=(W^u_t)_{t\in [0,T]}$ and an independent real random variable $Z^u$ having uniform distribution in $(0,1)$. We assume that $\{(W^u,Z^u)\,:\, u\in U\}$ is an independent family. For every $u\in U$
we denote by $(\Fc^{W^u}_t)_{t\in [0,T]}$ the natural filtration generated by $W^u$, by $\sigma(Z^u)$ the $\sigma$-algebra generated by $Z^u$ and by  
$\F^u=(\Fc^{u}_t)_{t\in [0,T]}$ the filtration given by 
$$
\Fc^{u}_t= \Fc^{W^u}_t\vee \sigma(Z^u) \vee \Nc,
\qquad t\in[0,T]\;,
$$
where $\Nc$ is the family of $\P$-null sets.

The coefficients of the control problem are functions $b,\sigma,f,g$ satisfying suitable assumptions detailed later; $b,\sigma,f$ are functions of $u\in U$, $x\in\R^d$, $\mu\in L^2_\lambda(\Pc_2(\R^d))$, $\nu\in L^2_\lambda(\Pc_2(A))$ 
 $a\in A$, with values respectively in $\R^d$, $\R^{d\times \ell}$, $\R$; $g$ is a real function of $u\in U$, $x\in\R^d$, $\mu\in L^2_\lambda(\Pc_2(\R^d))$.

The dynamics of the controlled system is described as follows. For every   starting time $t\in [0,T]$,  we solve a system of controlled stochastic It\^o differential equations, indexed by $u\in U$,  of the form 
%(\red{HP: replace $v$ by $v$?}) 
 \begin{equation}\label{stateeq}
     \left\{
\begin{array}{rcl}
    \d X^u_s & = & b\big(u,X^u_s, \alpha^u_s,  \P_{X^{\cdot}_s}, \P_{\alpha^{\cdot}_s}\big)\d s \\
     & & \quad + \;  \sigma\big(u,X^u_s,\alpha^u_s,  \P_{X^{\cdot}_s},\P_{\alpha^{\cdot}_s}\big) \d W^u_s, \quad  t \leq s \leq T,
    \\
    X_t^u & = & \xi^u,    \;\; u\in U.
\end{array}\right.
 \end{equation}
%where $\P_{X^{u}_s} $, $\P_{\alpha^{u}_s} $ denote the law of $X^{u}_s $, $\alpha^{u}_s $ under $\P$.
Here, the initial condition is given by a collection ${\xi}=(\xi^u)_u$ of $\R^d$-valued random variable such that $\xi^u$ is $\Fc_t^u$-measurable for each $u\in U$. From the definition of $\Fc_t^u$, $\xi^u$ is also independent from $(W^u_s-W^u_t)_{s\geq t}$. We will also require that the map %\red{definire $W_{[0,t]}$ dans la notation} 
\begin{equation}
    \label{measinIt}
u \; \longmapsto \;  \Lc\big(W^u_{[0,t]},Z^u,\xi^u\big),
\end{equation}
%(where $\Lc$ denotes the law under $\P$)
is Borel measurable as a mapping from $U$ to $\Pc_2(C^\ell_{[0,t]}\times (0,1)\times\R^d)$ and 
$$
\int_U\E[|\xi^u|^2]\lambda(\d u) <\infty.
$$
This way we have
$\P_{\xi^\cdot}\in L^2_\lambda(\Pc_2(\R^d))$ and even
\beqs
\left(\Lc\big(W^u_{[0,t]},Z^u,\xi^u\big)\right)_{u\in U} & \in & 
L^2_\lambda\Big(
\Pc_2(
C^\ell_{[0,t]}\times (0,1)\times\R^d
)\Big).
\enqs
When these conditions are met we say that $\xi$ is an admissible initial condition at time $t$ and we write  $\xi\in \Ic_t$.

We recall that, for every $u\in U$, assuming that the random variable  $\xi^u$ is $\Fc_t^u$-measurable implies that  it is $\P$-almost surely equal to a variable of the form $\underline{\xi}^u(W^u_{[0,t]},Z^u)$ for a measurable function $\underline{\xi}^u:
C^\ell_{[0,t]}\times (0,1)
\to\R^d$.  If 
$\underline{\xi}^u(w,z)$ is jointly measurable in $(u,w,z)$ then the measurability condition 
\eqref{measinIt} is satisfied. 
However in general 
$\underline{\xi}^u(w,z)$ is only measurable as a function of $(w,z)$ and so \eqref{measinIt} should be explicitly required in the definition of the set $\Ic_t$.

For every $u\in U$,  the control processes $(\alpha^u_t)_{t\in [0,T]}$ are defined as follows. 
For an arbitrary Borel measurable function
$$
\alpha: U\times [0,T]\times C_{[0,T]}^\ell\times (0,1)\to A
$$
we define 
$$
\alpha_t^u=\alpha(u,t, W^u_{\cdot\wedge t},Z^u),
\quad t\in [0,T],\,u\in U,
$$
where $W^u_{\cdot\wedge t}$ is the path $s\mapsto 
W^u_{s\wedge t}$, $s\in [0,T]$.
We note that each process $(\alpha_t^u)_t$ is $\F^u$-predictable.
We say that $\alpha$ is an 
 admissible control policy (or simply a policy) if
 \begin{equation}
\label{alphaadmissible}
 \int_U \int_0^T\E[d(\alpha_s^u,0)^2]\,\d s\,\lambda(\d u) \; < \; \infty.
 \end{equation}
 We denote by
 $\Ac$ the class of  all admissible policies $\alpha$.
We remark that, denoting $\mathbb{W}_T$ the Wiener measure on $C^\ell_{[0,T]}$, condition \eqref{alphaadmissible} is equivalent to
 $$
 \int_U \int_0^T\int_{C^\ell_{[0,T]}}\int_0^1\Big[d\big(
 \alpha(u,s, w(\cdot\wedge s),z),0\big)^2\Big]\,\d z \,\mathbb{W}_T(\d w)\,\d s\,\lambda(\d u) \; < \; \infty, 
 $$
 which shows that the class $\Ac$ does not depend on the choice of 
$\Omega$, $\Fc$, $\P$,  $W^u $ and   $Z^u$.

\begin{comment}
Finally, the cost functional to  minimize is
\begin{align*}
    J(t,\xi,\alpha) = \int_U\E\Big[ \int_t^T f(u, X^u_s,\alpha^u_s, \P_{X^{\cdot}_s}, 
    \P_{\alpha^{\cdot}_s})\,\d s + g(u, X^u_T, \P_{X^{\cdot}_T}) \Big]\,\lambda(\d u)
\end{align*}
and the value function is
$$
\upsilon(t,\xi)=\sup_{\alpha\in \Ac} J(t,\xi,\alpha),
\qquad t\in[0,T],\,\xi\in \Ic_t.
$$
Since the solution to the controlled equation \eqref{stateeq} also depends on $t,\alpha,\xi$, it will be denoted $X_s^{t,\xi,\alpha,u}$, which explains why   $J$ and $V$ depend on the indicated arguments.
\end{comment}

A few explanations are in order. We note that \eqref{stateeq} is a  system of stochastic differential equations, indexed by $u\in U$, which is coupled due to occurrence of the terms $(\P_{X^{u}_s})_{u\in U}$. We will give conditions on $b$ and $\sigma$ implying that each equation in \eqref{stateeq} has a unique $\F^u$-adapted continuous solution (up to indistinguishability), for $\lambda$-almost all $u\in U$.

In particular, $(X^u)_u$ will be an independent family of stochastic processes, because this holds for the family of Brownian motions $(W^u)_u$. As mentioned in the introduction, this raises an issue concerning the mea\-su\-ra\-bi\-li\-ty with respect to the parameter $u\in U$.  To overcome this issue, we shall work with the probability laws, and  we will show that the obtained solution to \eqref{stateeq} has the additional property that the map
$u\mapsto \Lc(X^{u }, W^u, Z^u)$ is Borel measurable on $C_{[t,T]}^d\times C_{[0,T]}^\ell\times (0,1)$. 
Under some conditions to be precised later on the running and terminal rewards, this will ensure that the gain  functional 
in \eqref{defJ} is well defined.

%$f$ and $g$, this will imply that the map
%\begin{align} 
%u &\mapsto \;   \E\Big[ \int_t^T f(u, X^u_s,\alpha^u_s, \P_{X^{\cdot}_s}, (\P_{\alpha^{v}_s})_{v})\,d s 
%+ g(u, X^u_T, \P_{X^{\cdot}_T}) \Big]
%\end{align}
%is Borel measurable and that the functional $J$ and the value function $V$ will be well defined.

A further comment concerns the introduction of the random variables $Z^u$ as an additional source of noise besides the Brownian motions $W^u$. Recall that each $Z^u$ has uniform distribution on $(0,1)$; we will use the well known property that any probability in $\R^d$ is the image of the uniform distribution under an appropriate Borel map $(0,1)\to\R^d$. For our purposes it is important that the initial conditions $\xi^u$ (for the state equation \eqref{stateeq} starting at time $t$) may have  an arbitrary element of $\Pc_2(\R^d)$ as its law. 
Our  requirement on  $\xi^u$ is that it should be square summable and $\Fc_t^u$-measurable, so we may define $\xi^u$ as an appropriate function of $Z^u$ and obtain the required distribution.  We also let $\xi^u$ depend on the trajectory of $W^u$ up to time $t$ to include initial conditions derived from the flow property of the considered processes. This will be helpful for establishing the dynamic programming principle.

Using the previous notation, it is convenient to recall that the index set $U$ is a Polish space with a  Borel   finite positive measure $\lambda$, 
the space of control actions $A$ is a Polish space, and $(\Omega,\Fc, \P)$ is a complete probability space, s.t. 
for every $u\in U$, $W^u$ is an $\R^\ell$-valued standard Brownian motion, $Z^u$ is a 
real random variable  with uniform distribution in $(0,1)$, independent of $W^u$, and   
$\{(W^u,Z^u): u\in U\}$ is an independent family. We next formulate the requirements that we need on the coefficients $b$, $\sigma$. 
 %The assumptions on $f$, $g$ will be introduced later.

\begin{Assumption}\label{assumptionbasic}$\,$
%\begin{enumerate}
%\item[(1)] The index set $U$ is a Polish space with a  Borel   finite positive measure $\lambda$. 
%\item[(2)] $(\Omega,\Fc, \P)$ is a complete probability space. 
%For every $u\in U$, $W^u$ is an $\R^\ell$-valued standard Brownian motion, $Z^u$ is a 
%real random variable  with uniform distribution in $(0,1)$, independent of $W^u$, and   
%$\{(W^u,Z^u): u\in U\}$ is an independent family.
%\item[(3)]The space of control actions $A$ is a Polish space. 
%\item[(1)] 
The functions
  \beqs
  b,\sigma : ~ U\times \R^d\times A\times L^2_\lambda(\Pc_2(\R^d))\times L^2_\lambda(\Pc_2(A)) & \longrightarrow &  \R^d,\;\R^{d\times \ell}
  \enqs
%  \qquad
%  \sigma:U\times \R^d\times A\times L^2_\lambda(\Pc_2(\R^d))\times L^2_\lambda(\Pc_2(A))\to 
%\R^{d\times \ell},
%$$
 are Borel measurable. There exist constants $L\ge0$, $M\ge 0$ such that
\begin{itemize} 
\item[(i)] $$ |b(u,x,a,\mu,\nu)-b(u,x',a,\mu',\nu)| \le L\left(|x-x'|+ \bd(\mu,\mu') \right)$$
 \item[(ii)] 
 $$
|\sigma(u,x,a,\mu,\nu)-\sigma(u,x',a,\mu',\nu)|
 \le L\left(|x-x'|+ \bd(\mu,\mu') \right),
 $$
\item[(iii)] 
$$
|b(u,x,a,\mu,\nu)|+
|\sigma(u,x,a,\mu,\nu)|
 \le M\big( 1+|x|+d(a,0)+ \bd(\mu,\delta_0)+\bd(\nu,\delta_0)\big)$$
\end{itemize} 
 for every  $u\in U$, $x,x'\in\R^d$, $\mu,\mu'\in L^2_\lambda(\Pc_2(\R^d))$, $\nu\in L^2_\lambda(\Pc_2(A))$, $a\in A$.
%\end{enumerate}
\end{Assumption}
We recall that $0$ also denotes a fixed element of $A$ and we note that in the previous expressions we write $\delta_0$ to denote the collection $(\delta_0)_{u\in U}$.

\begin{Remark} 
%{\em
  We do not explicitly consider time-depending coefficients $b,\sigma$ (and later $f$). How\-ever all our results have immediate extensions to this case, with almost identical proofs, provided $b,\sigma,f$ are required to be measurable in time (jointly with the other arguments) and the assumptions hold for constants that do not depend on time.
%}
\end{Remark}

%\color{blue}

\begin{Remark} 
Examples covered by Assumption \ref{assumptionbasic} 
 include the graphon interaction 
\eqref{X-graphon}
dis\-cus\-sed in the 
 introduction.
 Many other examples can be considered, for instance a drift (or a volatility) of the form
$b(\mu)=\int_U h(  \bar\mu^u)\,\lambda(du)$ for some Lipschitz function  $h:\R^d\to\R^{d}$, where $\bar\mu^u=\int_{\R^d}x\,\mu^u(dx)$. 
%\red{Marco: I hoped to treat the case  
% $b(\mu)=\int_U h(\sqrt{Var (\mu^u)})\,\lambda(du)$ for some Lipschitz function  $h$,
%but I am not able.}
\end{Remark}
%\normalcolor

\bigskip

We are now able to state and prove the basic existence and uniqueness result for the controlled state equation.
We first define a set of stochastic processes where a unique solution will be found.

\begin{Definition} \label{solutionstate}
Given $t\in [0,T]$, 
we say that a family $X=(X^u)_u$ of stochastic processes with values in $\R^d$ belongs to the space $\Sc_t$ if 
\begin{enumerate}
    \item 
 the map
$u\mapsto \Lc(X^{u}, W^u, Z^u)$
is Borel measurable from $U$ to $\Pc_2(C_{[t,T]}^d\times C_{[0,T]}^\ell\times (0,1))$;

\item 
each process $X^u$ is  continuous and $\F^u$-adapted;
\item 
the following norm is finite:
$$
\|X\|:=\left(\int_U \E \bigg[\sup_{s\in [t,T]} |X_s^u|^2\bigg] \,\lambda(\d u)\right)^{1/2}.
$$
\end{enumerate} 
We say that $(X^u)_u\in \Sc_t$ is a solution to 
\eqref{stateeq} 
if the equations in 
  \eqref{stateeq} are satisfied for $\lambda$-almost every $u$.
We say that the solution is unique if, whenever 
$(X^u)_u, (\tilde X^u)_u\in \Sc_t$   solve
  \eqref{stateeq}   then  the processes $X^u$ and $\tilde X^u$ coincide, up to a $\P$-null set,  for $\lambda$-almost all $u\in U$.
\end{Definition}

%\red{comment on the fact that $\Sc_t$ is not Banach.}    
\begin{Remark}  As defined above, the space $\Sc_t$ endowed with $\|\cdot\|$ is not a Banach space, and even a vector space. As a matter of fact, $\|X-Y\|$ is not well defined for any $X,Y\in\Sc_t$ as the joint law of $(X^u,Y^u)$ may not be a measurable function of $u\in U$. In particular, one cannot use a Picard iteration on $\Sc_t$ to construct solutions to \reff{stateeq}. To overcome this issue,  we shall work on the laws of processes which have the expected measurability property in $u\in U$.
\end{Remark}

\begin{Theorem} \label{exuniq} Suppose that Assumption \ref{assumptionbasic} holds.
Let $t\in [0,T]$ and $\xi\in \Ic_t$ be an admissible initial condition.
Let $\alpha\in\Ac$ be an admissible policy and define
$$
\alpha_t^u=\alpha(u,t, W^u_{\cdot\wedge t},Z^u),
\quad t\in [0,T],\,u\in U.
$$
Then there exists a unique solution $X=(X^u)_u\in \Sc_t$ to
the
 equation \eqref{stateeq}, in the sense of 
  Definition \ref{solutionstate}.
\end{Theorem}
\begin{proof}
We borrow some ideas from  Proposition 2.1 in \cite{bayetal23}, but we need very different arguments because of the lack of time continuity of the coefficients in the stochastic equations due to the occurrence of the control process. The proof is postponed to Appendix \ref{secappenthmexuuniq}.   
\end{proof}

We recall that   ${\xi}\in\Ic_t$ requires the random variable  $\xi^u$ to be $\Fc_t^u$-measurable  for every $u\in U$. Therefore  it is $\P$-almost surely equal to a variable of the form $\underline{\xi}^u(W^u_{[0,t]},Z^u)$ for a measurable function $\underline{\xi}^u:
C^\ell_{[0,t]}\times (0,1)
\to\R^d$. The state equation \eqref{stateeq} corresponding to a given admissible control $\alpha\in\Ac$ can be written 
\begin{equation}\label{stateeqwithxiunder}
     \left\{
\begin{array}{l}
    d X^{u}_s =   b\big(u,X^{u}_s,\alpha^u_s, \P_{X^{\cdot}_s},\P_{\alpha^{\cdot}_s}\big)\d s 
    + \sigma\big (u,X^{u}_s,\alpha^u_s, \P_{X^{\cdot}_s},\P_{\alpha^{\cdot}_s} \big)\d W^u_s,
    \;\; s\in [t,T],
    \\
X_t^{u}=\underline{\xi}^u(W^u_{[0,t]},Z^u),
    \\
    \alpha_s^u=\alpha(u,s, W^u_{\cdot\wedge s},Z^u)\;\; u\in U.
\end{array}\right.
 \end{equation}

\vspace{2mm}

% \textcolor{red}{Idris: If we do not talk about regularity of the value function, we can put the following result in the appendix and gather it with the 'measurable representation'}

We next present a result providing estimates on solution to those systems of SDEs. 

\begin{Proposition} \label{stimeeqstate} Suppose that Assumption \ref{assumptionbasic} holds.
Let $t\in [0,T]$ and $\xi\in \Ic_t$ be an admissible initial condition.
Let $\alpha\in\Ac$ be an admissible policy and define
$$
\alpha_t^u=\alpha(u,t, W^u_{\cdot\wedge t},Z^u),
\quad t\in [0,T],\,u\in U.
$$
Then the unique solution $X=(X^u)_u\in \Sc_t$ to
the
 equation \eqref{stateeq} satisfies the following: there exists a constant $C\ge0$, depending on $T$, $\lambda(U)$ and on the constants $L$, $M$ in 
Assumption \ref{assumptionbasic}, such that
\begin{align} \label{estimX}
\int_U \E \bigg[\sup_{s\in [t,T]} |X_s^u|^2\bigg] \,\lambda(\d u)
& \le \;  
C\left( 1 + \int_U \E[|\xi^u|^2]\,\lambda(\d u)
+\int_U \int_t^T\E[|\alpha_s^u|^2]\,ds\,\lambda(\d u)\right).
\end{align}
Finally, if $(X^u)_u$, 
$(\bar X^u)_u$ are solutions corresponding to $\xi,\bar\xi\in \Ic_t$  and we assume
 that
the map  
\begin{equation}\label{jointonxixibar}
u\mapsto \Lc\big(W^u_{[0,t]},Z^u,\xi^u,\bar\xi^u\big),
\end{equation}
is Borel measurable as a mapping from $U$ to $\Pc_2(C^\ell_{[0,t]}\times (0,1)\times\R^d\times\R^d)$ then
we have
\begin{equation}\label{lipwrtxi}
    \int_U \E \bigg[\sup_{s\in [t,T]} |X_s^u-\bar X_s^u|^2\bigg] \,\lambda(\d u)
\le 
C\,\int_U \E[|\xi^u-\bar \xi^u|^2]\,\lambda(\d u)
\end{equation}
for a constant $C$ that only depends on the Lipschitz constant $L$, on $T$ and on $\lambda(U)$.
\end{Proposition}
\begin{proof}
We write the proof for the case $b=0$. We only prove \eqref{lipwrtxi}, the other assertion being proved by similar arguments. 

We first note that the joint measurability condition \eqref{jointonxixibar} allows to apply Theorem \ref{exuniq} to the equation satified by the pair $(X^u,\bar X^u)$ and to conclude in particular that 
the map $u\mapsto \Lc\big(X^u,\bar X^u\big)$ is  
 Borel measurable as a mapping from $U$ to $\Pc_2(C^d_{[t,T]}\times C^d_{[t,T]})$.
Subtracting the equations for $X$ and $\bar X$, for some constant $C$ (possibly different from line to line) we have 
\begin{align}
    &\E\,\Big[ \sup_{r\in [t,s]}|X^u_r-\bar X_r^u|^2\Big]
\le C\,
    \E\,\Big[  |\xi^u-\bar \xi^u|^2\Big] 
    \\\quad &
    +
    C\, \E\, \int_t^s \Big|
\sigma\left(u,X^{u}_r,\alpha^u_r, \P_{X^{\cdot}_r}
    , \P_{\alpha^{\cdot}_r} \right)-
\sigma\left(u,\bar X^{u}_r,\alpha^u_r, 
     \P_{\bar X^{\cdot}_r}
    , \P_{\alpha^{\cdot}_r} \right)\Big|^2\,\d r
    \\\quad &
    \le C\, \E\,\Big[  |\xi^u-\bar \xi^u|^2\Big] +
    C\,  \int_t^s \Big\{\E\,[|
X^{u}_r- \bar X^{u}_r|^2 ]+\bd \Big( 
     \P_{X^{\cdot}_r},
     \P_{\bar X^{\cdot}_r}
    \Big)^2
    \Big\}\,\d r.
\end{align}
Since 
$$
\bd \Big(      \P_{X^{\cdot}_r},
     \P_{\bar X^{\cdot}_r}
\Big)^2\le\int_U
    \E\,   \Big[|
X^{u}_r- \bar X^{u}_r|^2 \Big]\,\lambda(\d u)
$$
integrating with respect to $\lambda(du)$ we obtain
\begin{align}&
\int_U
\E\,\Big[ \sup_{r\in [t,s]}|X^u_r-\bar X_r^u|^2\Big]\,\lambda(\d u)
\\&\quad 
\le C\,\int_U
    \E\,\Big[  |\xi^u-\bar \xi^u|^2\Big]\,\lambda(\d u) +
     C\,  \int_t^s \int_U \E\,[\sup_{q\in [t,r]}|
X^{u}_q- \bar X^{u}_q|^2 ]\,\lambda(\d u) \,\d r
\end{align}
and \eqref{lipwrtxi} follows from the Gronwall lemma.
\end{proof}

\begin{Remark} \label{rembsig} 
If the condition (iii) in Assumption \ref{assumptionbasic} is strengthened to 
\begin{align} \label{bsigfort}
%\hspace{-2cm} (iii') \hspace{3cm}  
|b(u,x,a,\mu,\nu)|+
|\sigma(u,x,a,\mu,\nu)| \; \le \;  M\big( 1+|x| + \bd(\mu,\delta_0) \big)
\end{align} 
for every  $u\in U$, $x\in\R^d$, $\mu\in L^2_\lambda(\Pc_2(\R^d))$, $\nu\in L^2_\lambda(\Pc_2(A))$, $a\in A$, 
then we get a stronger estimate  than \eqref{estimX}, namely
\begin{align} \label{estimX2}
\int_U \E \Big[\sup_{s\in [t,T]} |X_s^u|^2\Big] \,\lambda(\d u)
& \le \;  
C\Big( 1 + \int_U \E[|\xi^u|^2]\,\lambda(\d u) \Big).
\end{align}
\end{Remark}

\subsection{The control problem}

We make the following assumptions on the running and terminal cost functions.

\begin{Assumption}\label{assumptiononfgbis}$\,$
The functions
$$ 
f:U\times \R^d\times A\times L^2_\lambda(\Pc_2(\R^d))\times L^2_\lambda(\Pc_2(A))\to \R,
  \qquad
  g:U\times \R^d\times L^2_\lambda(\Pc_2(\R^d))\to \R
  $$
 are Borel measurable. 
 In addition, we assume that 
 %one of the two following conditions is satisfied (where $0$ denotes both the origin in $\R^d$ and a fixed element in $A$):
% \begin{enumerate}
%     \item[(1)] 
% The functions $f$,$g$ are bounded from below and there exists a constant   $M\ge 0$ such that
% $$
%|f(u,x,0,\mu,\delta_0)|+
% |g(u,x,\mu)|
% \le M\left(1+|x|^2+  \bd(\mu,\delta_0)^2\right),$$
% for every  $u\in U$, $x\in\R^d$, $\mu\in L^2_\lambda(\Pc_2(\R^d))$.
% \item[(2)] 
 Condition \eqref{bsigfort} holds and there exists a constant   $M\ge 0$ such that 
 %\red{HP: the upper bound should not depend on $a$}
 \begin{align}
%f(u,x,a,\mu,\nu)|+ |g(u,x,\mu)| &\le \;  M\left(1+|x|^2+d_A(a,0)^2+  \bd(\mu,\delta_0)^2+  %\bd(\nu,\delta_0)^2\right), \\
| f(u,x,a,\mu,\nu)|+ |g(u,x,\mu)| &\le \; { M\left(1+|x|^2+  \bd(\mu,\delta_0)^2 \right)}, 
\end{align} 
 for every  $u\in U$, $a\in A$, $x\in\R^d$, $\mu\in L^2_\lambda(\Pc_2(\R^d))$, $\nu\in L^2_\lambda(\Pc_2(A))$. 
 %(Here, $0$ denotes both the origin in $\R^d$ and a fixed element in $A$). 
 %\end{enumerate}
 \end{Assumption}

%\begin{Remark}
%Condition (2) in  Assumption \ref{assumptiononfgbis} assumes a linear growth condition on $b,\sigma$, and quadratic growth condition $f,g$, uniformly w.r.t. the control variables %$(a,\nu)$, while  Condition (1) includes the case of linear quadratic model. 
%\end{Remark}

\vspace{3mm}

From the measurability of the map $u$ $\mapsto$ $\Lc(X^u,W^u,Z^u)$ for the solution $X$ $=$ $X^{t,\xi,\alpha}$ 
$=$ $(X^{t,\xi,\alpha,u})_u$ $\in$ $\Sc_t$ to \eqref{stateeq}, given $t$ $\in$ $[0,T]$, $\xi$ $\in$ $\Ic_t$, $\alpha$ $\in$ $\Ac$,  and under Assumption \ref{assumptiononfgbis} on $f,g$, together with the square integrability conditions 
of $X$ in $\Sc_t$, we see that the map
\begin{align}
u &\mapsto \;   \E\Big[ \int_t^T f(u, X^u_s,\alpha^u_s, \P_{X^{\cdot}_s}, \P_{\alpha^{\cdot}_s})\,\d s 
+ g(u, X^u_T, \P_{X^{\cdot}_T}) \Big]
\end{align}
is Borel measurable, and we can then define the cost functional to be minimized 
\begin{align*}
    J(t,\xi,\alpha) = \int_U\E\Big[ \int_t^T f\big(u, X^u_s,\alpha^u_s, \P_{X^{\cdot}_s}, 
    \P_{\alpha^{\cdot}_s} \big)\,\d s + g\big(u, X^u_T, \P_{X^{\cdot}_T} \big) \Big]\,\lambda(\d u)
\end{align*}
and the associated  value function: 
\begin{align} \label{Vgrowth} 
V(t,\xi) &= \; \inf_{\alpha\in \Ac} J(t,\xi,\alpha),
\qquad t\in[0,T],\,\xi\in \Ic_t.
\end{align}
Moreover using Proposition \ref{stimeeqstate}, and under condition \eqref{bsigfort} (see Remark \ref{rembsig}), 
we have  
\begin{align} \label{growthV}
|V(t,\xi)| & \leq \;   C\Big( 1 + \int_U \E[|\xi^u|^2]\,\lambda(\d u) \Big), \quad  t \in [0,T], \; \xi \in \Ic_t.   
\end{align}

\begin{comment}

Next, we present a continuity result with respect to $\xi$ for the cost functional $J(t,\xi,\alpha)$ and the value function $\upsilon(t,\xi)$. 
\marginpar{\red{Do we need such continuity result as we shall deal later with discontinuous viscosity solution}}

\begin{Lemma}\label{Vcontinuous}
    Suppose that $\alpha\in \Ac$ and, for $t\in [0,T]$, $\xi$ and $\xi_n$ belong to $\Ic_t$. Suppose that
the maps  
$$
u\mapsto \Lc\big(W^u_{[0,t]},Z^u,\xi^u,\xi_n^u\big),
$$
are Borel measurable as   mappings from $U$ to $\Pc_2(C^\ell_{[0,t]}\times (0,1)\times\R^d\times\R^d)$ 
    and
    $$
    \int_U \E\,|\xi_n^u-\xi^u|^2\,\lambda(\d u)\to 0
    $$
    as $n\to\infty$. Then $J(t,\xi_n,\alpha)\to J(t,\xi,\alpha)$ and $\upsilon(t,\xi_n)\to \upsilon(t,\xi)$.
\end{Lemma}
\begin{proof} 
TO DO. MAYBE SOME CONTINUITY IS REQUIRED FOR $f,g$, POSSIBLY WE MAY TAKE ASSUMPTION \ref{assumptiononfg} BELOW.
\end{proof} 

\end{comment}

\section{Law invariance of the value function and DPP} \label{secDPP}

\subsection{Law invariance}

We show in this section the law invariance property of the value function, namely that 
$V(t,\xi)$ depends on $\xi$ only through its law. We impose additional assumption on the functions $f$, $g$.

\begin{Assumption}\label{assumptiononfg}
There exist  constants $K\ge 0$, $\gamma_i\in (0,1]$ ($i=1,2,3,4$) such that
\begin{align}&
|f(u,x,a,\mu,\nu)- f(u,x',a,\mu,\nu)|
\\&\quad
\le K \Big(|x-x'|^{\gamma_1}\,(1+ |x|+|x'|)^{2-\gamma_1}+ \bd(\mu,\mu) ^{\gamma_2}\,(1+ \bd(\mu,\delta_0)+\bd(\mu,\delta_0))^{2-\gamma_2}\Big),
\end{align}
\begin{align}&
 |g(u,x,\mu)-g(u,x',\mu)|
 \\&\quad
 \le K \Big(|x-x'|^{\gamma_3}\,(1+ |x|+|x'|)^{2-\gamma_3}+  \bd(\mu,\mu)^{\gamma_4}\,(1+ \bd(\mu,\delta_0)+\bd(\mu,\delta_0))^{2-\gamma_4}\Big),
 \end{align}
 for every  $u\in U$, $x,x'\in\R^d$, $\mu,\mu\in L^2_\lambda(\Pc_2(\R^d))$, $\nu\in L^2_\lambda(\Pc_2(A))$, $a\in A$.
\end{Assumption}

\begin{Remark}
The above assumption on $f,g$ are local H\"older dependence on $x,\mu$, uniformly  with respect to $a,\nu,u$. 
%and we notice that the LQ case, see Section [], is covered with   $\gamma_i=1$.     \red{Marco: now we have dropped the LQ case}
\end{Remark}
 
 \begin{Lemma}\label{LemIdLaw} Let Assumptions \ref{assumptionbasic}, \ref{assumptiononfgbis} hold.
Suppose that $\xi= (\xi^u)_u$, $\bar\xi = (\bar\xi^u)_u$ $\in$ $\Ic_t$  for some $t\in [0,T]$ are such that 
 \begin{equation}
     \label{cond-eq-law}
 \Lc(\xi^u,W^u_{[0,t]},Z^u )  =  \Lc(\bar\xi^u,W^u_{[0,t]},Z^u )
  \end{equation}
for $\lambda$-almost every $u\in U$. Then we have for any $\alpha\in \Ac$,
\beqs
J(t,\xi,\alpha) & = & J(t,\bar\xi,\alpha).
\enqs
%for any $\alpha\in \Ac$.
\end{Lemma}
\begin{proof} We first notice by Proposition \ref{EQUALAWSOLSDE} under condition \eqref{cond-eq-law} that 
%\red{Marco: by Proposition \ref{EQUALAWSOLSDE}} we have
 \beq\label{cond-eq-lawX}
 \Lc(X^u,W^u_{},Z^u ) & = & \Lc(\bar X^u,W^u_{},Z^u )
 \enq
 for $\lambda$-almost every $u\in U$, 
 where $(X^u)_u$ and $(\bar X^u)_u$ are the respective solutions to \eqref{stateeq}, with control $\alpha$, initial time $t$ and initial condition $\xi$ and  $\bar \xi$ respectively. In particular we get
 \beqs%\label{cond-eq-lawX}
 \Lc(X^u,\alpha^u_{} ) & = & \Lc(\bar X^u,\alpha^u_{} )
 \enqs
 for $\lambda$-almost every $u\in U$.
 As $J(t,\xi,\alpha)$ and $J(t,\bar \xi,\alpha)$ are expectations of  measurable functions of the $(X,\alpha)$ and $(\bar X,\alpha)$ respectively, we get the result.   
\end{proof}

\begin{Theorem}\label{lawinv}
   (law invariance). Let Assumptions \ref{assumptionbasic}, \ref{assumptiononfgbis}, \ref{assumptiononfg} hold 
   %. Suppose that 
   and fix $\xi= (\xi^u)_u$, $\bar\xi = (\bar\xi^u)_u$ $\in$ $\Ic_t$  for some $t\in [0,T]$.
%    and  that
%the map  
%\begin{equation}
%\label{jointonxixibarbis}
%u\mapsto \Lc\big(W^u_{[0,t]},Z^u,\xi^u,\bar\xi^u\big),
%\end{equation}
%is Borel measurable as a mapping from $U$ to $\Pc_2(C^\ell_{[0,t]}\times (0,1)\times\R^d\times\R^d)$.  
    If
    $\P_{\xi^u}=\P_{\bar\xi^u}$ for $\lambda$-almost every $u\in U$ then  $V(t,\xi)= V(t,\bar\xi)$.
\end{Theorem}

\begin{proof}
We write the proof in the case $b\equiv 0$, the general case being completely similar. We first notice that from Lemma \ref{lawinv} and 
Theorem \ref{respresentation THM}, we can assume w.l.o.g. that the map  
\begin{equation}
\label{jointonxixibarbis}
u\mapsto \Lc\big(W^u_{[0,t]},Z^u,\xi^u,\bar\xi^u\big) ~\mbox{ is Borel measurable }
\end{equation}
 as a mapping from $U$ to $\Pc_2(C^\ell_{[0,t]}\times (0,1)\times\R^d\times\R^d)$. 
As a matter of fact, we first use Proposition \ref{respresentation THM} which gives Borel maps $\tilde{\xi}, \tilde{\bar{\xi}}: U\times 
C^\ell_{[0,t]}\times (0,1)
\to\R^d$ 
such that 
\beqs
\Lc\big(\tilde{\xi}^u( W^u_{[0,t]},Z^u), W^u_{[0,t]},Z^u\big) & = & \Lc\big({\xi}^u, W^u_{[0,t]},Z^u\big)\;,\\
\Lc\big(\tilde{\bar{\xi}}^u( W^u_{[0,t]},Z^u), W^u_{[0,t]},Z^u\big) & = & \Lc\big(\bar{\xi}^u, W^u_{[0,t]},Z^u\big)
\enqs
for every $u\in U$. Then using Lemma \ref{LemIdLaw}, we get
\beqs
J(t,\xi,\alpha) & = & J(t,(\tilde{\xi}^u(W_{[0,t]}^u,Z^u))_u,\alpha)\\
J(t,\bar\xi,\alpha) & = & J(t,(\tilde{\bar\xi}^u(W_{[0,t]}^u,Z^u))_u,\alpha)
\enqs
for any $\alpha\in \Ac$ and 
\beqs
V(t,\xi,) & = & V(t,(\tilde{\xi}^u(W_{[0,t]}^u,Z^u))_u)\\
V(t,\bar\xi) & = & V(t,(\tilde{\bar\xi}^u(W_{[0,t]}^u,Z^u))_u).
\enqs
We can therefore replace $(\xi,\bar \xi)$ by $(\tilde{\xi}^u(W_{[0,t]}^u,Z^u))_u,\tilde{\bar{\xi}}^u(W_{[0,t]}^u,Z^u)_u)$ which satisfies \reff{jointonxixibarbis} as $\tilde \xi$ and $\tilde{\bar{\xi}}$ are Borel measurable. 

We consider, for fixed $t\in [0,T]$, $\xi\in\Ic_t$ and $\alpha\in\Ac$, the system: 
\begin{equation}%\label{stateeq}
     \left\{
\begin{array}{l}
    \d X^u_s =  \sigma\left(u,X^u_s,\alpha^u_s, \P_{X^{\cdot}_s},\P_{\alpha^{\cdot}_s}\right)\d W^u_s,
    \;\; s\in [t,T],
    \\
    X_t^u=\xi^u, \quad u \in U, 
    \\
    \alpha_s^u=\alpha(u,s, W^u_{\cdot\wedge s},Z^u), 
\end{array}\right.
 \end{equation}
 and the corresponding cost 
\begin{align*}
    J(t,\xi,\alpha) = \int_U\E\Big[ \int_t^T f(u, X^u_s,\alpha^u_s, \P_{X^{\cdot}_s}, \P_{\alpha^{\cdot}_s})\,\d s + g(u, X^u_T, \P_{X^{\cdot}_T}) \Big]\,\lambda(\d u).
\end{align*}
For every $\epsilon>0$ we will find another control policy $\alpha^\epsilon\in\Ac$ such that  $J(t,\alpha^\epsilon,\bar\xi)\to J(t,\alpha,\xi)$ as $\epsilon\to 0$. This way the required equality $V(t,\xi)= V(t,\bar\xi)$ will be proved.

We first look for a convenient expression for $J(t,\alpha,\xi)$. 
We recall that for every $u$ the random variable  $\xi^u$ is $\Fc_t^u$-measurable, so it is $\P$-almost surely equal to a variable of the form $\underline{\xi}^u(W^u_{[0,t]},Z^u)$ for a function $\underline{\xi}^u(w,z)\in\R^d$ defined for $w\in C^\ell_{[0,t]}$, $z\in (0,1)$ and measurable in $(w,z)$ (not necessarily in $u$). 
Recalling the notation in $\eqref{operatorbeta}$-$\eqref{operatorbeta2}$ we rewrite the controlled equation in the equivalent way: 
\begin{equation*}%\label{stateeq}
     \left\{
\begin{array}{rrl}
    \d X^u_s &=&\displaystyle  \sigma\big(u,X^u_s,\alpha^u_s, \P_{X^{\cdot}_s},\P_{\alpha^{\cdot}_s}\big) \d W^u_s,
    \;\; s\in [t,T],
    \\
X_t^u&=&\displaystyle
\underline{\xi}^u(W^u_{[0,t]},Z^u),
    \\
\alpha_s^u&=&\displaystyle
\tilde{\alpha}(u,s, W^u_{\cdot\wedge s},W^u_{[0,t]},Z^u).
\end{array}\right.
 \end{equation*}
For every $u$ the random variable  $\bar \xi^u$, being also  $\Fc_t^u$-measurable,  is $\P$-almost surely equal to $\underline{\bar\xi}^u(W^u_{[0,t]},Z^u)$ for a measurable function $\underline{\bar\xi}^u:C^\ell_{[0,t]}\times (0,1)
\to\R^d$. By our assumptions we have, for $\lambda$-almost every $u$, 
$$
\Lc\left(\underline{\xi}^u(W^u_{[0,t]},Z^u)\right)=
\P_{\xi^u}=\P_{\bar\xi^u}
=\Lc\left(\underline{\bar\xi}^u(W^u_{[0,t]},Z^u)\right)
$$ 
Since
$\Lc(
W^u_{[0,t]},Z^u)=\mathbb{W}_T\otimes m$, 
it follows that 
$\underline{\xi}^u(w,z)=\underline{\bar\xi}^u(w,z)$ for almost all $(w,z)$ with respect to 
$\mathbb{W}_T\otimes m$, which is a non-atomic measure on the Polish space $C^\ell_{[0,t]}\times (0,1)$. By a classical  result (see e.g. \cite[Lemma 5.23]{cardel18a})
%Carmona-Delarue, Vol I, Lemma 5.23 -\textcolor{red}{put citation}
for every $\epsilon>0$ there exists a Borel measurable map
$$
\tau^{\epsilon,u}: C^\ell_{[0,t]}\times (0,1)\to C^\ell_{[0,t]}\times (0,1)
$$
that preserves the measure $\mathbb{W}_T\otimes m$ and satisfies, for $\lambda$-almost all $u$,
\begin{equation}
    \label{epsilonnear}
|\underline{\xi}^u(\tau^{\epsilon,u}(w,z))-\underline{\bar\xi}^u(w,z)|\le\epsilon,
\qquad (w,z)\in C^\ell_{[0,t]}\times (0,1), \,\mathbb{W}_T\otimes m-a.s.
\end{equation}
Denote
$(w,z)\mapsto \tau_1^{\epsilon,u}(w,z)\in C^\ell_{[0,t]}$ and  
$(w,z)\mapsto \tau_2^{\epsilon,u}(w,z)\in (0,1)$ the coordinate maps of 
$\tau^{\epsilon,u}=(\tau_1^{\epsilon,u},\tau_2^{\epsilon,u})$.
 Using the independence of $(W^u_{[0,t]},Z^u)$ and $W^u_{[t,T]}-W^u_t$, and the measure-preserving property of $\tau^{\epsilon,u}$
it is easy to check that the pair
 $$
W^{\epsilon,u}:= \tau_1^{\epsilon,u}(W^u_{[0,t]},Z^u)\oplus W^u_{[t,T]}, \quad Z^{\epsilon,u}:=\tau_2^{\epsilon,u}(W^u_{[0,t]},Z^u)
 $$
consists of a Wiener process on $[0,T]$ and an independent random variable with uniform distribution in $(0,1)$. 
Then we consider the equation
\begin{equation}\label{lawinvX1}
     \left\{
\begin{array}{rrl}
    \d X^{\epsilon,u}_s &=&\displaystyle  \sigma\big(u,X^{\epsilon,u}_s,\alpha^{\epsilon,u}_s, \P_{X^{\eps,\cdot}_s},\P_{\alpha^{\eps,\cdot}_s}\big)\d W^{u}_s,
    \;\; s\in [t,T],
    \\
X_t^{\epsilon,u}&=&\displaystyle
\underline{\xi}^u(W^{\epsilon,u}_{[0,t]},Z^{\epsilon,u}),
    \\
\alpha_s^{\epsilon,u}&=&\displaystyle
\tilde{\alpha}(u,s, W^{u}_{\cdot\wedge s},W^{\epsilon,u}_{[0,t]},Z^{\epsilon,u}).
\end{array}\right.
 \end{equation}
Since the increments of $W^u$ and $W^{\epsilon,u}$ coincide on the interval $[t,T]$, in the above equations  $\d W^{u}_s$ might be replaced by $\d W^{\epsilon,u}_s$ and
$W^{u}_{\cdot\wedge s}$ by $W^{\epsilon,u}_{\cdot\wedge s}$.
Then we see that the process $X^{\epsilon,u}$ is 
 the trajectory corresponding to the  control policy $\alpha\in\Ac$, the initial condition
$\underline{\xi}^u(w,z)$
 and the driving noise $(W^{\epsilon,u},Z^{\epsilon,u})$. 
 From
 %Corollary \ref{uniqinlaw}
 Proposition \ref{EQUALAWSOLSDE},
 it follows that 
$
\Lc\left(
X^{\epsilon,u}, W^{\epsilon,u}_{[0,T]},  Z^{\epsilon,u}\right)
=
\Lc\left(
X^{u}, W^u_{[0,T]} ,Z^u\right)$
which implies
\beqs
\Lc\left(X^{\epsilon,u}_s,\alpha^{\epsilon,u}_s%,%(\P_{X^{v,\epsilon}_s})_{v}, (\P_{\alpha^{v,\epsilon}_s})_{v}
\right)
& = &
\Lc\left(X^{u }_s,\alpha^{u }_s%,(\P_{X^{v }_s})_{v}, (\P_{\alpha^{v }_s})_{v}
\right)
\enqs
for every $s\in [t,T]$ and $u\in U$
and finally
\begin{align*}
    J(t,\xi,\alpha) = \int_U\E\Big[ \int_t^T f\big(u, X^{\epsilon,u}_s,\alpha^{\epsilon,u}_s,\P_{X^{\epsilon,\cdot}_s}, \P_{\alpha^{\epsilon,\cdot}_s}\big)\,\d s + g\big(u, X^{\epsilon,u}_T, \P_{X^{\epsilon,\cdot}_T}\big) \Big]\,\lambda(\d u),
\end{align*}
which is the expression we were looking for.

Next we consider the equation
\begin{equation}\label{lawinvX2}
     \left\{
\begin{array}{rrl}
    \d \bar X^{\epsilon,u}_s &=&\displaystyle  \sigma\big(u,\bar X^{\epsilon,u}_s,\alpha^{\epsilon,u}_s,  \P_{\bar X^{\epsilon,\cdot}_s} ,\P_{\alpha^{\eps,\cdot}_s}\big)\d W^u_s,
    \;\; s\in [t,T],
    \\
\bar X_t^{\epsilon,u}&=&\displaystyle
\bar\xi^u=\underline{\bar \xi}^u(W^u_{[0,t]},Z^u),
    \\
\alpha_s^{\epsilon,u}&=&\displaystyle
\tilde{\alpha}(u,s, W^{\epsilon,u}_{\cdot\wedge s},W^{\epsilon,u}_{[0,t]},Z^{\epsilon,u})
\\&=&
\displaystyle
\tilde{\alpha}(u,s, W^u_{\cdot\wedge s},\tau_1^{\epsilon,u}(W^u_{[0,t]},Z^u),\tau_2^{\epsilon,u}(W^u_{[0,t]},Z^u)).
\end{array}\right.
 \end{equation}
  The process $X^{\epsilon,u}$ starts at $\bar \xi^u$ and the equation
 contains the same control processes $\alpha^{\epsilon,u}$ as in \eqref{lawinvX1}, but it is now driven by 
 the original  noise $(W^u,Z^u)$. We see that
 the process $X^{\epsilon,u}$ is 
 the trajectory  corresponding to the control policy
 $$
\alpha^\epsilon(u,s,w,z):= \alpha\Big(u,s,\tau_1^{\epsilon,u}(w_{[0,t]},z)\oplus 
w_{[t,T]},\tau_2^{\epsilon,u}(w_{[0,t]},z)\Big),
\qquad s\in [t,T],
 $$
 while we may take both $\alpha$ and $\alpha^\epsilon$ to be constant for $s\in [0,t)$, without loss of generality. To check that $\alpha^\epsilon$ is indeed admissible,
 using the fact that $ (W^{\epsilon,u} ,Z^{\epsilon,u})$ and
 $ (W^{u} ,Z^{u})$ have the same law, we verify that
 \begin{align}\label{alphaepsconst}
 \int_U \int_t^T\E[|\alpha^{\epsilon,u}_s|^2]\,\d s\,\lambda(\d u)
&=
 \int_U \int_t^T\E[| 
\tilde{\alpha}(u,s, W^{\epsilon,u}_{\cdot\wedge s},W^{\epsilon,u}_{[0,t]},Z^{\epsilon,u})|^2]
\,\d s\,\lambda(\d u)
\\ &
=
\int_U \int_t^T\E[| 
\tilde{\alpha}(u,s, W^{u }_{\cdot\wedge s},W^{ u}_{[0,t]},Z^{ u})|^2]
\,\d s\,\lambda(\d u)
\\  &
=
 \int_U \int_t^T\E[|\alpha^{u}_s|^2]\,ds\,\lambda(\d u),
 \end{align}
 which is finite and even independent of $\epsilon$.
 The corresponding cost is then 
\begin{align*}
J(t,\bar\xi,\alpha^\epsilon) = \int_U\E\Big[ \int_t^T f\big(u, \bar X^{\epsilon,u}_s,\alpha^{\epsilon,u}_s,\P_{\bar X^{\epsilon,\cdot}_s}, \P_{\alpha^{\epsilon,\cdot}_s}\big)\,\d s + g\big(u, \bar X^{\epsilon,u}_T, \P_{\bar X^{\epsilon,\cdot}_T}\big) \Big]\,\lambda(\d u).
\end{align*}
To conclude the proof it remains to prove that 
$J(t,\bar\xi,\alpha^\epsilon)\to J(t,\xi,\alpha)$ as $\epsilon\to 0$. Comparing
  \eqref{lawinvX1} and
  \eqref{lawinvX2} and applying estimate \eqref{lipwrtxi} in Proposition \ref{stimeeqstate}  we see that there exists a constant $C\ge 0$, depending only on $T$ and the Lipschitz constants of $b,\sigma$ such that
$$
\int_U\E\Big[\sup_{s\in [t,T]}|X^{\epsilon,u}_s-\bar X^{\epsilon,u}_s|^2\Big]\,\lambda(\d u) \le C 
\int_U\E\Big[|\underline{\xi}^u(W^{\epsilon,u}_{[0,t]},Z^{\epsilon,u})-\bar \xi^{u}|^2\Big]\,
\lambda(\d u) \;.
$$  
  But we have
  $$
\Big|\underline{\xi}^u(W^{\epsilon,u}_{[0,t]},Z^{\epsilon,u})-\bar \xi^{u}\Big|
=\Big|\underline{\xi}^u(
\tau^{\epsilon,u}(W^u_{[0,t]},Z^u))-
\underline{\bar \xi}^u(W^u_{[0,t]},Z^u)\Big|\le \epsilon
$$
which follows from 
\eqref{epsilonnear}
and the fact that $(W^{u}_{[0,t]},Z^{u})$ has law $\,\mathbb{W}_T\otimes m$. So we obtain
\begin{equation}
    \label{unifxebar}
\int_U\E\Big[\sup_{s\in [t,T]}|X^{\epsilon,u}_s-\bar X^{\epsilon,u}_s|^2\Big]\,\lambda(\d u) \le C  \epsilon^2\,\lambda(U).
\end{equation}  
Still using Proposition \ref{stimeeqstate}, we also note that
\begin{align}    \label{unifinepsxebarx}
&\int_U\E\Big[\sup_{s\in [t,T]}|X^{\epsilon,u}_s |^2\Big]\,\lambda(\d u) 
+ \int_U\E\Big[\sup_{s\in [t,T]}|\bar X^{\epsilon,u}_s |^2\Big]\,\lambda(\d u) 
\\&\quad
\le C  
\left( 
 \int_U  \E[| 
\xi^{u }|^2]
\,ds\,\lambda(\d u)+
 \int_U  \E[| 
\xi^{u }|^2]
\,ds\,\lambda(\d u)
+\int_U \int_t^T\E[|\alpha^{\epsilon,u}_s|^2]\,ds\,\lambda(\d u)
\right)
\\&\quad \le C ,
\end{align}
for some constant $C$ which is not dependent on $\epsilon$, by \eqref{alphaepsconst}. Using our assumptions on $f$, $g$ we have
\begin{align}
& \; |J(t,\xi,\alpha)-    J(t,\bar\xi,\alpha^\epsilon)| \\
\le & \quad  C
\int_U \int_t^T \E\bigg[
  |X^{\epsilon,u}_s-\bar X^{\epsilon,u}_s|^{\gamma_1}\,(1+ |X^{\epsilon,u}_s|+|\bar X^{\epsilon,u}_s|)^{2-\gamma_1}\bigg] \,\d s\,\lambda(\d u) \\
 &\qquad    + \;  C
\int_U \int_t^T\Big( \bd(\P_{ X^{\epsilon,\cdot}_s},\P_{\bar X^{\epsilon,\cdot}_s}) ^{\gamma_2}\,(1+ \bd(\P_{ X^{\epsilon,\cdot}_s},\delta_0)+\bd(\P_{\bar X^{\epsilon,\cdot}_s},\delta_0))^{2-\gamma_2}\Big)
    \,\d s\,\lambda(\d u) \\ 
 & \qquad  + \; C\int_U\E\bigg[
   | X^{\epsilon,u}_T- \bar X^{\epsilon,u}_T|^{\gamma_3}\,(1+ |X^{\epsilon,u}_T|+|\bar X^{\epsilon,u}_T|)^{2-\gamma_3}\bigg]\,\lambda(\d u)\\
&\qquad
   + \; C\int_U \Big(
     \bd(\P_{ X^{\epsilon,\cdot}_T},\P_{\bar X^{\epsilon,\cdot}_T})^{\gamma_4}\,(1+ \bd(\P_{ X^{\epsilon,\cdot}_T},\delta_0)+\bd(\P_{\bar X^{\epsilon,\cdot}_T},\delta_0))^{2-\gamma_4}\Big)
\,\lambda(\d u).
\end{align}
The H\"older inequality, with conjugate exponents $2/\gamma_1$ and $2/(2-\gamma_1)$, gives
\begin{align}&
\int_U \int_t^T \E\bigg[
  |X^{\epsilon,u}_s-\bar X^{\epsilon,u}_s|^{\gamma_1}\,(1+ |X^{\epsilon,u}_s|+|\bar X^{\epsilon,u}_s|)^{2-\gamma_1}\bigg] \,\d s\,\lambda(\d u)
\\&\quad
\le 
\int_U \int_t^T \left\{\Big(\E[
  |X^{\epsilon,u}_s-\bar X^{\epsilon,u}_s|^2]
  \Big)^{\frac{\gamma_1}{2}}\,\Big(
 \E[ (1+ |X^{\epsilon,u}_s|+|\bar X^{\epsilon,u}_s|)^2]
 \Big)^{\frac{2-\gamma_1}{2}}
  \right\}\,\d s\,\lambda(\d u)
\\&\quad
\le
T\int_U   \left\{\Big(\E[\sup_{s\in [t,T]}
  |X^{\epsilon,u}_s-\bar X^{\epsilon,u}_s|^2]
  \Big)^{\frac{\gamma_1}{2}}\,\Big(
 \E[ \sup_{s\in [t,T]}(1+ |X^{\epsilon,u}_s|+|\bar X^{\epsilon,u}_s|)^2]
 \Big)^{\frac{2-\gamma_1}{2}}
  \right\} \,\lambda(\d u)
\\&\quad
\le
T\left\{\int_U    \E[\sup_{s\in [t,T]}
  |X^{\epsilon,u}_s-\bar X^{\epsilon,u}_s|^2]
   \,\lambda(\d u)\right\}^{\frac{\gamma_1}{2}}
  \left\{\int_U  
 \E[ \sup_{s\in [t,T]}(1+ |X^{\epsilon,u}_s|+|\bar X^{\epsilon,u}_s|)^2]
 \,\lambda(\d u) 
  \right\} ^{\frac{2-\gamma_1}{2}}
  \\&\quad
\le
C\left\{ \epsilon^2\,\lambda(U)\right\}^{\gamma_1/2}
  \end{align}
 by \eqref{unifxebar}
and \eqref{unifinepsxebarx}. 
Using the inequality 
$$
\bd(\P_{ X^{\epsilon,\cdot}_s},\P_{\bar X^{\epsilon,\cdot}_s})^2
=
 \int_U \Wc_2(
\P_{ X^{\epsilon,u}_s},\P_{ \bar X^{\epsilon,u}_s})^2\,\lambda(\d u)
\le \int_U \E[
  |X^{\epsilon,u}_s-\bar X^{\epsilon,u}_s|^2]\,\lambda(\d u)
$$
similar passages and  the H\"older inequality allow to treat the other terms and conclude that
$J(t,\xi,\alpha)-    J(t,\bar\xi,\alpha^\epsilon)\to 0$.
\end{proof}

\vspace{3mm}

In view of Theorem \ref{lawinv}  it is possible to define a function $\upsilon:[0,T]\times L^2_\lambda(\Pc_2(\R^d))\to \R$ as follows. For any  $t\in [0,T]$ and $\mu\in L^2_\lambda(\Pc_2(\R^d))$, let us choose $\xi=(\xi^u)_u\in\Ic_t$ such that $\P_{\xi^u}=\mu^u$ for $\lambda$-almost every $u\in U$, and let us set
\begin{align} \label{defupsilon} 
\upsilon(t,\mu) &= \;  V(t,\xi).
\end{align} 
\begin{comment} 
$\upsilon$ is well defined, since if  
$\tilde\xi$ satisfies the same assumptions as $\xi$ (with a different $\underline{\tilde\xi}^u(w,z)$) then 
\eqref{jointonxixibarbis}  holds and theorem 
implies that $\upsilon(t,\xi)=\upsilon(t,\tilde\xi)$.
The validity of 
\eqref{jointonxixibarbis} can be checked taking a bounded measurable
$\Phi: C^\ell_{[0,t]}\times  (0,1)\times \R^d\times \R^d\to\R$
and noting that
\begin{align}
& \E\bigg[\Phi\left(
W^u_{[0,t]},Z^u,\xi^u,\bar\xi^u\right)\bigg] \\
= & 
\int_{
C^\ell_{[0,t]}\times  (0,1)
}
\Phi\left(
w,z,\underline{\xi}^u(w,z),
\underline{\tilde\xi}^u(w,z)\right)
\Lc (
W^u_{[0,t]},Z^u)(\d w\,\d z) \\
= & 
\int_{
C^\ell_{[0,t]}\times  (0,1)
}
\Phi\left(
w,z,\underline{\xi}^u(w,z),
\underline{\tilde\xi}^u(w,z)\right)
(\mathbb{W}_T\otimes m)(\d w\,\d z)
\end{align}   
is a measurable function of $u$. 

%\textcolor{red}{Idris: discussion related to the representation thm} 
\end{comment} 

Finally we show how a required variable $\xi$ can be constructed, given arbitrary $t$ and $\mu$. We note that the map $(u,A)\mapsto \mu^u(A)$ defined for $u\in U$ and any Borel set of $\R^d$ is a transition kernel. By a known extension of the Skorohod construction, there exists a (jointly) measurable function $j:U\times (0,1)\to\R^d$ such that, for every $u$, the image of the Lebesgue measure on $(0,1)$ under the map $z\mapsto j(u,z)$  equals $\mu^u$; therefore we may define $\xi^u=j(u,Z^u)$. This way we obtain the required function $\underline{\xi}^u(w,z)=j(u,z)$ (that does not depend on $w$).

% \red{Again, we may remove this corollary}

%\begin{Corollary}
%For every $t\in [0,T]$  the function  $\mu\mapsto \upsilon(t,\mu)$ is continuous on $L^2_\lambda(\Pc_2(\R^d))$.
%\end{Corollary}

%TO DO: use lemma \ref{Vcontinuous}.

%Possibly insert a remark saying that under other assumptions (e.g. $f$ Lipschitz in $x$, or $f$ uniformly continuous in $x$) one has $\mu\mapsto \upsilon(t,\mu)$ uniformly continuous 

%find conditions on $f$, $g$ implying that  $\mu\mapsto \upsilon(t,\mu)$ is Lipschitz

%prove the growth estimate
%$$
%|\upsilon(t,\mu)|\le C\left(1+
%\int_U \Wc_2(\mu^u,\delta_0)^2\, \lambda(du)\right)
%= C\left(1+ \bd (\mu,\delta_0)^2\right)
%$$
%starting from the growth estimates on %$\upsilon(t,\xi)$

\subsection{Dynamic programming principle}

%THIS SUBSECTION IS ONLY A SKETCH

We note that for the solution $X_s^{t,\xi,\alpha,u}$ to equation \eqref{stateeq} we have the following flow property: for $0\le t\le \theta\le T$, $\alpha\in\Ac$, and $\xi\in\Ic_t$ we have, for every $\lambda-$a.e. $u\in U$, $\P$-a.s.
\begin{align} \label{FlowPPTY}
X_s^{t,\xi,\alpha,u} & = \; 
X_s^{\theta,
X_\theta^{t,\xi,\alpha}
,\alpha,u},
\qquad s\in [\theta,T].
\end{align} 
This follows immediately from the uniqueness statement in Theorem \ref{exuniq}, since both terms are solution to the state equation \eqref{stateeq} in the interval $[\theta,T]$. As a consequence, we can state the dynamic programming principle for the value function $V$.

\begin{Theorem} \label{DPP}
    For $t\in [0,T]$  and $\xi\in\Ic_t$,  we have
    \begin{eqnarray*}
V(t,\xi) & = & \inf_{\alpha\in \Ac} \left\{ \int_U
\E\Big[
\int_t^\theta
f\big(u,X_s^{t,\xi,\alpha,u},
{ \P_{X_s^{t,\xi,\alpha,\cdot}}  }, \alpha_s^u,  \P_{\alpha_s^{\cdot}} 
\big)\,\d s\Big]\,\lambda(\d u)\right.\\
 & & \left.\qquad\qquad + \;  V\left(\theta,
(X_\theta^{t,\xi,\alpha,u})_{u}\right) \right\}, 
\end{eqnarray*}
for any $\theta\in[t,T]$.
\end{Theorem}
\begin{proof} 
Fix $0\le t\le \theta\le T$ and $\xi\in\Ic_t$. For $\alpha\in \Ac$ we have from \eqref{FlowPPTY} and the definitions of the cost functional $J$, and the value function $V$
\beqs
J(t,\xi,\alpha) & = &  \int_U\E\Big[ \int_t^\theta f(u, X^{t,\xi,\alpha,u}_s,\alpha^u_s, \P_{X^{t,\xi,\alpha,\cdot}_s}, \P_{\alpha^{\cdot}_s})\,\d s  \Big]\,\lambda(\d u)\\
 & & +\int_U\E\Big[ \int_\theta^T f\big(u, X^{u,\theta,
X_\theta^{t,\xi,\alpha}
,\alpha}_s,\alpha^u_s, \P_{X^{\theta,
X_\theta^{t,\xi,\alpha},\alpha,\cdot}_s}, \P_{\alpha^{\cdot}_s}\big)\,\d s  \\
 & & + \; g\big(u, X^{\theta, X_\theta^{t,\xi,\alpha} ,\alpha,u}_T,\P_{X^{\theta,
X_\theta^{t,\xi,\alpha},\alpha, \cdot}_T}\big) \Big]\,\lambda(\d u)\\
 & \geq & \int_U\E\left[ \int_t^\theta f\big(u, X^{t,\xi,\alpha,u}_s,\alpha^u_s, \P_{X^{t,\xi,\alpha,\cdot}_s}, \P_{\alpha^{\cdot}_s}\big)\,\d s  \right]\,\lambda(\d u)+V\left(\theta, ({X^{t,\xi,\alpha,u}_s})_{u}\right)\;.
\enqs
Since the previous inequality holds for any $\alpha\in\Ac$, we get
\beqs
V(t,\xi) & \geq & \inf_{\alpha\in \Ac} \left\{ \int_U
\E\left[
\int_t^\theta
f\big(u,X_s^{t,\xi,\alpha,u},\alpha_s^u,
 \P_{X_s^{t,\xi,\alpha,\cdot}} ,
\P_{\alpha^{\cdot}_s}\big)\,\d s\right]\,\lambda(\d u) +  V\left(\theta,
(X_\theta^{t,\xi,\alpha,u})_{u}\right)
 \right\}.
\enqs
We turn to the reverse inequality. Fix $\alpha\in\Ac$ and $\eps>0$. From the definition of the value function $V$, there exists some $\alpha^\eps=(\alpha^{\eps,u})_u\in\Ac$ such that
\begin{align} \label{IntContEps}
J\big(\theta,
(X_\theta^{t,\xi,\alpha,u})_{u},\alpha^\eps\big) & \leq \;  V\big(\theta,
(X_\theta^{t,\xi,\alpha,u})_u\big) + \eps
\end{align} 
We next define the control $\bar \alpha ^\eps=(\bar \alpha^{\eps,u})_u$ by 
\beqs
\bar \alpha^{\eps,u}_s & = &  \alpha^{u}_s\mathds{1}_{[t,\theta)}(s)+\alpha^{\eps,u}_s \mathds{1}_{[\theta,T]}(s)\;,\quad s\in[0,T]\;,
\enqs
for $u\in U$ and $s\in[t,T]$. We obviously   have $\bar \alpha ^\eps\in\Ac$. Moreover, using the flow property \reff{FlowPPTY}, we have
\beqs
J(t,\xi,\bar \alpha ^\eps) & = &  \int_U\E\Big[ \int_t^\theta f(u, X^{t,\xi,\alpha,u}_s,\alpha^u_s, \P_{X^{t,\xi,\alpha,\cdot}_s}, \P_{\alpha^{\cdot}_s})\,\d s  \Big]\,\lambda(\d u)
 +  J\big(\theta,(X_\theta^{t,\xi,\alpha,u})_u,\alpha^\eps\big).
\enqs
From \eqref{IntContEps}, we get
\beqs
J(t,\xi,\bar \alpha ^\eps) & \leq &  \int_U\E\Big[ \int_t^\theta f(u, X^{t,\xi,\alpha,u}_s,\alpha^u_s, \P_{X^{t,\xi,\alpha,\cdot}_s}, \P_{\alpha^{\cdot}_s})\,\d s  \Big]\,\lambda(\d u)\\
 & & + \;  V\left(\theta,
(X_\theta^{t,\xi,\alpha,u})_u\right) +\eps, 
\enqs
and so 
\beqs
V(t,\xi) & \leq &  \int_U\E\Big[ \int_t^\theta f\big(u, X^{t,\xi,\alpha,u}_s,\alpha^u_s, \P_{X^{t,\xi,\alpha,\cdot}_s}, \P_{\alpha^{\cdot}_s}\big)\,\d s  \Big]\,\lambda(\d u) 
+ V\big(\theta,(X_\theta^{t,\xi,\alpha,u})_u\big) + \eps. 
\enqs
Taking the infimum over $\alpha\in\Ac$ and sending $\eps$ to 0, we get the result.
\end{proof}

\vspace{3mm}

The following corollary is an immediate consequence of this result and Theorem \ref{lawinv}.

\begin{Corollary}\label{CorDPPv}
    For $t\in[0,T]$ and $\mu\in L^2_\lambda(\Pc_2(\R^d))$ we have
    \begin{eqnarray*}
\upsilon(t,\mu) & = & \inf_{\alpha\in \Ac} \Big\{ \int_U
\E\Big[
\int_t^\theta
f\big(u,X_s^{t,\xi,\alpha,u},\alpha_s^u,
 \P_{X_s^{t,\xi,\alpha,\cdot}} ,  \P_{\alpha_s^{\cdot}} 
\big)\,\d s\Big]\,\lambda(\d u) \\
 & & \hspace{3cm}   + \;  \upsilon\big(\theta,\P_{X_\theta^{t,\xi,\alpha,\cdot}}\big) \Big\}
    \end{eqnarray*}
    for any $\theta\in[t,T]$,  
  $\xi\in\Ic_t$ satisfying $\P_{\xi^u}=\mu^u$ for $\lambda$-almost every $u\in U$.
\end{Corollary}
\begin{proof} 
This follows from the definition of the function $\upsilon$ in \eqref{defupsilon}.  
%\red{The invariance is valid for $\xi$ $\in$ $\Ic_t$ + measurability condition.}
\end{proof}

%\red{again, we can remove this corollary} 

%We also have the following result.

%\begin{Corollary}
%The function $t\mapsto \upsilon(t,\mu)$ is left-continuous for every $\mu\in L^2_\lambda(\Pc_2(\R^d))$.
%\end{Corollary}
% \proofname. Let $t\in[0,T]$ and $\mu\in L^2_\lambda(\Pc_2(\R^d))$. Using Corollary \ref{CorDPPv}, we have for $h>0$ 
%\beqs
%\upsilon(t+h,\mu)-\upsilon(t,\mu) & = & \left| \sup_{\alpha\in \Ac} \left\{ \int_U
%\E\left[
%\int_t^\theta
%f\left(u,X_s^{t,\xi,\alpha,u},
%( \P_{X_s^{v,t,\xi,\alpha}}  )_{v},\alpha_s^u
%\right)\,ds\right]\,\lambda(du) 
% \right. \right.\\
%  & & +\left.\left.\upsilon(t+h,\mu)- v\left(t+h,
%(\P_{X_{t+h}^{v,t,\xi,\alpha}})_{v}\right)\right\}\right|\\
%  & \leq & \sup_{\alpha\in \Ac}\left| \int_U
%\E\left[
%\int_t^\theta
%f\left(u,X_s^{t,\xi,\alpha,u},
%( \P_{X_s^{v,t,\xi,\alpha}}  )_{v},\alpha_s^u
%\right)\,ds\right]\,\lambda(du) 
% \right| 
%\enqs
 
% NEED ADITIONAL ASSUMPTIONS
% \qedsymbol
 
%prove that $t \mapsto \upsilon(t,\mu)$ is continuous (starting from DPP and estimating $\upsilon(t,\mu)-v(\theta,\mu)$) 

%\section{The Ito formula and the Bellman equation}

\section{It\^o formula} \label{secIto} 

\subsection{Derivatives on square integrable measure maps}

In our context, the It\^o formula describes the time derivative of the composition of a real function $v(\mu)$ of $\mu\in L^2_\lambda(\Pc_2(\R^d))$ and a map $s\mapsto \mu_s$ corresponding to the law of a family of stochastic processes $(X^u)_u$, namely $\mu_s^u=\P_{X^u_s}$.
It holds under regularity assumptions on $v$ - that may also depend explicitly on time - and requires the definition of derivatives for functions $v:L^2_\lambda(\Pc_2(\R^d))\to\R$ that we are now going to introduce. Given $\mu$ $\in$ $L_\lambda^2(\Pc_2(\R^d))$, and a measurable function $(u,x)$ $\in$ $U\times\R^d$ $\mapsto$  $\varphi(u,x)$, with quadratic growth in $x$, uniformly in $u$, we define the duality product:
\begin{align}
<\varphi,\mu> & := \; \int_U \int_{\R^d} \varphi(u,x) \mu^u(\d x) \lambda(\d u).     
\end{align}

\begin{Definition}
\label{Defderivative} Given a  function
$v:L^2_\lambda(\Pc_2(\R^d))\to \R$, we say that a measurable function 
%$\dmu v:U\times \R^d\times L^2_\lambda(\Pc_2(\R^d)) \to\R $ is the derivative of $v$
\begin{align}
\dmu v:  L_\lambda^2(\P^2(\R^d)) \times U \times \R^d \; \ni \; (\mu,u,x)   \; \longmapsto \; \dmu v(\mu)(u,x)     
\end{align}
is the linear functional derivative of $v$ if
\begin{enumerate}
    \item for every compact set $K\subset L^2_\lambda(\Pc_2(\R^d))$ there exists a constant $C_K>0$ such that
    $$\left|
    \dmu v(\mu)(u,x) \right|\le C_K\, (1+|x|^2),  $$
    for every $u\in U$, $x\in \R^d$, $\mu\in K$;
    \item for every $\mu,\nu\in L^2_\lambda(\Pc_2(\R^d))$ we have
\begin{align}
v(\nu)-v(\mu) & = \int_0^1 < \dmu v(\mu + \theta(\nu-\mu))(.),\nu - \mu> \d \theta \\
 & = \;  \int_0^1 \int_U  \int_{\R^d} \dmu v(\mu + \theta(\nu-\mu)(u,x) \,(\nu^u-\mu^u)(\d x) 
 \lambda(\d u) \d \theta.
\end{align}
\end{enumerate}
\end{Definition}

\begin{Remark}
Under the additional condition that $\mu$ $\mapsto$ $\dmu v(\mu)(u,x)$ is continuous, the above definition  is equivalent to the existence of the Gateaux-derivative 
\begin{align}
\lim_{\eps \rightarrow 0} \frac{v(\mu + \eps(\nu-\mu) - v(\mu)}{\eps}  & = \;  < \dmu v(\mu)(.),\nu - \mu> \\
& = \; \int_U \int_{\R^d} \dmu v(\mu)(u,x) (\nu^u-\mu^u)(\d x) \lambda(\d u).   
\end{align}
\end{Remark}

The following definition encompasses the minimal assumptions on $v$ that are required to obtain the It\^o formula.

\begin{Definition}
We say that $v:[0,T]\times L^2_\lambda(\Pc_2(\R^d))\to \R$ is of class $\tilde C^{1,2}([0,T]\times L^2_\lambda(\Pc_2(\R^d))$ if \begin{enumerate}
    \item for every $\mu\in L^2_\lambda(\Pc_2(\R^d))$ the function 
$t\mapsto \upsilon(t,\mu)$ is continuously differentiable on $[0,T]$; we denote $\partial_t \upsilon(t,\mu)$ its time derivative;
\item 
for every $t\in[0,T]$
the derivative $\dmu \upsilon(t,\mu)(u,x)$ exists and it is measurable in all its arguments;
\item
$\dmu \upsilon(t,\mu)(u,x)$ is twice continuously differentiable on $\R^d$ as a function of $x$
and the gradient and the Hessian matrix
$$
\partial_x\dmu v  : [0,T]\times L^2_\lambda(\Pc_2(\R^d)) \times U\times \R^d \to\R^d,
\qquad
\partial_x^2\dmu v  
: [0,T]\times L^2_\lambda(\Pc_2(\R^d)) \times U\times \R^d \to\R^{d\times d}
$$
satisfy the following growth conditions: there exists a constant $C\ge 0$
such that
\begin{align}\label{growthpartialv}
\displaystyle
\left|\partial_x\dmu \upsilon(t,\mu)(u,x) \right| &\le\displaystyle C\left(1+|x| +  \bd(\mu,\delta_0)\right),
\\\label{growthpartial2v}
\displaystyle\left|\partial_x^2\dmu \upsilon(t,\mu)(u,x)\right| &\le C,
\end{align}
for every $t$ $\in$ $[0,T]$, $u\in U$, $x\in\R^d$, $\mu\in L^2_\lambda(\Pc_2(\R^d))$;
\item 
the function
$\partial_t\upsilon(t,\mu)$ is continuous on $[0,T]\times L^2_\lambda(\Pc_2(\R^d))$;
\item for every $u\in U$ and every compact set $H$ of   $\R^d$,
the functions  $\partial_x\dmu \upsilon(t,\mu)(u,x)$ and  $\partial^2_x\dmu \upsilon(t,\mu)(u,x)$ 
are continuous 
 functions of $(t,\mu)\in [0,T]\times L^2_\lambda(\Pc_2(\R^d))$, uniformly in $x\in H$; more precisely, whenever $t_n\to t$, ${\bf d}(\mu_n,\mu)\to 0$, $u\in U$ and $H\subset \R^d$ is compact
we have
$$\sup_{x\in H}\left|
\partial_x\dmu \upsilon(t_n,\mu_n)(u,x) 
-
\partial_x\dmu \upsilon(t,\mu)(u,x)\right|\to 0,
$$
$$\sup_{x\in H}\left|
\partial_x^2\dmu \upsilon(t_n,\mu_n)(u,x) 
-
\partial_x^2\dmu \upsilon(t,\mu)(u,x)\right|\to 0.
$$ 
\end{enumerate}
\end{Definition}

\begin{Remark}\label{unifcontpartialv}{\em
By standard arguments, if $K$ is a compact subset of $L^2_\lambda(\Pc_2(\R^d))$ then 
the functions  
$\partial_x\dmu \upsilon(t,\mu)(u,x)$ and 
$\partial^2_x\dmu \upsilon(t,\mu)(u,x)$  
are uniformly continuous functions of $(t,\mu)\in [0,T]\times K$, uniformly in $x\in H$;
namely: for every
$u\in U$ and compact sets $H\subset \R^d$, $K\subset L^2_\lambda(\Pc_2(\R^d))$, and every $\epsilon>0$, there exists $\delta>0$ such that
$$\left|
\partial_x\dmu \upsilon(t',\mu')(u,x) 
-
\partial_x\dmu \upsilon(t,\mu)(u,x) \right|+\left|
\partial_x^2\dmu \upsilon(t',\mu')(u,x)
-
\partial_x^2\dmu \upsilon(t,\mu)(u,x)\right|< \epsilon
$$ 
whenever $u\in U$, $x\in H$, $t,t'\in [0,T]$, $  \mu,\mu'\in K$, ${\bf d}(\mu,\mu)<\delta$, $|t-t'|<\delta$.
}\end{Remark}

\vspace{2mm}

We next give some examples of functions for which we compute the linear functional derivatives.

%in $\tilde C^{1,2}([0,T]\times L^2_\lambda)$.
%ADD ASSUMPTIONS AND COMPUTE DERIVATIVES IN THE EXAMPLES BELOW

\begin{Example}
Since we are mainly interested in showing the form of the indicated functions $v$ we do not spell precise conditions on  
$\varphi$, $F$, $\phi_i$ used below ensuring that   $v$ is smooth.  Detailed assumptions are easily  determined.
\begin{enumerate}
\item[(i)] {\it Linear functions}
$$v(\mu)=
\int_U\int_{\R^d} \varphi(u,x)
%\{x^T\cdot K^u_tx\} \,
\mu^u(\d x)\,\lambda(\d u),
$$
where $\varphi$ is a measurable function with quadratic growth in $x$. 
%$(u,t)\mapsto K^u_t$ is a bounded measurable function with values in $\R^{d\times %d}$ and $\cdot$ denotes the scalar product in $\R^d$. 
Then, 
\begin{align}
\dmu v(\mu)(u,x) &= \varphi(u,x).     
\end{align}

%\item[(ii)] 
%Similarly we may consider (\red{this is a particular case of (vi)})
%$$\upsilon(t,\mu)=
%\int_U\int_U\bigg\{\left(\int_{\R^d}x \,\mu^u(dx)\right)^T
%\cdot K^{u,v}_t 
%\left(\int_{\R^d}x \,\mu^v(dx)\right)
%\bigg\}\,\lambda(du)\,\lambda(dv),
%$$
%where $(u,v,t)\mapsto K^{u,v}_t$ is a bounded measurable function with values in $\R^{d\times d}$.
%$$
%\dmu v(\mu)(u,x) = \int_U\int_{(\R^d)^2} x^T(K_t^{u,v}+(K_t^{v,u})^T)y\mu(\d y)\lambda(\d u)\lambda(\d v).
%$$

%    \item[(iii)] 
%    In a similar way  we can consider the functions
%$$v_{ij}(\mu)=
%\int_U\left(\int_{\R^d}\ x_i\, x_j \,\mu^u(dx) -
%\int_{\R^d}\ x_i \,\mu^u(dx)\int_{\R^d}\ x_j \,\mu^u(dx)
%\right)\,\lambda(du),
%$$
%which are related to the variances of $\mu^u$. \red{This is a particular case of (iv)}
%$$
%\dmu v_{ij}(\mu)(u,x) = x_ix_j - x_i\int_{\R^d}x_j\mu^u(\d x) - x_j\int_{\R^d}x_i\mu^u(\d x).
%$$
\item[(ii)] Collection of cylindrical functions: 
\begin{align}
v(\mu) &= \; \int_U     
F\Big(\int_{\R^d}\phi_1(u,x) \,\mu^u(\d x),\ldots, 
\int_{\R^d}\phi_k(u,x) \,\mu^u(\d x) \Big) \lambda(\d u). 
\end{align}
Denoting the partial derivatives of $F$ by $\partial_iF$, we have, 
\begin{align*}
\dmu v(\mu)(u,x)  &=  \sum_{i=1}^k \partial_iF
\bigg(
\int_{\R^d}\phi_1(u,x) \,\mu^u(\d x),\ldots,
\int_{\R^d}\phi_k(u,x) \,\mu^u(\d x) \bigg) \,\phi_i(u,x).
\end{align*}
\item[(iii)] Cylindrical functions of measure collection: 
\begin{align} 
v(\mu) & = \; F\Big(
\int_U\int_{\R^d}\phi_1(u,x) \,\mu^u(dx)\lambda(\d u),\ldots, 
\int_U\int_{\R^d}\phi_k(u,x) \,\mu^u(dx)\lambda(\d u)
\Big),
\end{align} 
where $F:\R^k\to \R$ and $\phi_i$ are real functions on $\R^d$. 
Then, 
%Denoting the partial derivatives of $F$ by $\partial_iF$ we have
\begin{align*}
\dmu v(\mu)(u,x) &= \sum_{i=1}^k \partial_iF
\Big(
\int_U\int_{\R^d}\phi_1(u,x) \,\mu^u(\d x)\lambda(\d u),\ldots,  \\ 
& \quad \quad \quad  \int_U\int_{\R^d}\phi_k(u,x) \,\mu^u(\d x)\lambda(\d u) \Big) \,\phi_i(u,x).
\end{align*}
\item[(iv)] $k$-interaction functions:
\begin{align}
v(\mu) &= \; \int_{U^k} 
\int_{(\R^d)^k} \varphi(u_1,\ldots,u_k,x_1,\ldots,x_k) 
\mu^{u_1}(\d x_1) \ldots \mu^{u_k}(\d x_k) \lambda(\d u_1)\ldots\lambda(\d u_k).      
\end{align}
Then, 
\begin{align}
\dmu v(\mu)(u,x) = \sum_{i=1}^k \int_{U^{k-1}} \int_{(\R^d)^{k-1}} 
\varphi(u_1,\ldots,u_{i-1},u,u_{i+1},\ldots,u_k,x_1,\ldots,x_{i-1},x,x_{i+1},\ldots,x_k) \\
  \mu^{u_1}(\d x_1) \ldots  \mu^{u_{1-1}}(\d x_{i-1}) 
 \mu^{u_{1+1}}(\d x_{i+1}) \ldots \mu^{u_k}(\d x_k) \lambda(\d u_1) \ldots 
 \lambda(\d u_{i-1}) \lambda(\d u_{i+1}) \ldots \lambda(\d u_k).  
\end{align}
%\item[(v)] Function in the form
%\begin{align}
%v(\mu) &= \; \int_U F\Big(u,\int_U \int_{\R^d} \varphi_1(u,v,x) \mu^v(dx) \lambda(dv),\ldots,
%\int_U \int_{\R^d} \varphi_k(u,v,x) \mu^v(dx) \lambda(dv) \Big) \lambda(du)     
%\end{align}
%Then, 
%\begin{align}
%\dmu v(\mu)(u,x) &=  \sum_{i=1}^k \int_U  \partial_i F\Big(v,\int_U \int_{\R^d} \varphi_1(v,w,x) \mu^w(dx) %\lambda(dw),\ldots, \\
%& \quad \int_U \int_{\R^d} \varphi_k(v,w,x) \mu^w(dx) \lambda(dw) \Big) \varphi_i(v,u,x) \lambda(dv).  
%\end{align}
\end{enumerate}
%}
\end{Example}

\subsection{The chain rule}

We are now ready to present the It\^o formula (chain rule) that will be needed in the sequel. We suppose that the conditions ${\it (1)}$ and ${\it (2)}$ in Assumption \ref{assumptionbasic} hold true.

\begin{Theorem}\label{ito}
Suppose that, for $u\in U$, $b^u$ and $\sigma^u$  are stochastic processes
defined
 on $[0,T]$, with values in $\R^d$ and $\R^{d\times \ell}$ respectively,
progressively measurable with respect to $\F^u$;  let $X_0^u$ be an $\Fc_0^u$-measurable random variable in $\R^d$. Also assume that 
 $(u,t)\mapsto \P_{(X_0^u,b^u_t,\sigma^u_t)}$ is Borel measurable and
 \begin{equation}
%\label{bsigmaintegrable}
\int_U\Big\{\E[|X_0^u|^2]+\int_0^T\E\left[|b^u_t|^2+|\sigma^u_t|^2 \right] \,dt\Big\}\,\lambda(\d u) \; < \; \infty.
 \end{equation}
Define 
\begin{align} 
X_t^u& =\; X_0^u+\int_0^t b^u_s\,ds+\int_0^t\sigma^u_s\,dW_s^u,
\qquad t\in [0,T], \quad u \in U. 
\end{align} 
 Suppose  $v\in \tilde C^{1,2}([0,T]\times L^2_\lambda(\Pc_2(\R^d)))$. Then, denoting $\mu^u_t=\P_{X_t^u}$ and $\mu_t=(\mu^u_t)_u$, we have
\begin{equation}
\label{Itoformulageneral}
    \begin{array}{l}\displaystyle
\upsilon(t,\mu_t)-v(0,\mu_0) \; = \; \int_0^t \Big\{
\partial_tv(s,\mu_s)\, + 
\\ \displaystyle
 \; \int_U   \E\Big[\partial_x\dmu v (s,\mu_s)(u, X_s^u) \cdot b_s^u + 
\frac{1}{2}\partial_x^2\dmu v (s,\mu_s)(u, X_s^u): \sigma_s^u(\sigma_s^u)\trans \Big]\,\lambda(\d u) \Big\}
   \,\d s  
\end{array}
\end{equation}
for every $t\in [0,T]$.  
%In differential notation,
%\begin{equation}
%\label{Itoformulageneral2}
%    \begin{array}{l}\displaystyle
%d\upsilon(t,\mu_t)= 
%\partial_t\upsilon(t,\mu_t)\,dt
%\\\displaystyle
%+\int_U   \E\left[\partial_x\dmu v (u, X_t^u,\mu_t) \cdot b_t^u +\frac{1}{2}\partial_x^2\dmu v (u, X_t^u,\mu_t): %\sigma_t^u(\sigma_t^u)^T \right]\,\lambda(du) \,dt.
%    \end{array}
%\end{equation}
\end{Theorem}

\begin{Remark}
We note that the measurability assumption on  $\P_{(X_0^u,b^u_t,\sigma^u_t)}$ implies that the maps
\begin{align} 
u \; \mapsto \;  \E[|X_0^u|^2], & \qquad
(u,t) \; \mapsto \; \E\big[|b^u_t|^2+|\sigma^u_t|^2 \big]
\end{align} 
\begin{align} 
(u,t) & \mapsto \; 
\E\Big[\partial_x\dmu v (s,\mu_s)(u,X_s^u) \cdot b_s^u +\frac{1}{2}\partial_x^2\dmu v (s,\mu_s)(u, X_s^u): \sigma_s^u(\sigma_s^u)\trans\Big]
\end{align}
are measurable, even if there is no measurability condition imposed on the random elements 
$X^u_0,b^u_t,\sigma^u_t$ as functions of $u\in U$: this is important for our future applications.
\end{Remark} 

\vspace{2mm}

In the following lemma we collect some preliminary facts which are used in the proof of Theorem \ref{ito}.

\begin{Lemma}
\begin{enumerate}
\item[a)] For $\nu_1,\nu_2,\mu_1,\mu_2\in L^2_\lambda(\Pc_2(\R^d))$, $\theta\in[0,1]$ we have
\begin{equation}
    \label{Wconv}
\bd ((1-\theta)\nu_1+\theta \nu_2,(1-\theta)\mu_1+\theta\mu_2)^2\le (1-\theta)\bd (\nu_1, \mu_1 )^2+\theta \bd(\nu_2, \mu_2)^2.
\end{equation}
In particular if $\nu_1=\nu_2=:\nu$ then
\begin{equation}
    \label{Wconv2}
\bd (\nu,(1-\theta)\mu_1+\theta\mu_2)^2\le (1-\theta)\bd (\nu, \mu_1 )^2+\theta \bd(\nu, \mu_2)^2.
\end{equation}
    \item [b)] We have
%    \begin{align}
$\int_U \E\big[\sup_{t \in [0,T]}|X_t^u|^2\big]\,\lambda(\d u) \; < \; \infty$, 
%    \end{align} 
    and the map $t\mapsto \mu_t$ is continuous from $[0,T]$ to $L^2_\lambda(\Pc_2(\R^d))$.
    \item[c)] The set
\begin{equation}
    \label{Kcompact}
K:=\{(1-\theta) \mu_t+\theta\mu_s\,:\, \theta\in [0,1];\,s,t\in [0,T]\}
\end{equation}
    is compact in $L^2_\lambda(\Pc_2(\R^d))$.    
\end{enumerate}    
\end{Lemma}
\begin{proof}
$(a)$ For every $u\in U$ and $i=1,2$ let $\gamma_i^u$ be an optimal coupling between $\nu^u_i$ and $\mu_i^u$ for the square distance cost used in the definition of the Wasserstein distance $\Wc_2$. 
Then $(1-\theta)\gamma^u_1+\theta\gamma_2^u$ is a coupling between $(1-\theta)\nu_1^u+\theta\nu_2^u$ and $(1-\theta)\mu_1^u+\theta\mu_2^u$. It follows that
\begin{align*}
\Wc_2\Big((1-\theta)\nu_1^u+\theta\nu_2^u,(1-\theta)\mu_1^u+\theta\mu_2^u\Big)^2
& \le \; 
\int_{\R^d\times \R^d}|x-y|^2\,
[(1-\theta)\gamma^u_1+\theta\gamma_2^u](\d x\,\d y) \\
& = \; 
(1-\theta)\Wc_2(\nu_1^u, \mu_1^u )^2+\theta \Wc_2(\nu^u_2, \mu_2^u )^2.
\end{align*}
Integrating with respect to $\lambda(\d u)$ and sending $\eps$ to 0 gives the required conclusion.

\medskip

\noindent $(b)$ We have
\begin{align} 
\sup_{t\in [0,T]}|X_t^u|^2 &\le \;  3\,|X_0^u|^2+ 3 \left(\int_0^T|b_s^u|\,\d s\right)^2+3\sup_{t\in [0,T]}  \left|\int_0^t\sigma_s^u\, \d W^u_s\right|^2
\end{align} 
and by the H\"older and Doob inequality and the It\^o isometry we have, for some absolute constant $c>0$,
\begin{align} 
\E\Big[
\sup_{t\in [0,T]}|X_t^u|^2\Big] &\le \;  c\,\E[|X_0^u|^2]+ c\,T \,\E\Big[\int_0^T|b_s^u|^2\,\d s\Big]
+ c\,\E \Big[\int_0^T|\sigma_s^u|^2\,\d s \Big].
\end{align} 
Integrating with respect to $\lambda(\d u)$, we obtain
$\int_U\E[\sup_{t \in [0,T]}|X_t^u|^2]\,\lambda(\d u)<\infty$. 
This shows that  $\mu_t\in L^2_\lambda(\Pc_2(\R^d))$ for every  $t\in [0,T]$. To prove the required continuity we note that, by similar passages as before, for $0\le s\le t\le T$,
 \begin{align*}
\bd(\mu_t,\mu_s)^2&=\int_U\Wc_2(\mu_t^u,\mu_s^u)^2\,\lambda(\d u) \;\le \; 
\int_U\E[|X_t^u-X_s^u|^2]\,\lambda(\d u) \\
&\le \;  c  \int_U\Big\{(t-s)\,\E\Big[\int_s^t|b_r^u|^2\,\d r\Big] + 
\E \Big[\int_s^t|\sigma_r^u|^2\,\d r\Big] \Big\}\,\lambda(\d u),
\end{align*}
 which tends to $0$ as $t-s\to 0$.

\medskip

\noindent $(c)$ As a consequence of the previous point,
the set $K_1:=\{\mu_t\,:\, t\in [0,T]\}$ is compact. 
We note that $K$ is the image of the compact set 
$K_1\times K_1\times [0,1]$ under the mapping $(\mu,\nu,\theta)\mapsto (1-\theta) \mu+\theta\nu$. 
To prove compactness of  $K$ it is therefore enough to show that this mapping is continuous from 
$K_1\times K_1\times [0,1]$ to $L^2_\lambda(\Pc_2(\R^d))$.
We will show that
\begin{enumerate}
    \item[(i)]
    the map $(\mu,\nu)\mapsto (1-\theta) \mu+\theta\nu$ is continuous, uniformly in  $\theta\in[0,1]$.
    \item[(ii)]
    the map $\theta\mapsto(1-\theta) \mu+\theta\nu$
     is continuous, for every fixed   $(\mu,\nu)\in K_1\times K_1$.
\end{enumerate}
These two statements easily imply the required continuity.

To prove (i) we note that if $\bd(\mu_n,\mu)\to 0$, $\bd(\nu_n,\nu)\to 0$, it follows from \eqref{Wconv} that
$$
\bd ((1-\theta)\mu+\theta \nu,(1-\theta)\mu_n+\theta\nu_n)^2\le (1-\theta)\bd (\mu, \mu_n )^2+\theta \bd(\nu, \mu_n)^2
\le \bd (\mu, \mu_n )^2+ \bd(\nu, \mu_n)^2
$$
which tends to $0$ uniformly in $\theta$.

To prove (ii) we recall the Kantorovich duality:  given $\eta,\rho\in\Pc_2(\R^d)$, the squared Wasser\-stein distance $\Wc_2(\eta,\rho)^2$ equals the supremum of
$$
\int_{\R^d}f(x)\,\eta(\d x)+ 
\int_{\R^d}g(y)\,\rho(\d y)$$
for $f,g$ varying in the set of bounded Lipschitz function on $\R^d$ satisfying the constraint
\begin{equation}
\label{dualconstraint}
f(x)+g(y)\le |x-y|^2,\qquad x,y\in\R^d.
\end{equation}
Given $\mu,\nu\in L^2_\lambda(\Pc_2(\R^d))$ and $\theta,\theta'\in[0,1]$  we apply the duality result to $\eta=(1-\theta')\mu^u+\theta'\nu^u$ and $\rho= (1-\theta)\mu^u+\theta\nu^u$ for fixed $u\in U$. 
We estimate the supremum of
$$\begin{array}
    {l}
\displaystyle
\int_{\R^d}f(x)\,[(1-\theta')\mu^u+\theta'\nu^u](\d x)+ 
\int_{\R^d}g(y)\,[(1-\theta)\mu^u+\theta\nu^u](\d y)
    \\
    \displaystyle
 =(1-\theta)
 \int_{\R^d}(f+g)\,\d\mu^u
 +\theta'
 \int_{\R^d}(f+g)\,\d\nu^u
 +( \theta-\theta')\Big(
\int_{\R^d}f\,\d\mu^u
+
\int_{\R^d}g\,\d\nu^u\Big) . 
\end{array}$$
From \eqref{dualconstraint} evaluated at $x=y$ it follows that
$f+g\le 0$, so that the first two integrals are also $\le 0$. If $\theta>\theta'$ we also have
\begin{align} 
( \theta-\theta')\left(
\int_{\R^d}f\,\d\mu^u + \int_{\R^d}g\,\d\nu^u\right) &  \le \;  ( \theta-\theta')\,\Wc_2(\mu^u,\nu^u)^2 ,
\end{align} 
and by duality it follows that
\begin{align} 
\Wc_2\Big((1-\theta')\mu^u+\theta'\nu^u,(1-\theta)\mu^u+\theta\nu^u\Big)^2 &\le \; 
| \theta-\theta'| \,\Wc_2(\mu^u,\nu^u)^2 .
\end{align} 
Interchanging $\theta$ and $\theta'$ the same inequality also holds for $\theta<\theta'$ and so for every $\theta,\theta'\in[0,1]$.

Integrating with respect to $\lambda(\d u)$ we conclude that
$$
\bd\Big((1-\theta')\mu+\theta'\nu,(1-\theta)\mu+\theta\nu\Big)^2\le
| \theta-\theta'|\, \bd(\mu,\nu)^2,
$$
which gives the required conclusion.
\end{proof}

\bigskip

\noindent {\bf Proof of Theorem \ref{ito}}. 
We first verify that the terms in the It\^o formula are well defined.
By \eqref{growthpartialv}, and recalling that $s\mapsto\mu_s$ is continuous in $L^2_\lambda(\Pc_2(\R^d))$, 
\begin{align*}  
\int_{0}^{T}\Big|\partial_x  \dmu v(s,\mu_s)(u,X_{s}^u) \cdot b_s^u \Big|\,\d s
& \le \;   C\, 
\int_{0}^{T}
 \Big(
 \left(1+|X_{s}^u|+
 \bd  (  \mu_{s},\delta_0)\right) \,|b_s^u|\Big)\,\d s \\
&\le \; C\, 
 \Big(1+\sup_{s\in [0,T]}|X_{s}^u|+\sup_{s\in [0,T]}
 \bd  (  \mu_s,\delta_0)\Big) \int_{0}^{T} |b_s^u|\,\d s.
\end{align*}
The right-hand side does not depend on $\theta$ and we have 
\begin{align} 
\int_U\E \Big[
 \Big(1+\sup_{s\in [0,T]}|X_{s}^u|+\sup_{s\in [0,T]}
 \bd  (  \mu_s,\delta_0)\Big) \int_{0}^{T}
 |b_s^u|\,\d s \Big] \,\lambda(\d u) & < \; \infty
\end{align}  
by the H\"older inequality, since 
$$
\int_U\E\Big[
 \sup_{s\in [0,T]}|X_{s}^u|^2\Big] \,\lambda(\d u) \; < \; \infty,
\qquad
\int_U\E\Big[
\int_{0}^{T}|b_s^u|^2\,\d s
 \Big] \,\lambda(\d u) \; < \; \infty.
$$
In a similar and simpler way, using the boundedness condition
\eqref{growthpartial2v} instead of \eqref{growthpartialv}, we can check that
\begin{align} 
\int_U\E \Big[
\int_{0}^{T}\Big|\frac{1}{2}\partial_x^2  \dmu v(s,\mu_s)(u,X_{s}^u): \sigma_s^u(\sigma_s^u)\trans 
\Big|\,\d s
\Big]\,\lambda(\d u)
<\infty.
\end{align}  

\vspace{2mm}

In order to prove the formula we
fix $t$ and, for every positive integer $n$, we choose a subdivision of the interval $[0,t]$ by points $t_k^n:=kt/n$, $k=0,1,\ldots, n-1$. We evaluate the difference
\begin{align}
\upsilon(t_{k+1}^n,\mu_{t^n_{k+1}})-\upsilon(t^n_{k},\mu_{t_{k}^n})
&= \; 
\upsilon(t^n_{k+1},\mu_{t_{k+1}^n})-\upsilon(t_{k}^n,\mu_{t_{k+1}^n})
+\upsilon(t_{k}^n,\mu_{t_{k+1}^n})-\upsilon(t^n_{k},\mu_{t_{k}^n}) \nonumber  \\
&=\; 
\int_{t_k^n}^{t_{k+1}^n}
\partial_tv(s, \mu_{t_{k+1}^n})\,\d s
+\upsilon(t_{k}^n,\mu_{t_{k+1}^n})-\upsilon(t^n_{k},\mu_{t_{k}^n}). \label{firstdifference}
\end{align}
We note that the last difference can be written, by the definition of the derivative $\dmu$, as
\begin{equation}\label{secondincrement}
\begin{array}{l}\displaystyle
\upsilon(t^n_{k},\mu_{t_{k+1}^n}) - \upsilon(t^n_{k},\mu_{t_{k}^n}) \; = \; 
\int_U  \int_0^1 \int_{\R^d} \dmu v\big(t_k^n,\mu_{\theta,k,n}\big)(u,x)\,
(\mu_{t_{k+1}^n}^u-\mu_{t_{k}^n}^u)(\d x) \,\d\theta \, \lambda(\d u) \\ 
\displaystyle = \; 
\int_U  \int_0^1 \E\Big[  \dmu v\big(t_k^n,\mu_{\theta,k,n}\big)(u,X^u_{t_{k+1}^n}) 
- \dmu v\big(t_k^n,\mu_{\theta,k,n}\big)(u,X^u_{t_{k}^n}) \Big]\,\d\theta \, \lambda(\d u),
\end{array}
\end{equation}
where we have set $\mu_{\theta,k,n}=\mu_{t_{k}^n} + \theta(\mu_{t_{k+1}^n}-\mu_{t_{k}^n})$. Since we assume that $\dmu v$ is twice continuously differentiable in $x$
we can apply the classical It\^o formula to the process $X^u$ on $[t_k^n,t_{k+1}^n]$:
\begin{align}
& \dmu v\big(t_k^n,\mu_{\theta,k,n}\big)(u,X^u_{t_{k+1}^n}) 
- \dmu v\big(t_k^n,\mu_{\theta,k,n}\big)(u,X^u_{t_{k}^n}) \;= \; 
\int_{t_{k}^n}^{t_{k+1}^n} 
\partial_x  \dmu v\big(t_k^n,\mu_{\theta,k,n}\big)(u,X_{s}^u)\trans \,\sigma_s^u\,\d W^u_s
\nonumber \\
& + \; \int_{t_{k}^n}^{t_{k+1}^n} \Big[  
\partial_x  \dmu v\big(t_k^n,\mu_{\theta,k,n}\big)(u,X_{s}^u) \cdot 
\,b_s^u + \frac{1}{2} \partial_x^2  \dmu v\big(t_k^n,\mu_{\theta,k,n}\big)(u,X_{s}^u)  :\,\sigma_s^u(\sigma_s^u)\trans 
\Big]\,\d s. \label{itointer} 
\end{align}
We next verify that the stochastic integral has zero expectation, by checking that the square root of its quadratic variation has finite expectation. Using \eqref{growthpartialv},
\begin{align*}
& \Big( \int_{t_{k}^n}^{t_{k+1}^n} 
\Big|  \partial_x\dmu v\big(t_k^n,\mu_{\theta,k,n}\big)(u,X_{s}^u)\trans 
\,\sigma_s^u\Big|^2\,ds\Big)^{1/2} \\
\le & \;  C\, \Big(
\int_{t_{k}^n}^{t_{k+1}^n} 
 \left(1+|X_{s}^u|^2+
 \bd  (  \mu_{\theta,k,n},\delta_0)^2\right)
\,|\sigma_s^u|^2\,\d s\Big)^{1/2} \\
\le& \;  C\,\Big(1+\sup_{s\in [t_k^n,t^n_{k+1}]}|X_{s}^u|^2+\bd 
(\mu_{\theta,k,n},\delta_0)^2 \Big)^{1/2}
 \Big(\int_{t_{k}^n}^{t_{k+1}^n}  |\sigma_s^u|^2\,ds\Big)^{1/2}.
\end{align*}
Both terms are square summable, by our assumptions, and the integrability  property of the quadratic variation follows from the H\"older inequality.
Taking expectation in the It\^o formula \eqref{itointer} and replacing in
\eqref{secondincrement}
we obtain
\begin{equation*}
%\label{}
\begin{array}{lll}\displaystyle
\upsilon(t^n_{k},\mu_{t_{k+1}^n})
-\upsilon(t^n_{k},\mu_{t_{k}^n})&=&\displaystyle 
\int_U  \int_0^1
\int_{t_k^n}^{t_{k+1}^n}\E\Big[\partial_x  \dmu v\big(t_k^n,\mu_{\theta,k,n}\big)(u,X_{s}^u) \cdot 
\,b_s^u \\ 
&&\displaystyle + \; 
\frac{1}{2} \partial_x^2  \dmu v\big(t_k^n,\mu_{\theta,k,n}\big)(u,X_{s}^u) 
:\,\sigma_s^u(\sigma_s^u)\trans \Big]\,\d s\,\d \theta \, \lambda(\d u).
\end{array}
\end{equation*}
Coming back to  \eqref{firstdifference} and summing over $k$ we arrive at
$$
\begin{array}{lll}
\upsilon(t,\mu_{t})-v(0,\mu_{0})
&=&\displaystyle
\sum_{k=0}^{n-1}
\int_{t_k^n}^{t_{k+1}^n}
\partial_tv(s, \mu_{t_{k+1}^n})\,\d s
\\&&\displaystyle
+ \; \int_U  \int_0^1
\sum_{k=0}^n
\int_{t_k^n}^{t_{k+1}^n}\E\Big[\partial_x  \dmu v\big(t_k^n,\mu_{\theta,k,n}\big)(u,X_{s}^u) \cdot 
\,b_s^u \\ & & \displaystyle
+ \;  \frac{1}{2} \partial_x^2  \dmu v\big(t_k^n,\mu_{\theta,k,n}\big)(u,X_{s}^u)  
:\,\sigma_s^u(\sigma_s^u)\trans \Big]\,\d s\,\d\theta \, \lambda(\d u).
\end{array}
$$
We may write the It\^o formula in similar way, decomposing the integral over $[0,t]$ into a sum of integrals over $[t_k^n,t_{k+1}^n]$, and we see that setting
\begin{align} 
I_1^n & := \;  \sum_{k=0}^{n-1} \int_{t_k^n}^{t_{k+1}^n}
\Big(\partial_tv(s, \mu_{t_{k+1}^n})- \partial_tv(s, \mu_s)\Big) \,\d s,
\end{align} 
\begin{align} 
I_2^n & := \; \int_U  \int_0^1\E\Big[
\sum_{k=0}^n
\int_{t_k^n}^{t_{k+1}^n}\!\!\Big\{\partial_x  \dmu v\big(t_k^n,\mu_{\theta,k,n}\big)(u,X_{s}^u) \cdot 
\,b_s^u - \partial_x  \dmu  v\big(s,\mu_{s}\big)(u,X_{s}^u) \cdot  \,b_s^u
\Big\}\,\d s\Big]\,\d \theta \, \lambda(\d u)
\end{align} 
$$\begin{array}{lll}
I_3^n &: =& \displaystyle
\int_U  \int_0^1\E\Big[
\sum_{k=0}^n
\int_{t_k^n}^{t_{k+1}^n}\Big\{
  \frac{1}{2}
\partial_x^2  \dmu v\big(t_k^n,\mu_{\theta,k,n}\big)(u,X_{s}^u) 
:\,\sigma_s^u(\sigma_s^u)\trans 
\\&&\displaystyle \qquad \quad \quad - \; 
\frac{1}{2}
\partial_x^2  \dmu v\big(s,\mu_{s}\big)(u,X_{s}^u) 
:\,\sigma_s^u(\sigma_s^u)\trans  \Big\}\,ds\Big]\,\d\theta \, \lambda(\d u),
\end{array}
$$
it is enough to prove that $I^n_1+I_2^n+I_3^n\to 0$. 
It is immediate to see that $I^n_1\to 0$, since $\partial_t v$ is continuous, hence uniformly continuous, on $[0,T]\times K_1$, where $K_1=\{\mu_s\,:\, s\in [0,T]\}$, because $K_1$ is compact in $L^2_\lambda(\Pc_2(\R^d))$ by the continuity of $s\mapsto \mu_s$.

Next we consider $I_2^n$. 
We fix $u,\theta,\omega$. We note that for $\P$-almost all $\omega$ the set
$$
H^{u,\omega}:=\{X_s^u(\omega)\,:\, s\in [0,T]\}
$$
is compact in $\R^d$ due to continuity of $s\mapsto X_s^u(\omega)$. 
By our assumptions and 
Remark \ref{unifcontpartialv}, 
for every compact $K\subset L^2_\lambda(\Pc_2(\R^d))$ and $\epsilon>0$  there exists $\delta>0$ such that
$$
\left|
\partial_x\dmu \upsilon(t',\mu)(u,x)
-
\partial_x\dmu \upsilon(t,\mu)(u,x) \right|< \epsilon
$$ 
whenever   $x\in H^{u,\omega}$, $t,t'\in [0,T]$, $  \mu,\mu'\in K$, ${\bf d}(\mu,\mu')<\delta$, $|t-t'|<\delta$.
As the set $K$, we choose the one defined in \eqref{Kcompact}
and we note that points of the form $(s,\mu_s)$ and $(t^n_k, \mu_{\theta,k,n})$ belong to $[0,T]\times K$. It follows that, $\P$-a.s.,
$$\left|
\partial_x\dmu \upsilon(t^n_k,\mu_{\theta,k,n})(u,X^u_s) 
-
\partial_x\dmu v(s,\mu_s)(u,X^u_s) \right|< \epsilon
$$
provided the inequalities
$$|s-t^n_k|<\delta,\quad {\bf d}(\mu_s,\mu_{\theta,k,n})<\delta, \qquad \quad s\in [t^n_k,t^n_{k+1}],\;
k=0,\ldots,n-1,
$$ are satisfied. 
We finally note that 
$|s-t^n_k|\le |t^n_{k+1}-t^n_k| \le 1/n$ and, by \eqref{Wconv2}, the squared distance
   \begin{align*}
   {\bf d}(\mu_s,\mu_{\theta,k,n})^2 
   & = \;   
   {\bf d}(\mu_s,(1-\theta)\mu_{t_k^n}+\theta \mu_{t_{k+1}^n})^2
   \\
   &\le \;   (1-\theta)  
   {\bf d}(\mu_s,\mu_{t_k^n})^2
   +\theta {\bf d}(\mu_s,\mu_{t_{k+1}^n})^2 
   %\\   
   %& 
   \; \le   \; 
   {\bf d}(\mu_s,\mu_{t_k^n})^2
   +  {\bf d}(\mu_s,\mu_{t_{k+1}^n})^2
   \end{align*} 
can be made arbitrarily small by taking $n$ sufficiently large, by the uniform continuity of $s\mapsto \mu_s$ on $[0,T]$. It follows that for all $n$ large enough (depending on $u,\theta,\omega$)
$$
\sum_{k=0}^n
\int_{t_k^n}^{t_{k+1}^n}\Big|\partial_x  \dmu v\big(t_k^n,\mu_{\theta,k,n}\big)(u,X_{s}^u) \cdot  
\,b_s^u -\partial_x  \dmu 
v\big(s,\mu_{s}\big)(u,X_{s}^u) \cdot  \,b_s^u \Big|\,\d s\le \epsilon\,T.
$$
We have thus proved that for fixed $\theta\in [0,1]$ and $u\in U$ we have, $\P$-a.s.,
$$
\sum_{k=0}^n
\int_{t_k^n}^{t_{k+1}^n}\partial_x  \dmu v\big(t_k^n,\mu_{\theta,k,n}\big)(u,X_{s}^u) \cdot 
\,b_s^u\,\d s \; \to \; 
\int_{0}^{t} \partial_x  \dmu  v\big(s,\mu_{s}\big)(u,X_{s}^u) \cdot \,b_s^u\,\d s.
$$ 
In order to conclude that $I_2^n\to 0$ we wish to apply the dominated convergence theorem and pass to the limit under the expectation sign and under the integrals over $[0,1]$ and $U$. To this end we consider the following estimates.
By \eqref{growthpartialv}, 
\begin{align*}
 \sum_{k=0}^n
\int_{t_k^n}^{t_{k+1}^n}\Big|\partial_x  \dmu v\big(t_k^n,\mu_{\theta,k,n}\big)(u,X_{s}^u) \cdot 
\,b_s^u \Big|\,\d s
&\le  \; 
C\,\sum_{k=0}^n
\int_{t_k^n}^{t_{k+1}^n}
 \Big(
 \left(1+|X_{s}^u|+
 \bd  (  \mu_{\theta,k,n},\delta_0)\right)
\,|b_s^u|\Big)\,\d s
\\
& \le \; 
C\, 
 \Big(1+\sup_{s\in [0,T]}|X_{s}^u|+\sup_{\mu\in K}
 \bd  (  \mu,\delta_0)\Big)
\int_{0}^{t}
 |b_s^u|\,\d s
\end{align*}
 The right-hand side does not depend on $n$ nor $\theta$ and satisfies
$$
\int_U\E \Big[
 \Big(1+\sup_{s\in [0,T]}|X_{s}^u|+\sup_{\mu\in K}
 \bd  (  \mu,\delta_0)\Big)
\int_{0}^{t}
 |b_s^u|\,\d s
\Big]\,\lambda(\d u)
\; < \; \infty
$$ 
by the H\"older inequality, since 
$$
\int_U\E\Big[
 \sup_{s\in [0,T]}|X_{s}^u|^2\Big] \,\lambda(\d u)
\; < \; \infty,
\qquad
\int_U\E\Big[
\int_{0}^{T}|b_s^u|^2\,\d s
 \Big] \,\lambda(\d u)
<\infty.
$$ 
So we can apply the dominated convergence theorem (three times) and we conclude that $I_2^n\to 0$.

Finally, the proof that $I_3^n\to 0$ is similar to the proof that $I_2^n\to 0$ and even simpler, since we apply the boundedness condition
\eqref{growthpartial2v} instead of \eqref{growthpartialv}.
\ep
%\qedsymbol

\bigskip

\begin{Remark}
We note that Theorem \ref{ito}   applies 
to  the process $(X^u)_u$ solution to the state equation \eqref{stateeq}. As a
consequence, if $v\in \tilde C^{1,2}([0,T]\times L^2_\lambda(\Pc_2(\R^d)))$ we have, on the interval $[t,T]$,
\begin{align} \label{itov} 
%    \begin{array}{lll}\displaystyle
\d v(s,\mu_s)&= \;  \partial_s v(s,\mu_s)\,\d s +
\int_U \E \Big[\partial_x\dmu v (s,\mu_s)(u, X_s^u) \cdot b(u,X^u_s,\alpha^u_s,\mu_s, 
\P_{\alpha^{\cdot}_s}) \\ 
& \quad  + \; \frac{1}{2}\partial_x^2\dmu v (s,\mu_s)(u, X_s^u): \sigma\sigma\trans(u,X^u_s,\alpha_s^u,\mu_s, 
\P_{\alpha^{\cdot}_s}) \Big]\,\lambda(\d u) \,\d s.
%    \end{array}
\end{align}
%In the special case when the control process is constant, $\alpha\equiv a$, we find \red{do we use the following equality?}
%\begin{align} \label{Itoupsilon}
%\d v(s,\mu_s)&= \;  \partial_sv(s,\mu_s)\,\d s 
%+ \int_U \int_{\R^d} \Big(\partial_x\dmu v (s,\mu_s)(u,x) \cdot b(u,x,a,\mu_s,(\delta_a)_{v} )
%\\ 
%&\quad   + \; \frac{1}{2}\partial_x^2\dmu v(s,\mu_s)(u,x) : \sigma\sigma\trans(u,x,a,\mu_s,(\delta_a)_{v})
%\Big) \mu^u(\d x)\,\lambda(\d u) \,\d s.
%\end{align}
\end{Remark}

%\beq\label{test}
%toto & = & toto
%\enq
%we now use \eqref{test}
\section{The Bellman equation} \label{secHJB} 

In this section, we make the standing Assumptions \ref{assumptionbasic}, \ref{assumptiononfgbis}, \ref{assumptiononfg}.

\subsection{The equation}

%\end{document}

For $\pi=(\pi^u)_u\in L^2_\lambda(\Pc_2(\R^d\times A))$ we denote by $\pi^u_1\in\Pc_2(\R^d)$, $\pi^u_2\in \Pc_2(A)$ the marginals of $\pi^u\in \Pc_2(\R^d\times A)$ and we set  $\pi_1=(\pi^u_1)_u\in L^2_\lambda(\Pc_2(\R^d))=L^2_\lambda(\Pc_2(\R^d))$, 
 $\pi_2=(\pi^u_2)_u\in L^2_\lambda(\Pc_2(A))$. 

We next introduce the Hamiltonian $\Hc$ defined by 
\beqs
\Hc(u,t,\pi,\varphi) & = &
\int_{\R^d\times A} \Big(\partial_x\dmu \varphi (t,\pi_1)(u,x) \cdot b(u,x,a,\pi_1,\pi_2) \\
 & & \quad  + \;  \frac{1}{2} \partial_x^2\dmu \varphi (t,\pi_1)(u,x) : \sigma\sigma\trans(u,x,a,\pi_1,\pi_2) 
 +f(u,x ,a,\pi_1 ,\pi_2)\Big) \pi^u(\d x,\d a)
\enqs
for $(u,t,\pi,\varphi)\in U\times [0,T]\times L^2_\lambda(\Pc_2(\R^d\times A))\times \tilde C^{1,2}([0,T]\times L^2_\lambda(\Pc_2(\R^d)))$. In our framework, the Bellman equation is written as 
\begin{align} \label{HJB}
- \partial_t  \upsilon(t,\mu )  - \Inf_{\pi\in L^2_\lambda(\Pc_2(\R^d\times A)),\pi_1=\mu}
\int_U \Hc(u,t,\pi,\upsilon)\lambda(\d u) &  = \;  0, 
\end{align}
for  $(t,\mu) \in[0,T)\times L^2_\lambda(\Pc_2(\R^d))$, together with the terminal condition
\begin{align} \label{TCHJB}
\upsilon(T,\mu )  & =  \;  \int_U \int_{\R^d} g(u,x,\mu)\,\mu^u(\d x)\,\lambda(\d u)\;, \quad  \mu\in L^2_\lambda(\Pc_2(\R^d)).
\end{align}

\subsection{The regular case}

In the case where the value function is smooth, we provide a verification result.
%\textcolor{red}{We do not need the quadratic growth assumption on $w$ as previously written. Actually, if $w\in \tilde C^{1,2}([0,T]\times L^2_\lambda)$ it has quadratic growth.}

\begin{Theorem}[Verification]\label{VERIFTHM} Let $w\in \tilde C^{1,2}([0,T]\times L^2_\lambda(\Pc_2(\R^d)))$. %satisfying the growth condition
%\beqs
%|w(t,\mu)| & \leq & C\big(1+\bd(\mu,\delta_0)^2\big), \quad t \in [0,T], \; \mu \in L_\lambda^2(\Pc_2(\R^d)). 
%\enqs
%for some positive constant $C$. 
\begin{enumerate}[(i)]
\item Suppose that $w$ is solution to \reff{HJB}-\reff{TCHJB}
%satisfies
%\beqs
%        w(T,\mu) & = & \displaystyle
%\int_U \int_{\R^d} g(u,x,\mu)\,\mu^u(dx)\,\lambda(du)\;,\quad \mu\in L^2_\lambda(\Pc_2(\R^d))\;,
%        \enqs
%        and
%         \beq\label{HJBVERIF}
%        -\partial_t  w(t,\mu )   -\displaystyle
%\inf_{\tiny{\begin{array}{c}\pi\in L^2_\lambda(\Pc_2(\R^d\times A))\\ \pi_1=\mu\end{array}}}\int_U \int_{\R^d} \bigg(\partial_x\dmu w (t,u,x,\mu) \cdot b(u,x,a,\mu,\pi_2) \qquad\qquad\qquad & & 
%\\ \nonumber +
%\displaystyle
%\frac{1}{2}\partial_x^2\dmu w (t,u,x,\mu): %\sigma\sigma^T(u,x,a,\mu,\pi_2) \bigg)\mu^u(dx)\,\lambda(du)  & = & 0
%        \enq       
%        for all $(t,\mu)\in [0,T]\times L^2_\lambda(\Pc_2(\R^d))$.
Then we have $w\leq \upsilon$ on $[0,T]\times L^2_\lambda(\Pc_2(\R^d))$.
\item  Suppose further that there exists a Borel map $\hat a:~U\times[0,T]\times \R^d\times L^2_\lambda(\Pc_2(\R^d))\rightarrow A$ such that for any $(t,\mu)\in[0,T]\times L^2_\lambda(\Pc_2(\R^d))$  the infimum in \eqref{HJB}, with $w$ in place of $\upsilon$, is reached at a point $\hat\pi=(\hat\pi^u)$  which has the form
$$
\hat\pi^u=\mu^u\circ (Id \times \hat a(u,t,\cdot,\mu))^{-1}\,\quad u\in U\;,
$$
i.e. such that its marginal $\pi^u_1$ is $\mu^u$ and the marginal $\pi^u_2$ is the image of $\mu^u$ under a suitable map $x\mapsto \hat a(u,t,x,\mu)$, and the system 
\begin{equation*}%\label{SDEHATVERIF}
     \left\{
\begin{array}{l}
    d X^u_s  =  b\Big(u,X^u_s, \hat a(u,s,X^u_s, \P_{X^{\cdot}_s}),  \P_{X^{\cdot}_s}, ( \P_{\hat a(v,s,X^{v}_s,\P_{X^{\cdot}_s} )} )_{v}\Big)\d s \\
        + \; \sigma\Big(u,X^u_s, \hat a(u,s,X^u_s, \P_{X^{\cdot}_s}),  \P_{X^{\cdot}_s}, ( \P_{\hat a(v,s,X^{v}_s,\P_{X^{\cdot}_s} )} )_{v}\Big)\d W^u_s,
    \;\; s\in [t,T],
    \\
    X_t^u  =  \xi^u, \;\; u\in U,
\end{array}\right.
 \end{equation*}
 admits a unique solution, in the sense of Definition \ref{solutionstate}, denoted by $(\hat X^{t,\xi,u})_u$ such that $(\hat \alpha^u)_u:=(\hat a(u,s,\hat X^{t,\xi,u}_s,\P_{\hat X^{t,\xi,\cdot}_s}))_{u\in U,s\in[0,T]}\in\Ac$  for any $t\in [0,T]$ and any $\xi\in \Ic_t$. Then $w=v$ on $[0,T]\times L^2_\lambda(\Pc_2(\R^d))$ and $\hat \alpha$ is an optimal Markov control:
 \beqs
\upsilon(t,\mu) & = & J\big(t,\xi,(\hat \alpha^u)_{u}\big)
 \enqs
 for $t\in[0,T]$, $\mu\in L^2_\lambda(\Pc_2(\R^d))$ and $\xi\in\Ic_t$ such that $(\P_{\xi_u})_u=\mu$.
        \end{enumerate}
\end{Theorem}
\begin{proof}
(i) Fix $(t,\mu)\in [0,T]\times L^2_\lambda(\Pc_2(\R^d))$. Let $\xi\in \Ic_t$ such that $\P_{\xi^\cdot}=\mu$  and a control $\alpha\in\Ac$. Denote by $( X^{u})_u$ the unique solution to the system
\begin{equation}\label{SDEverif(i)}
     \left\{
\begin{array}{rcl}
    d X^u_s & = & b\big(u,X^u_s, \alpha^u_s,  \P_{X^{\cdot}_s},  \P_{\alpha_s^{\cdot}}\big) \d s 
    %\\
    %  & &  
      + \sigma\big(u,X^u_s, \alpha^u_s,  \P_{X^{\cdot}_s},  \P_{\alpha_s^{\cdot}}\big)\d W^u_s,
    \; s\in [t,T],
    \\
    X_t^u & = & \xi^u, \;\; u\in U.
\end{array}\right.
 \end{equation} 
 We next write $\mu_s$, $\pi_s$, $b_s^u(x,a)$, $\sigma_s^u(x,a)$ and $f_s^u(x,a)$
 for $\P_{X^{\cdot}_s}$, $\P_{(X^\cdot_s,\alpha^\cdot_s)}$, $b(u,x,a,\P_{X^{\cdot}_s},\P_{\alpha^{\cdot}_s})$, $\sigma(u,x,a,\P_{X^{\cdot}_s}, \P_{\alpha^{\cdot}_s})$, and 
 $f(u,x,a,\P_{X^{\cdot}_s}, \P_{\alpha^{\cdot}_s})$ respectively. 
From Theorem \ref{ito} we have
\begin{equation*}
    \begin{array}{l}\displaystyle
w(t,\mu) \; = \; w(T,\mu_T)-\int_t^T \Big\{
\partial_t w(s,\mu_s)\,\d s \; + \; 
\\ 
\displaystyle 
\int_U   \int_{\R^d\times A}\Big[\partial_x\dmu w(s,\mu_s)(u,x) \cdot b_s^u(x,a) 
+ \frac{1}{2} \partial_x^2\dmu w(s,\mu_s)(u,x) : \sigma_s^u(\sigma_s^u)\trans (x,a) \Big] \pi(\d x,\d a)\,
\lambda(\d u) \Big\} \, \d s  
\end{array}
\end{equation*}
Since $w$ is solution to \reff{HJB}-\reff{TCHJB}, we get
\beqs
w(t,\mu) & \leq & \int_U \Big[\int_{\R^d} g(u,x,\mu_T)\,\mu^u_T(\d x) + 
\int_t^T\int_{\R^d\times A}f_s^u(u,x)\pi_s(\d x,\d a)\d s\Big]\lambda(\d u) 
%\\& = & 
\; = \; J(t,\xi,\alpha).
\enqs
%and so 
%\beqs
%w(t,\mu) & \leq & J(t,\xi,\alpha)\;.
%\enqs
Since this inequality holds for any $\alpha\in \Ac$, we get $w$ $\leq$ $\upsilon$. 
%\beqs
%w(t,\mu) & \geq & \upsilon(t,\mu)\;.
%\enqs

\vspace{1mm}

\noindent (ii) We proceed as in (i) with the control $(\hat \alpha^u)_u$ instead of $\alpha$. We then get
\begin{equation*}
    \begin{array}{rcl}\displaystyle
w(t,\mu) & = &\displaystyle w(T,\mu_T)-\int_t^T \Big\{
\partial_tw(s,\mu_s)\,
\displaystyle
+\inf_{\pi\in L^2_\lambda(\Pc_2(\R^d\times A)),\pi_1=\mu
}\int_U \Hc(u,t,\pi,w)\lambda(\d u) 
%\int_U   \int_{\R^d\times A}\left[\partial_x\dmu v (s,u, x,\mu_s) \cdot b_s^u +\frac{1}{2}\partial_x^2\dmu w (s,u, x,\mu_s): \sigma_s^u(\sigma_s^u)^T\right]\pi(dx,da)\,\lambda(du) \bigg\}
%\\ & & \displaystyle+\int_{\R^d\times A}f(u,x,\mu_s,a)\hat \pi_s(dx,da)
\Big\}  
\,\d s  \\
& & \quad + \displaystyle \int_t^T \int_U \int_{\R^d\times A}f_s(u,x,a,\P_{\hat X^{t,\xi,\cdot}_s},\P_{\hat\alpha^{\cdot}_s})\P_{(\hat X^{t,\xi,u}_s,\hat \alpha^u_s)}(\d x,\d a) \Big]\lambda(\d u) \d s.  
\end{array}
\end{equation*}
Since $w$ is solution to \reff{HJB}-\reff{TCHJB}, we get
\beqs
w(t,\mu) & = & \int_U \E \Big[  g(u,\hat X_T^{t,\xi,u},\P_{\hat X^{t,\xi,\cdot}_T})  
+ \int_t^T f(u,\hat X^{t,\xi,u}_s,\hat \alpha^{u}_s,\P_{\hat X^{t,\xi,\cdot}_s},\P_{\hat\alpha^{\cdot}_s}) \d s \Big] 
\lambda(\d u) \\
&=& J(t,\xi,\hat\alpha) \;  \geq \; \upsilon(t,\mu).
%\int_U \Big[\int_{\R^d} g(u,x,(\P_{\hat X^{v,t,\xi}_T}))_{v})\,\P_{\hat X^{t,\xi,u}_T}(\d x)\\
% & & +\int_t^T\int_{\R^d\times A}f(u,x,(\P_{\hat X^{v,t,\xi}_s}))_{v}),a)\P_{(\hat X^{t,\xi,u}_s,\hat \alpha^u_s)}
% (\d x,\d a)\d s\Big]\lambda(\d u)
\enqs
Therefore, $w(t,\mu)=\upsilon(t,\mu)$ $=$ $J(t,\xi,\hat\alpha)$. 
\end{proof}

\begin{Remark} 
In the case where the coefficients $b$, $\sigma$ and $f$ do not depend on the second marginal $\pi_2$ of $\pi$, the PDE \reff{HJB} takes the following form
\begin{eqnarray}
 \label{HJBbis}
-\partial_t  \upsilon(t,\mu ) 
- \int_U \int_{\R^d}\inf_{a\in A}H(u,t,x,\mu,\upsilon,a)\mu(\d x)\lambda(\d u) &  = & 0
\;,\; (t,\mu)\in[0,T)\times L^2_\lambda(\Pc_2(\R^d))\;,
\end{eqnarray}
with
\beqs
H(u,t,x,\mu,\varphi,a) & = & \partial_x\dmu \varphi (t,\mu)(u,x) \cdot b(u,x,a,\mu) 
 + \frac{1}{2}\partial_x^2\dmu \varphi (t,\mu)(u,x) : \sigma\sigma\trans(u,x,a,\mu) \\
 & & + \; f(u,x,a,\mu), 
\enqs
for $(t,u,\mu,\varphi,a)\in[0,T]\times U\times L^2_\lambda(\Pc_2(\R^d))\times \tilde C^{1,2}([0,T]\times L^2_\lambda(\Pc_2(\R^d)))\times A$.

%Another possible formulation, perhaps under additional assumptions, is
%\begin{equation}
%    \label{HJBbis}
%    \left\{
%    \begin{array}{lll}
%-\partial_t  \upsilon(t,\mu ) 
%&-&\displaystyle
%\int_U \int_{\R^d} \inf_{a\in A}\bigg(\partial_x\dmu v (t,u,x,\mu) \cdot b(u,x,a,\mu,(\delta_a)_{v}) 
%\\ & +&
%\displaystyle
%\frac{1}{2}\partial_x^2\dmu v (t,u,x,\mu): \sigma\sigma^T(u,x,a,\mu,(\delta_a)_{v}) \bigg)\mu^u(dx)\,\lambda(du) =0,
%\\
%\upsilon(t,\mu ) &=&\displaystyle
%\int_U \int_{\R^d} g(u,x,\mu)\,\mu^u(dx)\,\lambda(du).
%    \end{array}
%    \right.
%\end{equation}
In this case, using the same arguments, one can prove Theorem \ref{VERIFTHM} (i). Moreover, if the Borel map $\hat a$ is replaced by a Borel map $\tilde a:~U\times\R^d\times L^2_\lambda(\Pc_2(\R^d))\rightarrow A$  such that
\beqs
H(u,t,x,\mu,w,\tilde a(u,t,x,\mu)) & = & \inf_{a\in A} H(u,t,x,\mu,w,a)
\enqs
for all $(t,u,x,\mu)\in [0,T]\times U\times \R^d\times L^2_\lambda(\Pc_2(\R^d))$ 
and the system 
\begin{equation*}%\label{SDEHATVERIF}
     \left\{
\begin{array}{lcl}
    d X^u_s  &=&  b\big(u,X^u_s, \tilde a(u,s,X^u_s, \P_{X^{\cdot}_s}),  \P_{X^{\cdot}_s}, \big)\d s \\
        & & + \;  \sigma\big(u,X^u_s, \tilde a(u,s,X^u_s, \P_{X^{\cdot}_s}),  \P_{X^{\cdot}_s}, \big)\d W^u_s,
    \;\; s\in [t,T],
    \\
    X_t^u  &=&  \xi^u, \;\; u\in U,
\end{array}\right.
 \end{equation*}
 admits a unique solution, in the sense of Definition \ref{solutionstate}, denoted by $(\tilde X^{t,\xi,u})_u$ such that $(\tilde \alpha^u)_u:=(\tilde a(u,.,\tilde X^{t,\xi,u},\P_{\tilde X^{t,\xi,\cdot}}))_{u}\in\Ac$  for any $t\in [0,T]$ and any $\xi\in \Ic_t$. Then $w=\upsilon$ on $[0,T]\times L^2_\lambda(\Pc_2(\R^d))$ and $\tilde \alpha$ is an optimal Markov control:
 \beqs
\upsilon(t,\mu) & = & J\big(t,\xi,(\tilde \alpha^u)_{u}\big)
 \enqs
 for $t\in[0,T]$, $\mu\in L^2_\lambda(\Pc_2(\R^d))$ and $\xi\in\Ic_t$ such that $\P_{\xi^\cdot}=\mu$. The proof follows the same lines as that of Theorem \ref{VERIFTHM} (ii).
\end{Remark}

%Verification theorem: suppose that in \eqref{HJB} the infimum is attained at a point $\hat\pi=(\hat\pi^u)$  which has the form
%$$
%\hat\pi^u=(Id \times \hat a(u,\cdot,\mu))_\#\,\mu
%$$
%i.e. such that its marginal $\pi^u_1$ is $\mu$ and the marginal $\pi^u_2$ is the image of $\mu$ under a suitable map $x\mapsto \hat a(u,x,\mu)$. 
%Then the optimal control processes are given by $  \hat a(u,X_s^u,(\P_{X_s^u})_u)$ where $(X_s^u)_u$
%is the solution to the closed-loop equation.

\subsection{Viscosity properties}

For a  locally bounded function $w:~[0,T]\times L^2_\lambda(\Pc_2(\R^d))\rightarrow \R$ (i.e. bounded on bounded sets), we define its lower and upper semicontinuous envelopes respectively by
\beqs
\displaystyle 
w_*(t,\mu) \;  = \;  \liminf_{(s,\nu)\rightarrow(t,\mu),\;s<T}w(s,\nu), &  \quad \quad 
\displaystyle 
w^*(t,\mu) \; = \;  \limsup_{(s,\nu)\rightarrow(t,\mu),\;s<T}w(s,\nu), 
\enqs
for $(t,\mu)\in [0,T]\times L^2_\lambda(\Pc_2(\R^d))$.

\begin{Definition}\label{DEFSolVisco} Let $w:~[0,T]\times L^2_\lambda(\Pc_2(\R^d))\rightarrow \R$ be a  locally bounded function.
    \begin{enumerate}[(i)]
        \item We say that $w$ is a viscosity subsolution to   \eqref{HJB}-\eqref{TCHJB} if
        \beqs
        w^*(T,\mu) & \leq & \displaystyle
\int_U \int_{\R^d} g(u,x,\mu)\,\mu^u(\d x)\,\lambda(\d u)\;,\quad \mu\in L^2_\lambda(\Pc_2(\R^d))\;,
        \enqs
        and for any $\varphi\in \tilde C^{1,2}([0,T]\times L^2_\lambda(\Pc_2(\R^d)))$  and $(t,\mu)\in [0,T]\times L^2_\lambda(\Pc_2(\R^d))$, such that
        \beqs
        (w^*-\varphi)(t,\mu) & = & \max_{[0,T]\times L^2_\lambda(\Pc_2(\R^d))}(w^*-\varphi) 
        \enqs
        we have
        \beqs
        -\partial_t  \varphi(t,\mu )   -\displaystyle
\inf_{\pi\in L^2_\lambda(\Pc_2(\R^d\times A)), \pi_1=\mu}\int_U \Hc(u,t,\pi,\varphi) %\int_{\R^d} \bigg(\partial_x\dmu \varphi (t,u,x,\mu) \cdot b(u,x,a,\mu,\pi_2) \qquad\qquad\qquad 
%& & 
%\\  +
%\displaystyle
%\frac{1}{2}\partial_x^2\dmu \varphi (t,u,x,\mu): \sigma\sigma^T(u,x,a,\mu,\pi_2) \bigg)\mu^u(dx)\,
\lambda(\d u)  & \leq & 0\;.
        \enqs

        \item We say that $w$ is a viscosity supersolution to   \eqref{HJB}-\eqref{TCHJB} if
        \beqs
        w_*(T,\mu) & \geq & \displaystyle
\int_U \int_{\R^d} g(u,x,\mu)\,\mu^u(\d x)\,\lambda(\d u)\;,\quad \mu\in L^2_\lambda(\Pc_2(\R^d))\;,
        \enqs
        and for any $\varphi\in \tilde C^{1,2}([0,T]\times L^2_\lambda(\Pc_2(\R^d)))$  and $(t,\mu)\in [0,T]\times L^2_\lambda(\Pc_2(\R^d))$, such that
        \beqs
        (w_*-\varphi)(t,\mu) & = & \min_{[0,T]\times L^2_\lambda(\Pc_2(\R^d))}(w_*-\varphi) 
        \enqs
        we have
        \beqs
        -\partial_t  \varphi(t,\mu )   -\displaystyle
\inf_{\pi\in L^2_\lambda(\Pc_2(\R^d\times A)), \pi_1=\mu}\int_U 
\Hc(u,t,\pi,\varphi)
%\int_{\R^d} \bigg(\partial_x\dmu \varphi (t,u,x,\mu) \cdot b(u,x,a,\mu,\pi_2) \qquad\qquad\qquad & & 
%\\  +
%\displaystyle
%\frac{1}{2}\partial_x^2\dmu \varphi (t,u,x,\mu): \sigma\sigma^T(u,x,a,\mu,\pi_2) \bigg)\mu^u(dx)
\,\lambda(du)   & \geq & 0\;.
        \enqs
        \item We say that $w$ is a viscosity solution to   \eqref{HJB}-\eqref{TCHJB} if  $w$ is both a viscosity subsolution and supersolution to   \eqref{HJB}-\eqref{TCHJB}. 

    \end{enumerate}
\end{Definition}

\begin{Assumption}\label{HamiltonianAss}
   \begin{enumerate}[(i)] 
   \item There exist a constant $M\ge 0$ such that
 \beqs
|b(u,x,a,\mu,\nu)|+
|\sigma(u,x,a,\mu,\nu)|
%+|f(u,x,a,\mu,\nu)|
  & \le &  M\big( 1+|x|+ \bd(\mu,\delta_0)\big)
  \enqs
 for every  $u\in U$, $x,x'\in\R^d$, $\mu,\mu'\in L^2_\lambda(\Pc_2(\R^d))$, $\nu\in L^2_\lambda(\Pc_2(A))$, $a\in A$.
%\textcolor{red}{Idris: I need this assumption to control the dynamics uniformly in the control to get a contradiction in the subsolution part}   
   \item The function
    \beqs
    (t,\mu)\in [0,T]\times L^2_\lambda(\Pc_2(\R^d)) & \mapsto & \inf_{\pi\in L^2_\lambda(\Pc_2(\R^d\times A)),\pi_1=\mu
}\int_U \Hc(u,t,\pi,\varphi)\lambda(\d u) %&  = & 0
%\;,\; (t,\mu)\in[0,T)\times L^2_\lambda
    \enqs
     is continuous for any $\varphi\in \tilde C^{1,2}([0,T]\times L^2_\lambda(\Pc_2(\R^d)))$.
      \end{enumerate}
    
\end{Assumption}

%\textcolor{red}{Idris: give examples :LQ pb, convex hamiltonian, $A$ compact...}
%\end{Theorem}

\begin{Theorem}\label{ThmVisco} 
Under Assumption \ref{HamiltonianAss}(i), 
%Suppose that the value function $\upsilon$ is locally bounded (\red{It is by \eqref{Vgrowth}).} 
the  value function $\upsilon$ is a viscosity supersolution to \reff{HJB}-\reff{TCHJB}. Furthermore, if Assumption \ref{HamiltonianAss}(ii) holds, 
then $\upsilon$ is a viscosity subsolution to \reff{HJB}-\reff{TCHJB}. 
\end{Theorem}
\begin{proof}  We first notice that the function $\upsilon$ is locally bounded by \eqref{growthV}. 
%\red{Marco: there is no formula (2.9). Also, are we assuming already 5.1?}
We turn to the viscosity properties.

\noindent \textbf{1.} {\it Viscosity subsolution property on $[0,T)\times L^2_\lambda(\Pc_2(\R^d))$.} 
Let $(t,\mu)\in[0,T)\times L^2_\lambda(\Pc_2(\R^d))$ and  $(t_n,\mu_n)_n$ be a sequence of $[0,T)\times L^2_\lambda(\Pc_2(\R^d))$ such that
\beqs
(t_n,\mu_n,\upsilon(t_n,\mu_n)) & \xrightarrow[n\rightarrow+\infty]{} & (t,\mu,\upsilon^*(t,\mu))\;.
\enqs 
Fix some $\pi\in L^2_\lambda(\Pc_2(\R^d\times A))$ such that $\pi_1=\mu$. Let $(\pi_n)_n$ be a sequence of $L^2_\lambda(\Pc_2(\R^d\times A))$ such that
$\pi_{n,1}=\mu_n$ for $n\geq 1$ and $\pi_n\rightarrow\pi$ as $n\rightarrow+\infty$. Such a sequence can be constructed by decomposing $\pi$ as $\pi^u(\d x,\d y)=\mu^u(\d x)\gamma^u(x,\d y)$ with $\gamma$  a probability kernel and taking $\pi_n^u(\d x,\d y) = \mu_n^u(\d x)\gamma^u(x,\d y)$ for $u\in U$. %\textcolor{red}{Idris: justify existence of such a sequence}.

Applying Theorem \ref{repseqTHM} with $t_n=t=0$, %\textcolor{red}{Idris: here we need the representation thm},
there exists a sequence $(\underline \xi^n,\underline{\mathfrak{a}}^n)_n$ such that $\underline \xi^n$,$\mathfrak{a}^n$ are Borel maps from $U\times (0,1)$ to $\R^d$ and $A$, and $(\P_{(\xi^{n,u},\mathfrak{a}^{n,u})})_u= \pi_n$ with $\xi^{n,u}=\underline{\xi}^{n,u}(Z^u)$ and $\mathfrak{a}^{n,u}=\underline{\mathfrak{a}}^{n,u}(Z^u)$ for all $n\geq 1$. 
We define the control $\alpha\in \Ac$ by $\alpha^u_t=\mathfrak{a}(Z^u)$ for $t\in[0,T]\times U$.  
We consider the family of processes $(X^n)_n=(X^{n,u})_{n,u}$ as the unique solution to the SDE
\begin{equation*}
     \left\{
\begin{array}{rcl}
    d X^{n,u}_s & = & b\big(u,X^{n,u}_s, \alpha^{n,u}_s,  \P_{X^{n,\cdot}_s},  \P_{\alpha^{n,\cdot}_s}\big)
    \d s \\
      & &  + \;  \sigma\big(u,X^{n,u}_s, \alpha^{n,u}_s,  \P_{X^{n,\cdot}_s}, \P_{\alpha_s^{n,\cdot}}\big)
      \d W^u_s, \;\; s\in [t_n,T],
    \\
    X_{t_n}^{n,u} & = & \xi^{n,u}, \;\; u\in U.
\end{array}\right.
 \end{equation*} 
From the DPP  given by Corollary \ref{CorDPPv},  we have
\begin{align} \label{DPPsupsol}
\upsilon(t_n,\mu_n) & \leq \;  \int_U\E\Big[ \int_{t_n}^{t_n+h} f\big(u, X^{n,u}_s,\alpha^{n,u}_s,\P_{X^{n,\cdot}_s}, \P_{\alpha_s^{n,\cdot}}\big)\,\d s  \Big]\,\lambda(\d u)\\\nonumber
 & \quad \quad  + \;  \upsilon\big(t_n+h, \P_{X^{n,\cdot}_{t_n+h}}\big)
\end{align} 
for $h>0$ small enough.  Let $\varphi\in \tilde C^{1,2}([0,T]\times L^2_\lambda(\Pc_2(\R^d)))$ such that 
$(\upsilon^*-\varphi)(t,\mu)$ $=$  $\Max_{[0,T]\times L^2_\lambda(\Pc_2(\R^d))}(\upsilon^*-\varphi)$, 
%\begin{align} \label{vsousphi} 
%        (v^*-\varphi)(t,\mu) & = \;  \max_{[0,T]\times L^2_\lambda}(v^*-\varphi) \;,
%        \end{align} 
so that with \eqref{DPPsupsol} 
\begin{align} \label{DPPsousphi}
0 & \leq \;  \int_U\E\Big[ \int_{t_n}^{t_n+h} f\big(u, X^{n,u}_s,\alpha^{n,u}_s,\P_{X^{n,\cdot}_s}, \P_{\alpha_s^{n,\cdot}}\big)\,\d s  \Big]\,\lambda(\d u)\\\nonumber
 & \quad \quad  + \;  \varphi\big(t_n+h, \P_{X^{n,\cdot}_{t_n+h}}\big) - \varphi(t_n,\mu_n) 
 + \gamma_n,
 %+ \big[ v^*(t,\mu) - \upsilon(t_n,\mu_n) \big]  +  \big[ \varphi(t_n,\mu_n) - \varphi(t,\mu) \big].  
\end{align}
where we set $\gamma_n$ $:=$ $\upsilon^*(t,\mu) - \upsilon(t_n,\mu_n) + \varphi(t_n,\mu_n) - \varphi(t,\mu)$ $\rightarrow$ $0$ as $n$ goes to infinity. 
By applying It\^o formula in Theorem \ref{ito} to $\varphi(s,(\P_{X^{n,u}_s})_{u})$ between $t_n$ and $t_n+h$, and substituting into the above inequality, we then get
\begin{align}
0 & \leq \; \int_{t_n}^{t_n+h} \Big[   \partial_t\varphi(s,\mu_{n,s}) + 
\int_U \Hc(u,s,\pi_{n,s},\varphi) \lambda(\d u) \Big] \d s + \gamma_n,  
%& \quad \quad + \;  \big[ v^*(t,\mu) - \upsilon(t_n,\mu_n) \big]  +  \big[ \varphi(t_n,\mu_n) - \varphi(t,\mu) \big],
\end{align}
where we set $\mu_{n,s}$, $\pi_{n,s}$ for $\P_{X^{n,\cdot}_s}$, $\P_{(X^{n,\cdot}_s,\alpha^{n,\cdot}_s)}$ to alleviate notations.  By sending $n$ to infinity, and using the continuity of all the involved functions, this implies
\begin{align}
0 & \leq \; \int_{t}^{t+h} \Big[   \partial_t\varphi(s,\mu_{s}) + 
\int_U \Hc(u,s,\pi_{s},\varphi) \lambda(\d u) \Big] \d s.  
%& \quad \quad + \;  \big[ v^*(t,\mu) - \upsilon(t_n,\mu_n) \big]  +  \big[ \varphi(t_n,\mu_n) - \varphi(t,\mu) \big],
\end{align}
Dividing by $h$ and sending $h$ to 0, we get 
\beqs
 \displaystyle
 -\partial_t\varphi(t,\mu_{}) - \int_U    \Hc(u,t,\pi,\varphi) \,\lambda(\d u) & \leq & 0\;.
\enqs
Since this inequality holds for any $\pi\in L^2_\lambda(\Pc_2(\R^d\times A))$ such that $\pi_1=\mu$, we get the viscosity subsolution property on $[0,T)\times L^2_\lambda(\Pc_2(\R^d))$. 

\vspace{1mm}

\textbf{2.} {\it Viscosity subsolution property on $\{T\}\times L^2_\lambda(\Pc_2(\R^d))$}.
Let $(t_n,\mu_n)_n$ be a sequence of $[0,T)\times L^2_\lambda(\Pc_2(\R^d))$ such that
\beqs
(t_n,\mu_n,\upsilon(t_n,\mu_n)) & \xrightarrow[n\rightarrow+\infty]{} & (T,\mu,\upsilon^*(T,\mu))\;.
\enqs 
Fix some $a\in A$ and define the control $\alpha\in \Ac$ by $\alpha^u_t=a$ for $t\in[0,T]\times U$. From Theorem \ref{repseqTHM}, there exists $\xi\in \Ic_T$ such that $\P_{\xi^\cdot}=\mu$ and $(\xi^n)_n$ such that $\xi^n\in\Ic_{t_n}$ $\P_{\xi^{n,\cdot}}=\mu_n$ for $n\geq 1$ and 
\begin{align} \label{convL2xixin}
\int_0^T\E\big[|\xi^{n,u}-\xi^{u}|^2\big]\lambda(\d u) & \xrightarrow[n\rightarrow+\infty]{} \;  0\;.
\end{align} 
 Define the family of processes $(X^n)_n=(X^{n,u})_{n,u}$ as the unique solution to the SDE
\begin{equation*}
     \left\{
\begin{array}{rcl}
    d X^{n,u}_s & = & b\big(u,X^{n,u}_s, \alpha^u_s,  \P_{X^{n,\cdot}_s},  \P_{\alpha_s^\cdot}\big)\d s \\
      & &  \quad + \;  \sigma\big(u,X^{n,u}_s, \alpha^u_s,  \P_{X^{n,{\cdot}}_s} , 
      \P_{\alpha_s^{\cdot}}\big)\d W^u_s,
    \;\; s\in [t_n,T],
    \\
    X_{t_n}^{n,u} & = & \xi^{n,u}, \;\; u\in U.
\end{array}\right.
 \end{equation*} 
By definition of the value function $\upsilon$,  we have
\beqs
\upsilon(t_n,\mu_n) & \leq & \int_U\E\Big[ \int_{t_n}^T f\big(u, X^{n,u}_s,a,\P_{X^{n,\cdot}_s}, \delta_a\big)\,\d s + g\big(u, X^{n,u}_T, \P_{X^{n,\cdot}_T}\big) \Big]\,\lambda(\d u)\;.
\enqs
From \eqref{convL2xixin} and Proposition \ref{WANTED} %\textcolor{red}{Idris: stability result}
and Assumption \ref{assumptiononfg}, we get
\beqs
\int_U\E\Big[ \int_{t_n}^T f\big(u, X^{n,u}_s,a,\P_{X^{n,\cdot}_s}, \delta_a\big)\,\d s + g\big(u, X^{n,u}_T, \P_{X^{n,\cdot}_T}\big) \Big]\,\lambda(\d u) &     & \\ 
\xrightarrow[n\rightarrow+\infty]{} \quad  \int_U \int_{\R^d} g(u,x,\mu)\,\mu^u(\d x)\,\lambda(\d u) & &
\enqs
and so 
\begin{align} \label{CondTestFuncSupSol}
\upsilon^*(T,\mu) & \leq \;  \int_U \int_{\R^d} g(u,x,\mu)\,\mu^u(\d x)\,\lambda(\d u)\;.
\end{align} 

\vspace{1mm}

\textbf{3.} {\it Viscosity supersolution property on $[0,T)\times L^2_\lambda(\Pc_2(\R^d))$}.
We argue by contradiction and suppose that there exist $(t,\mu)\in[0,T)\times L^2_\lambda(\Pc_2(\R^d))$, $\eta>0$ and $\varphi\in \tilde C^{1,2}([0,T]\times L^2_\lambda(\Pc_2(\R^d)))$ such that
        \begin{align} \label{CondTestFuncSubSol}
        (\upsilon_*-\varphi)(t,\mu) & = \;  \min_{[0,T]\times L^2_\lambda(\Pc_2(\R^d))}(\upsilon_* - \varphi) 
        \end{align} 
and
\beqs
        \partial_t  \varphi(t,\mu )   + \displaystyle
\inf_{\pi\in L^2_\lambda(\Pc_2(\R^d\times A)),\; \pi_1=\mu}\int_U 
\Hc(u,t,\pi,\varphi) \,\lambda(\d u)~~ =: ~~2\eta   & > & 0\;.
        \enqs
From Assumption \ref{HamiltonianAss} (ii), there exists some $\eps>0$ 
%and $\pi\in L^2_\lambda(\Pc_2(\R^d\times A))$ 
such that 
%$ \pi_1=\mu$ and
\begin{align} \label{eqContSubsol}
        \partial_t  \varphi(s,\nu )    + \displaystyle
\inf_{\pi\in L^2_\lambda(\Pc_2(\R^d\times A)) ,\; \pi_1=\nu}
\int_U  \Hc(u,s,\pi,\varphi) \,\lambda(\d u)~~   & \geq \;  \eta \;.
\end{align} 
for $s\in[t,(t+\eps)\wedge T]$ and $\nu\in B_{L^2_\lambda(\Pc_2(\R^d))}(\mu,\eps)$.
      %and $\gamma\in B_{\Pc_2(\R^d\times A)}(\pi,\eps)$. 
 Let  $(t_n,\mu_n)_n$ be a sequence of $[t,(t+\eps)\wedge T)\times B_{L^2_\lambda(\Pc_2(\R^d))}(\nu,\eps)$ such that
\begin{align} \label{ApproxSeqSubSol}
(t_n,\mu_n,\upsilon(t_n,\mu_n)) & \xrightarrow[n\rightarrow+\infty]{} \;  (t,\mu,\upsilon_*(t,\mu))\;.
\end{align} 
%Without loss of generality, we can assume that $(t_n,\mu_n)\in[t,(t+\eps)\wedge T]\times B_{\Pc_2(\R^d)}(\nu,\eps)$ for all $n\geq1$.
From Theorem \ref{repseqTHM}, there exists $\xi$ such that $\xi\in\Ic_t$ and $(\P_{\xi^{u}})_u= \mu$, and  a sequence $(\xi^n)_n$ such that $\xi^n\in\Ic_{t_n}$, $(\P_{(\xi^{n,u})})_u=\mu_n$ for all $n\geq 1$ %and $\P_{(\xi^{n,u},\alpha_{t_n}^{n,u})_u}= \pi^n\in B_{\Pc_2(\R^d\times A)}(\pi,\eps)$ for all $n\geq 1$
and
\beqs
\int_0^T\E\big[|\xi^{n,u}-\xi^{u}|^2
%+|\alpha^{n,u}-\alpha^{u}_t|^2
\big]\lambda(du) & \xrightarrow[n\rightarrow+\infty]{} & 0\;.
\enqs 
%Define the family of processes $(X^n)_n=(X^{n,u})_{n,u}$ such that $X^n=(X^{n,u})_{u}$ is the unique solution to the SDE
%\begin{equation*}
%     \left\{
%\begin{array}{rcl}
%    d X^{n,u}_s & = & b\left(u,X^{n,u}_s, \alpha^u_{t_n},  \P_{X^{n,\cdot}_s}, ( \P_{\alpha_{t_n}^{v}})_{v}\right)d s \\
%      & &  + \sigma\left(u,X^{n,u}_s, \alpha^u_{t_n},  \P_{X^{n,\cdot}_s}, ( \P_{\alpha_{t_n}^{v}})_{v}\right)d W^u_s,
%    \;\; s\in [t_n,T],
    \\
%    X_{t_n}^{n,u} & = & \xi^{n,u}, \;\; u\in U.
%\end{array}\right.
% \end{equation*} 
%\vspace{2mm}
We next define the sequence  $(\rho_n)_n$ by
\beqs
\rho_n & = & \upsilon(t_n,\mu_n)-\varphi(t_n,\mu_n)-\big(\upsilon_*(t,\mu)-\varphi(t,\mu)\big) %\inf\big\{ s\geq t_n~:~\big(\P_{(X^{n,u}_s, \alpha_{t_n}^{u} )}\big)_{u}\notin B_{\Pc_2(\R^d\times A)}(\pi,\eps)  \big\}\wedge T
\enqs
for $n\geq1$. From \eqref{CondTestFuncSubSol} and \eqref{ApproxSeqSubSol}, we have $\rho_n\geq 0$ for $n\geq 1$, 
and $\rho_n$ $\rightarrow$ $0$, as $n$ goes to infinity. We take a sequence $(h_n)_n$ of $(0,+\infty)$ such that
\beqs
h_n  \xrightarrow[n\rightarrow+\infty]{}  0 & \mbox{ and } & 
\frac{\rho_n}{h_n}  \xrightarrow[n\rightarrow+\infty]{}  0\;.
\enqs
 From the DPP given by Corollary \ref{CorDPPv}, there exists $ \alpha^n\in \Ac$ such that
\begin{align} \label{ineqDPPContSubSol}
\upsilon(t_n,\mu_n) + \frac{\eta h_n}{2} & \geq \;  \int_U\E\Big[ \int_{t_n}^{\theta_n} f\big(u, X^{n,u}_s,\alpha^{n,u}_s,\P_{X^{n,\cdot}_s}, \P_{\alpha_s^{n,\cdot}}\big)\,\d s  \Big]\,\lambda(\d u) \\
 & \quad  \quad + \;  \upsilon\big(\theta_n, \P_{X^{n,\cdot}_{\theta_n}}\big), \nonumber
\end{align} 
where the family of processes $(X^n)_n=(X^{n,u})_{n,u}$ %such that $X^n=(X^{n,u})_{u}$ is 
stands for the unique solution to the SDE
\begin{equation*}
     \left\{
\begin{array}{rcl}
    d X^{n,u}_s & = & b\big(u,X^{n,u}_s, \alpha^{n,u}_{s},  \P_{X^{n,\cdot}_s},  \P_{\alpha_{s}^{n,\cdot}}\big) \d s \\
      & &  \quad + \; \sigma\big(u,X^{n,u}_s, \alpha^{n,u}_{s},  \P_{X^{n,\cdot}_s}, \P_{\alpha_{s}^{n,\cdot}}\big)
      \d W^u_s,
    \;\; s\in [t_n,T],
    \\
    X_{t_n}^{n,u} & = & \xi^{n,u}, \;\; u\in U,
\end{array}\right.
 \end{equation*} 
for $n\geq 1$, and the sequence $(\theta_n)_n$ is defined by
\beqs
\theta_n & = & \tau_n\wedge (t_n+h_n), \quad \mbox{ with } \quad 
\tau_n \; = \;  \inf\big\{ s\geq t_n~:~\P_{X^{n,\cdot}_s}\notin B_{L^2_\lambda(\Pc_2(\R^d))}(\mu,\eps)  \big\}\wedge T. 
\enqs
From \eqref{ineqDPPContSubSol} and the definition of the sequence $(\rho_n)_n$, we have
\beqs
\varphi(t_n,\mu_n)+\rho_n + \frac{\eta h_n}{2} & \geq & \int_U\E\Big[ \int_{t_n}^{\theta_n} f\big(u, X^{n,u}_s,\alpha^{n,u}_s,\P_{X^{n,\cdot}_s}, \P_{\alpha^{n,\cdot}_s}\big)\,\d s  \Big]\,\lambda(\d u)\\
 & & \quad  + \;  \varphi\big(\theta_n, \P_{X^{n,\cdot}_{\theta_n}}\big),
\enqs
and thus by applying It\^o formula in Theorem \ref{ito},  we get
\beqs
\frac{1}{h_n}\int_{t_n}^{\theta_n} \Big\{
\partial_t\varphi(s,\mu_{n,s}) + \int_U \Hc(u,s,\pi_{n,s},\varphi)\lambda(\d u)\Big\}\,\d s   
& \leq &  \frac{\rho_n}{h_n} + \frac{\eta}{2}. 
\enqs
where we set again $\mu_{n,s}$, $\pi_{n,s}$ for $\P_{X^{n,\cdot}_s}$, $\P_{(X^{n,\cdot}_s,\alpha^{n,\cdot}_s)}$.  
From \eqref{eqContSubsol},  this yields 
\beq\label{LastIneqContSubSol}
\frac{\rho_n}{h_n} + \eta\Big(\frac{1 }{2}-\frac{\theta_n-t_n}{h_n}\Big)  & \geq & 0\;.
\enq
Now, we notice that $\tau_n$ $\geq$ $\gamma_n$ for $n$ large enough, where
\beqs
\gamma_n & := & \inf\Big\{ s\geq t_n~:~\int_U\E\Big[|X^{n,u}_s-\xi^{u,n}|^2\Big]\lambda(du)\geq  \frac{\eps^2}{2}  \Big\}\wedge T.
\enqs
From Assumption \ref{HamiltonianAss} (i) and classical estimates on diffusion processes, we have
$\inf_{n\geq 1} (\gamma_n-t_n)$ $>$ $0$, and thus 
\beqs
\frac{\theta_n-t_n}{h_n} & \xrightarrow[n\rightarrow+\infty]{} & 1. 
\enqs
Therefore, we get a contradiction by sending $n$ to $\infty$ in \eqref{LastIneqContSubSol}.

\vspace{1mm}

\textbf{4.} {\it Viscosity subsolution property on $\{T\}\times L^2_\lambda(\Pc_2(\R^d))$.} 
Let $(t_n,\mu_n)_n$ be a sequence of $[0,T)\times L^2_\lambda(\Pc_2(\R^d))$ such that
\beqs
(t_n,\mu_n,\upsilon(t_n,\mu_n)) & \xrightarrow[n\rightarrow+\infty]{} & (T,\mu,\upsilon_*(T,\mu)). 
\enqs 
By definition of the function $\upsilon$, there exists a control $\alpha^n\in \Ac$ such that
\begin{align} 
\upsilon(t_n,\mu_n) + \frac{1}{n} & \geq \;  \int_U\E\Big[ \int_{t_n}^T f\big(u, X^{n,u}_s,\alpha^{n,u}_s,\P_{X^{n,\cdot}_s}, \P_{\alpha^{n,\cdot}_s}\big)\,\d s \\
& \qquad \quad \quad \quad + \;  g\big(u, X^{n,u}_T, \P_{X^{n,\cdot}_T}\big) \Big]\,\lambda(\d u)\;.
\end{align} 
From Assumption \ref{HamiltonianAss} (i) and classical estimates on diffusion processes
we get
\beqs
\int_U\E\Big[ \int_{t_n}^T f\big(u, X^{n,u}_s,\alpha^{n,u}_s,\P_{X^{n,\cdot}_s}, \P_{\alpha^{n,\cdot}_s}\big)\,
\d s \Big]\,\lambda(\d u) & \xrightarrow[n\rightarrow+\infty]{} & 0\;.
\enqs
Therefore, using the continuity of $g$, we get by sending $n$ to $+\infty$ 
\beqs
 \upsilon_*(T,\mu) & \geq & \displaystyle
\int_U \int_{\R^d} g(u,x,\mu)\,\mu^u(\d x)\,\lambda(\d u)\;.
\enqs
\end{proof} 

\begin{Remark}
From Definition \ref{DEFSolVisco}, we notice that if $w\in\tilde C^{1,2}([0,T]\times L^2_\lambda(\Pc_2(\R^d)))$
is a viscosity solution to  \eqref{HJB}-\eqref{TCHJB}, then $w$ is a classical solution as one can take $\varphi=w$ as a test function. In particular, under the additional assumption $\upsilon\in\tilde C^{1,2}([0,T]\times L^2_\lambda(\Pc_2(\R^d)))$,  we get from Theorem \ref{ThmVisco} that $\upsilon$ is a classical solution to  \eqref{HJB}-\eqref{TCHJB}.
\end{Remark}
%Definition of viscosity solution, for functions $\upsilon(t,\mu)$ locally bounded, using the upper semicontinuous envelopes as supersolutions and lower semicontinuous envelopes as subsolutions

%DPP + Ito $\Longrightarrow$ viscosity property

\appendix

\section{Some results on the collection of state equations 
%\eqref{stateeq} 
} 
\label{secappenthmexuuniq}

\subsection{Proof of Theorem \ref{exuniq}}

We write the proof in the case $b\equiv 0$, the general case being completely similar.

We fix $\xi\in\Ic_t$ and $\alpha\in\Ac$. Adapting an idea in %Bayraktar et al., Graphon mean field systems, 
Proposition 2.1 \cite{bayetal23},
the existence for equation \eqref{stateeq} will be obtained by a fixed point argument in the space $L^2_\lambda(\Pc_2(C_{[t,T]}))$. For any $\nu$ in this space and for any $u\in U$ let us denote $X^{\nu,u}$ the solution to
\begin{equation}\label{stateeqwithnu}
     \left\{
\begin{array}{l}
    d X^{\nu,u}_s =  \sigma\left(u,X^{\nu,u}_s,\alpha^u_s, \left(\nu^{v}_s \right)_{v},\P_{\alpha^{\cdot}_s}\right) \d W^u_s,
    \\
    \alpha_s^u=\alpha(u,s, W^u_{\cdot\wedge s},Z^u),\;\; s\in [t,T],
    \\
    X_t^{\nu,u}=\xi^u\;,\quad u\in U.
\end{array}\right.
 \end{equation}
 These equations 
 differ from the original ones because the given $\left( \nu_s^v \right)_{v}$ replaces
$ \left( \P_{X^{v}_s} \right)_{v}$. The equations are no longer coupled so that they can be solved individually. Under our assump\-tions each of them satisfies  standard Lipschitz conditions and it has a 
continuous $\F^u$-adapted solution,
 unique up to a $\P$-null set. 
Let us define ${\bf \Psi}(\nu)=({\bf \Psi}(\nu)^u)_u$   setting ${\bf \Psi}(\nu)^u= \P_{X^{\nu,u}}$,  the law of the process $X^{\nu,u}$ seen as a probability on $C_{[t,T]}^d$. 
We will prove the following claim:
\begin{equation}
\begin{array}{c}
    \label{measurabilityclaim}
\text{the map }
u\mapsto\Lc (
X^{\nu,u},W^u_{[0,T]},Z^u)  \text{ is Borel measurable }\\
\text{from } U \text{ to }
\Pc_2(C^d_{[t,T]}\times C^\ell_{[0,T]}\times  (0,1)).
\end{array}
\end{equation}
Admitting this for a moment, standard estimates on the stochastic equation in \eqref{stateeqwithnu} show that
\begin{equation*}
\E\Big[\sup_{s\in [t,T]}|X_s^{\nu,u}|^2\Big] \le
c\left(1+ \E[|\xi^u|^2]
+\int_0^T\Big\{\E[d_A(\alpha_s^u,0)^2]+ \Wc_2(\nu_s^u,\delta_0)^2  
+\Wc_2(\P_{\alpha_s^v},\delta_0)^2\Big\}\,
\d s\right)
\end{equation*}
for some constant $c>0$ (that does not depend on $\nu$). Noting that  $\Wc_2(\P_{\alpha_s^u},\delta_0)^2\le \E[d_A(\alpha_s^u,0)^2]$ and
recalling the admissi\-bi\-li\-ty condition 
\eqref{alphaadmissible} we find
$\|X^\nu\|^2=
\int_U
\E[\sup_{s\in [t,T]}|X_s^{\nu,u}|^2]\,\lambda(\d u)
<\infty $ and so $X^\nu\in\Sc_t$. 
This implies that 
$( \P_{X^u})_u$ belongs to  $L^2_\lambda(\Pc_2(C_{[t,T]}^d))$, so that the map
 ${\bf \Psi}:L^2_\lambda(\Pc_2(C_{[t,T]}^d))\to L^2_\lambda(\Pc_2(C_{[t,T]}^d))$ is well defined.

We will next prove the following claim:
\begin{equation}\label{fixedpointPsi}
{\bf \Psi} \text{ has unique fixed point } \bar\nu \text{ in } 
L^2_\lambda(\Pc_2(C_{[t,T]}^d)).
\end{equation} 
This leads immediately to the required conclusion. Indeed, the process   $X^{\bar\nu}$ corresponding to the fixed point is  clearly a solution. Moreover,
if $(\tilde X^u)_u\in \Sc_t$ is another solution then its law  
$(\P_{\tilde X^u})_u$ also belongs to $L^2_\lambda(\Pc_2(C_{[t,T]}^d))$ and it is a fixed point of ${\bf \Psi}$; it must therefore coincide with $\bar\nu$ and it follows that 
 $X^u=\tilde X^u$  for $\lambda$-almost all $u\in U$ since they are both solutions to the same equations \eqref{stateeqwithnu}.

In order to conclude the proof we have to prove the two claims.

\medskip

 {\it Step I: proof of claim \eqref{measurabilityclaim}}.
 
%\textcolor{red}{Put this notation at the begining: We need to introduce some notation. For fixed $t\in[0,T]$ we define the concatenation $\hat{w}\oplus \check{w} \in C^\ell_{[0,T]}$ of paths  $\hat{w}\in C^\ell_{[0,t]}$ and $\check{w}\in C^\ell_{[t,T]}$    by the formula
%$$
%(\hat{w}\oplus \check{w})(s)=\left\{ \begin{array}{ll}
%\hat{w}(s)& \text{ if } s\in [0,t],
%\\
%\hat{w}(t)+\check{w}(s)-\check{w}(t)& \text{ if } s\in [t,T].
%\end{array}\right.
%$$}
We note that $\hat{w}\oplus \check{w}$ depends on $\check{w}$ only through its increments, namely denoting $\check{w}-\check{w}_t$ the function $(\check{w}(s)-\check{w}(t))_{s\in [t,T]}$ we have
$\hat{w}\oplus \check{w}=\hat{w}\oplus (\check{w}- \check{w}_t)$.

%We will use this notation to write the control processes in a suitable way. 
%We recall that the definiton of the operator  $\oplus$ is given by \reff{DefConc}. 
Given $t\in [0,T]$ and a Borel measurable function $\alpha: U\times [0,T]\times C_{[0,T]}^\ell\times (0,1)\to A
$, we define a Borel measurable function   $\tilde\alpha: U\times [0,T]\times C_{[t,T]}^\ell\times C_{[0,t]}^\ell\times (0,1)\to A$
setting
\begin{equation}\label{operatorbeta}
    \tilde{\alpha}(u,s,\check{w},\hat{w},z)=\alpha(u,s,\hat{w}\oplus \check{w},z)
\end{equation}
for $u\in U$, $s\in [0,T]$, $\hat{w}\in C^\ell_{[0,t]}$,  $\check{w}\in C^\ell_{[t,T]}$, $z\in (0,1)$. We note that this formula establishes a bijection between those classes of functions whose inverse is
\beqs
\alpha(u,s,w,z) & = & \tilde{\alpha}(u,s, w_{ [0,t]},w_{[t,T]},z),
\enqs
where we recall that $w_{ [0,t]}$, $w_{[t,T]}$ denote the restrictions of $w\in C^\ell_{[0,T]}$ to the indicated intervals. 
%Both $\oplus$ and $\tilde{\alpha}$ depend on $t$, but we do not make it explicit in the notation. 
Finally we note that, for $s\in [t,T]$,
\begin{equation}\label{operatorbeta2}
\alpha(u,s,w_{s\wedge \cdot}, z)  =   \tilde{\alpha}(u,s,w_{s\wedge \cdot} ,w_{ [0,t]},z).
\end{equation}
where we write $w_{s\wedge \cdot}$ instead of the more cumbersome $\big(w_{[t,T]}\big)_{s\wedge \cdot}$.

Below we need a similar  notation for functions defined on spaces of paths. Given
$\Phi:C^d_{[t,T]}\times
C^\ell_{[0,T]}\times (0,1)\to\R$ 
it is convenient to consider the function $\tilde{\Phi}:C^d_{[t,T]}\times
C^\ell_{[t,T]}\times C^\ell_{[0,t]}\times (0,1)\to\R$
defined   by
\begin{equation}\label{operatorbetaphi}
\tilde{\Phi}(x,\check{w},\hat{w},z)=
\Phi(x,\hat{w}\oplus\check{w},z),
\quad
x\in C^d_{[t,T]}, \check{w}\in C^\ell_{[t,T]}, \hat{w}\in C^\ell_{[0,t]},
  z\in  (0,1).
\end{equation}

Using these notations
we write equation \eqref{stateeqwithnu} in a different way. Recalling 
 $\eqref{operatorbeta}$ and $\eqref{operatorbeta2}$ we first have
$$\alpha_s^u=\tilde{\alpha}(u,s, W^u_{\cdot\wedge s},W^u_{[0,t]},Z^u)=
\tilde{\alpha}(u,s, W^u_{\cdot\wedge s}-W^u_t,W^u_{[0,t]},Z^u)\;,\quad s\in[t,T]\;.
$$ 
We note that $( W^u_{[0,t]},Z^u)$ is independent of the increments  $W^u_{\cdot\wedge s}-W^u_t$ and the law 
of $(W^u,Z^u)$ is the product $\mathbb{W}_T\otimes m$ of the Wiener measure $\mathbb{W}_T$ on $C^\ell_{[0,T]}$ and the Lebesgue measure $m$ on $(0,1)$. Given a bounded continuous $\phi:A\to\R$ we have
\begin{align}
\E\Big[
\phi (\alpha^u_s)
\Big]
&=
\E\Big[
\phi \Big(\tilde{\alpha}(u,s, W^u_{\cdot\wedge s},W^u_{[0,t]},Z^u)\Big)
\Big]
\\
&=
\int_{ C^\ell_{[0,T]}\times  (0,1)}
\E\Big[
\phi\Big(\tilde{\alpha}(u,s, W^u_{\cdot\wedge s},w'_{[0,t]},z')\Big)\Big]\,
\Lc (W^{u},Z^u)(\d w'\,\d z')
\\
&=
\int_{ C^\ell_{[0,T]}\times  (0,1)}
\E\Big[
\phi\Big(\tilde{\alpha}(u,s, W^u_{\cdot\wedge s},w'_{[0,t]},z')\Big)\Big]\,
(\mathbb{W}_T\otimes m)(\d w'\,\d z').
\end{align} 
Setting $
 \alpha_s^{u,w,z}=\tilde{\alpha}(u,s,   W^u_{\cdot\wedge s},w_{[0,t]},z)$ this 
shows that the laws $\P_{\alpha^u_s}$  may be written
\begin{equation}\label{palphaint}
    \P_{\alpha^u_s}=
    \int_{ C^\ell_{[0,T]}\times  (0,1)}
    \P_{\alpha_s^{u,w',z'}}\,
(\mathbb{W}_T\times m)(\d w'\,\d z').
\end{equation}

Equation \eqref{stateeqwithnu} then becomes
\begin{equation}\label{stateeqwithnurewritten}
     \left\{
\begin{array}{l}
    d X^{\nu,u}_s \!
    =  
    \!\sigma
    \!
    \left(u,X^{\nu,u}_s,\alpha^u_s,  ( \nu_{s}^v  )_{v},\left(  
    \int_{ C^\ell_{[0,T]}\times  (0,1)}
    \P_{\alpha_s^{v,w',z'}}\,
(\mathbb{W}_T\times m)(dw'\,dz')
    \right)_{v}\right)\d W^u_s,
    \; s\in [t,T],
    \\
    X_t^{\nu,u}=\xi^u,
    \\
    \alpha_s^u=\tilde{\alpha}(u,s, W^u_{\cdot\wedge s},W^u_{[0,t]},Z^u).
\end{array}\right.
 \end{equation}
Let us consider the analogue of this   equation where the random elements $\xi^u$, $W^u_{[0,t]}$, $Z^u $ are ``freezed'' at given points $x\in\R^d$, $w\in C^\ell_{[0,t]}$, $z\in (0,1)$, namely
\begin{equation}\label{Xuxwz}
     \left\{
\begin{array}{l}
    d   X^{\nu,u,x,w,z}_s 
    \!\!
    =  \sigma\!\left(
    \!
    u,  X^{\nu,u,x,w,z}_s
    \!\!
    ,\alpha^{u,w,z}_s\!\!,  ( \nu^{v}_s  
 )_{v},\left(
 \int_{ C^\ell_{[0,T]}\times  (0,1)}
    \P_{\alpha_s^{v,w',z'}}\,
(\mathbb{W}_T\times m)(dw'\,dz')\right)_{v}\right) \d  W^u_s
    \\
    X_t^{\nu,u,x,w,z}=x,
    \\
 \alpha_s^{u,w,z}=\tilde{\alpha}(u,s,   W^u_{\cdot\wedge s},w,z).
\end{array}\right.
 \end{equation}
For fixed $u\in U$, this is a stochastic equation depending measurably on the parameters $x,w,z$ and it admits as a solution a measurable function $(\omega,s,x,w,z)\mapsto X_s^{\nu,u,x,w,z}(\omega)$. 
Measurability is understood in the following sense. Since the equation is driven by the increments of the Brownian motion $(W_{s}^u-W_t^u)_{s\ge t}$ the solution is  predictable with respect to the corresponding filtration, i.e. measurable for the corresponding $\sigma$-algebra on $[t,T]$, say $\Pc^t$;  measurability of $(\omega,s,x,w,z)\mapsto X_s^{\nu,u,x,w,z}(\omega)$ is understood with respect to the $\sigma$-algebra $\Pc^t\otimes \Bc(\R^d\times C^\ell_{[t,T]}\times (0,1))$: see Stricker-Yor and the references therein.
Therefore we may consider the composition 
$X^{\nu,u,\xi^u,W^u_{[0,t]},Z^u}$ obtained substituting $(x,w,z)$ with $(\xi^u,W^u_{[0,t]},Z^u)$. Using the fact that the latter is independent of 
$X^{\nu,u,x,w,z}$ we may see that this is well defined; indeed, if $\tilde X^{\nu,u,x,w,z}$ is another solution to \eqref{Xuxwz} with the same measurability properties we have
$$
G(x,w,z):=
\E[ \sup_{s\in[t,T]} | X_s^{\nu,u,x,w,z}-\tilde X_s^{\nu,u,x,w,z}|^2]=0
$$
for every $x,w,z$ and by independence
$$
\E[ \sup_{s\in[t,T]} | X_s^{\nu,u,\xi^u,W^u_{[0,t]},Z^u}-\tilde X_s^{\nu,u,\xi^u,W^u_{[0,t]},Z^u}|^2]
=\E[G(\xi^u,W^u_{[0,t]},Z^u)] =0
$$
so that $X^{\nu,u,\xi^u,W^u_{[0,t]},Z^u}$ and $\tilde X^{\nu,u,\xi^u,W^u_{[0,t]},Z^u}$
are indistinguishable. By similar arguments one concludes that  the process $X^{\nu,u,\xi^u,W^u_{[0,t]},Z^u}$
  satisfies 
equation
\eqref{stateeqwithnurewritten}
and therefore  it coincides with
 $X^{\nu,u}$, up to a $\P$-null set, for $\lambda$-almost all $u$.

 Our aim is to prove that  the law
 $\Lc (
X^{\nu,u},W^u_{[0,T]},Z^u)$  (a measure on 
$C^d_{[t,T]}\times C^\ell_{[0,T]}\times  (0,1)$) depends in a   measurable way on $u\in U$. To this end we  use the criterion \eqref{meascriterion} and we 
consider the integral of an arbitrary   bounded continuous function $\Phi:C^d_{[t,T]}\times
C^\ell_{[0,T]}\times (0,1)\to \R$ that we  write  in the form
\begin{align}
\int_{C^d_{[t,T]}\times C^\ell_{[0,T]}\times  (0,1)}
\Phi(x,w,z)\,
\Lc (
X^{\nu,u},W^{u}_{[0,T]},Z^u)(\d x\,\d w\,\d z)
=
\E\bigg[ {\Phi}\left(
X^{\nu,u}, W^u_{[0,T]}, 
Z^u\right)\bigg]
\\
\qquad =
\E\bigg[ {\Phi}\left(
X^{\nu,u,\xi^u,W^u_{[0,t]},Z^u}, W^u_{[0,T]}, 
Z^u\right)\bigg]
=
\E\bigg[\tilde{\Phi}\left(
X^{\nu,u,\xi^u,W^u_{[0,t]},Z^u}, W^u_{[t,T]}, 
W^u_{[0,t]},Z^u\right)\bigg],
\end{align}
using the notation $\tilde{\Phi}$ introduced in \eqref{operatorbetaphi}.
We recall that we may replace $W^u_{[t,T]}$ by its increments 
$W^u_{[t,T]}-W^u_{t}$, which are independent of $(W^u_{[0,t]},Z^u)$.
Noting that  equation \eqref{Xuxwz} is driven by  $W^u_{[t,T]}-W^u_{t}$ we conclude that the pair $(X^{\nu,u,x,w,z},W^u_{[t,T]}-W^u_{t})$ is independent of $(W^u_{[0,t]},Z^u,\xi^u)$. 
It follows that 
%\textcolor{red}{Idris: also need a function representation of the SDE a la Roger-Williams. Marco: at this point we just use independence and the freezing lemma}
\begin{align}&
\E\bigg[ {\Phi}\left(
X^{\nu,u}, W^u_{[0,T]}, 
Z^u\right)\bigg]
\\&\qquad
=
\int_{
C^\ell_{[0,t]}\times  (0,1)\times \R^d
}
\E\bigg[\tilde{\Phi}\left(
X^{\nu,u,x,w,z}, W^u_{[t,T]}, 
w,z\right)\bigg]\,
\Lc (
W^u_{[0,t]},Z^u,\xi^u)(\d w\,\d z\,\d x).
\end{align}

Now let us take an arbitrary complete probability space $(\hat\Omega,\hat\Fc, \hat\P)$ with an  $\R^\ell$-valued standard Brownian motion  $\hat W$ and denote $\hat\F=(\hat\Fc_t)_{t\in [0,T]}$ the corresponding completed Brownian filtration.
For every $u\in U$, $x\in\R^d$, $w\in C^\ell_{[0,t]}$, $z\in (0,1)$ we consider the equation
\begin{equation}\label{stateeqauxnu}
     \left\{
\begin{array}{l}
    d \hat X^{\nu,u,x,w,z}_s\!\! =  \sigma\!\left(\!u,\hat X^{\nu,u,x,w,z}_s\!,\hat\alpha^{u,w,z}_s\!, \left(  \nu_{s}^v \right)_{v},\!\left(     
 \int_{ C^\ell_{[0,T]}\times  (0,1)}
    \hat\P_{\hat\alpha_s^{v,w',z'}}\,
(\mathbb{W}_T\times m)(\d w'\,\d z')
\right)_{v}\right)\d \hat W_s
    \\
    \hat X_t^{\nu,u,x,w,z}=x,
    \\
\hat\alpha_s^{u,w,z}=\tilde{\alpha}(u,s, \hat W_{\cdot\wedge s},w,z).
\end{array}\right.
 \end{equation}
Since $\hat W$ does not depend on $u$, this is a stochastic equation depending measurably on all the parameters $u,x,w,z$ (including $u$) and it admits as its solution a measurable function $(\omega,s,u,x,w,z)\mapsto \hat X_s^{\nu,u,x,w,z}(\omega)$. 
Similar as before, measurability is understood with respect to the $\sigma$-algebra $\Pc^t\otimes \Bc(U\times \R^d\times C^\ell_{[t,T]}\times (0,1))$: see Stricker-Yor.
Comparing this equation with \eqref{Xuxwz} we see that they have the same coefficients and are both driven by the increments of Brownian motions - $W^u$ and $\hat W$ respectively - on the interval $[t,T]$. 
As we have strong uniqueness to the considered SDEs, we conclude that on the space $C^d_{[t,T]}\times C^\ell_{[t,T]}$ the law of $(X^{\nu,u,x,w,z},W^u-W^u_t)$ under $\P$ is the same as the  
law of $(\hat X^{\nu,u,x,w,z},\hat W-\hat W_t)$ under $\hat\P$. It follows that
\begin{align}\label{formulaforjointlaw}&
\E\bigg[\Phi\left(
X^{\nu,u}, W^u_{[0,T]}, 
Z^u\right)\bigg]
\\&\qquad
=
\int_{
C^\ell_{[0,t]}\times  (0,1)\times \R^d
}
\hat\E\bigg[\tilde\Phi\left(
\hat X^{\nu,u,x,w,z}, \hat W_{[t,T]}, 
w,z\right)\bigg]\,
\Lc (
W^u_{[0,t]},Z^u,\xi^u)(\d w\,\d z\,\d x).
\end{align}   
Now what occurs under the sign $\hat\E$ is a measurable function of $u,x,w,z$. Since we are assuming that
$\Lc (
W^u_{[0,t]},Z^u,\xi^u)$ depends measurably on $u$ we conclude that $\Lc (
X^{\nu,u},W^u_{[0,T]},Z^u)$ is also a measurable function of $u$ (as a measure on 
$C^d_{[t,T]}\times C^\ell_{[0,T]}\times  (0,1)$).
 This concludes the proof of the claim \eqref{measurabilityclaim}.

We note for later use that, when the function $\Phi$ only depends on its first argument, equation 
\eqref{formulaforjointlaw} 
becomes a formula for the map 
 ${\bf \Psi}(\nu)=({\bf \Psi}(\nu)^u)_u$ introduced at the beginning of the proof. Indeed, recalling that ${\bf \Psi}(\nu)^u= \P_{X^{\nu,u}}$, it follows from \eqref{formulaforjointlaw} that for every continuous bounded $\phi: C^d_{[t,T]}\to \R$, 
\begin{align}\label{formulaforPsi}
\int_{
C^d_{[t,T]} 
} \phi(x)\, 
{\bf \Psi}(\nu)^u(\d x)&=
\E\bigg[\phi\left(
X^{\nu,u}\right)\bigg]
\\&
=
\int_{
C^\ell_{[0,t]}\times  (0,1)\times \R^d
}
\hat\E\bigg[ \phi\left(
\hat X^{\nu,u,x,w,z}\right)\bigg]\,
\Lc (
W^u_{[0,t]},Z^u,\xi^u)(\d w\,\d z\,\d x).
\end{align}   
 
\medskip

{\it Step II: proof of claim \eqref{fixedpointPsi}.}

For any $r\in[t,T]$ we consider the space $L^2_\lambda(\Pc_2(C_{[t,r]}^d))$   with the cor\-respon\-ding distance, that will be denoted $\bd_t^r$.
Given  $\nu,\mu\in L^2_\lambda(\Pc_2(C_{[t,T]}^d))$, let  $(X^{\nu,u})_u$,  $(X^{\mu,u})_u$ denote the corresponding solutions to
\eqref{stateeqwithnu}. 
For suitable constants 
$C_1,C_2$
we obtain
\begin{align}
&    
\E\Big[\sup_{s\in [t,r]}|X_s^{\nu,u}-X_s^{\mu,u}|^2\Big]
\\&\quad
\le C_1\,\E\,\int_t^r\Big|
\sigma\left(u,X^{\nu,u}_s,\alpha^u_s, \left(\nu_s^v \right)_{v},\left( \P_{\alpha^{v}_s} \right)_{v}\right)-
\sigma\left(u,X^{\mu,u}_s,\alpha^u_s, \left(\mu_s^v \right)_{v},\left( \P_{\alpha^{v}_s} \right)_{v}\right)\Big|^2\,\d s
\\&\quad
\le C_2\,\int_t^r\Big\{
\E\,[|X^{\nu,u}_s-X^{\mu,u}_s|^2]+\bd(\nu_s,\mu_s)^2\Big\}\,\d s
\\&\quad
\le C_2\,\int_t^r\Big\{\sup_{q\in [t,s]}
\E\,[|X^{\nu,u}_q-X^{\mu,u}_q|^2]+\bd_t^s(\nu,\mu)^2\Big\}\,\d s.
\end{align}
Next we note that 
$(X^{\nu,u},X^{\mu,u})_u$ (a collection of $\R^{d\times d}$-valued processes) 
satisfies a stochastic equation to which the assumptions of Theorem \ref{exuniq} apply. In particular, it starts 
at time $t$ from the initial condition  $(\xi^{u},\xi^{u})_u$, which is admissible, since the map $u\mapsto \Lc(\xi^{u},\xi^{u}, W_{[0,t]},Z^u) $ is Borel measurable. 
So we can apply the already proved claim \eqref{measurabilityclaim} and conclude in particular that 
$u\mapsto \Lc(X^{\nu,u},X^{\mu,u}) $ is Borel measurable.  It follows that both sides of the previously displayed inequality are measurable functions of $u$.
Integrating with respect to $\lambda(du)$ and applying the Gronwall lemma yields
\begin{equation}
    \label{Phicontinuousestimate}
\bd_t^r(\Psi(\nu),\Psi(\mu))^2\le \int_U
\E\Big[\sup_{s\in [t,r]}|X_s^{\nu,u}-X_s^{\mu,u}|^2\Big]\,\lambda(\d u) \le
C   \int_t^r 
 \bd_t^s(\nu,\mu)^2\, \,\d s, 
 \qquad r\in [t,T],
\end{equation}
for some constant $C>0$ that only depends on the Lipschitz constants of $b,\sigma$, on $T$ and on $\lambda(U)$. Setting $r=T$ we obtain 
\begin{equation}
\label{Phicontinuousestimatebis}
\bd_t^T(\Psi(\nu),\Psi(\mu))^2 \le \|X^{\nu}-X^{\mu}\|^2
\le 
C \cdot (T-t)\,  \bd_t^T(\nu,\mu)^2
\end{equation}
which proves in particular the continuity of $\Psi$. 
Iterating
 \eqref{Phicontinuousestimate} one proves that
\begin{equation}\label{psiiterates}
\bd_t^r({\bf \Psi}^{(k+1)}(\nu),
{\bf \Psi}^{(k+1)}(\mu))
^2\le
\frac{C^{k+1}}{k!} \int_t^r
\bd_t^s(\nu,\mu)^2(r-s)^k\,\d s. 
 \end{equation}
Choosing an arbitrary $\nu^{(0)}\in L^2_\lambda(\Pc_2(C_{[t,T]}^d))$ and setting $\nu^{(k+1)}={\bf \Psi}(\nu^{(k)})$ for $k\ge 0$, it follows   that
$$ \bd_t^T(\nu^{(k+1)},\nu^{(k)}) ^2\le
\frac{C^k(T-t)^k}{k!}
\bd_t^T(\nu^{(1)},\nu^{(0)})^2 .
$$
Now  standard arguments  allow to conclude  that  the sequence $(\nu^{(k)})_k$ is Cauchy for $\bd_t^T$ and it converges in $L^2_\lambda(\Pc_2(C_{[t,T]}^d))$ to a limit, denoted $\bar \nu$, which is a fixed point of the map  ${\bf \Psi}$. The uniqueness of the fixed point follows from \eqref{psiiterates}. 
The claim \eqref{fixedpointPsi} is proved.
%\red{Stop the proof here}
\qedsymbol

\begin{comment}
We may also remark that the processes $X^{\nu^{(k)}}$,  corresponding to
$\nu^{(k)}$, converge in the space $\Sc_t$ \textcolor{red}{Idris: pb with the norm on $\Sc_t$} to $X^{\bar \nu}$, namely that $\|X^{\nu^{(k)}}- X^{\bar \nu}\|\to 0 $.
Indeed, applying \eqref{Phicontinuousestimatebis}
 to ${\nu^{(k)}}$ and ${\nu^{(j)}}$  we deduce that $(X^{\nu^{(k)}})_k$
 is also a Cauchy sequence in the space $\Sc_t$, hence   converging to a limit
$\bar X\in \Sc_t$.  Recalling \eqref{distScedtT} we see that  $\bd_t^T(\nu^{(k)},\P_{\bar X})\le 
\|X^{\nu^{(k)}}-\bar X\|
\to 0$, so that $\bd_t^T(\bar\nu,\P_{\bar X})=0$,    $\P_{\bar X}=\bar \nu$, and hence  $\bar X=X^{\bar \nu}$. 
\end{comment}

\subsection{Uniqueness in law}

\begin{Proposition}\label{EQUALAWSOLSDE}
Let $t\in[0,T]$ and $\alpha: U\times [0,T]\times C_{[0,T]}^\ell\times (0,1)\to A$ a Borel measurable function
%$$
%\alpha: U\times [0,T]\times C_{[0,T]}^\ell\times (0,1)\to A
%$$
Fix $\xi=(\xi^u)_u\in \Ic_t$ and denote by $X=(X^u)_u$ the unique solution to \reff{stateeq}.

For $u\in U$, we consider an $\R^\ell$-valued random process $(\tilde W^u_t)_{t\geq0}$, a real random variable $\tilde Z^u$ and an $\R^d$-valued random variable $\tilde \xi^u$ 
(possibly defined on a different probability space) 
such that
\begin{equation}
\label{eqinitiallawsbis}
\Lc(\xi^u,W^u_{[0,t]},Z^u)  =  \Lc(\tilde{\xi}^u,\tilde{W}^u_{[0,t]},\tilde{Z}^u)
\end{equation}
for all $u\in U$.
We define $(\tilde{X}^u)_u$ as the unique solution to the SDE
 \begin{equation}%\label{stateeq}
     \left\{
\begin{array}{rcl}
    d \tilde{X}^u_s & = & b\left(u,\tilde{X}^u_s, \tilde{\alpha}^u_s,  \P_{\tilde{X}^{\cdot}_s}, \P_{\tilde{\alpha}^{\cdot}_s} \right)\d s\\
     & & + \sigma\left(u,\tilde{X}^u_s,\tilde{\alpha}^u_s, 
     \P_{\tilde{X}^{\cdot}_s},\P_{\tilde{\alpha}^{\cdot}_s}\right)\d \tilde{W}^u_s,
    \;\; s\in [t,T],
    \\
    \tilde{X}_t^u & = & \tilde{\xi}^u,    \;\; u\in U,
\end{array}\right.
 \end{equation}
where $\tilde \alpha$ is defined by
 \beqs
\tilde{\alpha}_t^u={\alpha}(u,t, \tilde{W}^u_{\cdot\wedge t},\tilde{Z}^u),
\quad t\in [0,T],\,u\in U.
\enqs
Then we have
\beqs
\Lc(X^u,W^u_{[0,T]},Z^u) & = & \Lc(\tilde X^u,\tilde W^u_{[0,T]},\tilde Z^u) 
\enqs
for all $u\in U$.
\end{Proposition}
\proofname. We only  sketch the proof. 
In the proof of Theorem \ref{exuniq} the solution was obtained via a fixed point for the map $\nu\mapsto \Psi(\nu)$ introduced there. The fixed point can be obtained by a Picard iteration scheme $\nu^{n+1}=\Psi(\nu^n)$ starting from $\nu^0=\Lc(\xi)=\Lc(\tilde\xi)$. 
In view of
\eqref{eqinitiallawsbis},
formula
\eqref{formulaforjointlaw}
makes it clear that at each iteration we have
\beqs
\Lc(X^{\nu^n,u},W^u_{[0,T]},Z^u) & = & \Lc(\tilde X^{\nu^n,u},\tilde W^u_{[0,T]},\tilde Z^u) 
\enqs
for all $u\in U$. We know 
 that  the sequence $(\nu^n)$  converges in $L^2_\lambda(\Pc_2(C_{[t,T]}^d))$ to the fixed point.  This allows to pass to the limit in
\eqref{formulaforjointlaw}
 and conclude that
$\Lc(X^u,W^u_{[0,T]},Z^u)  =  \Lc(\tilde X^u,\tilde W^u_{[0,T]},\tilde Z^u)$ as required.
\ep

%We next precise the definition of the convergence of the law of a process starting from an initial time that may vary with the sequence.

%For that, we need to define the distance 
%\beqs
%d(\mu,\nu) & = & \inf_{\pi\in \Pi(\mu,\nu)}\int_{C^m_[s,t]\times C^m_[r,t]}\sup_{[0,T]}
%\enqs
%for $\mu\in C^m_{[s,t]}$ and $\nu\in C^m_{[r,t]}$

\section{Some auxiliary results} \label{secappen}

\begin{Proposition}\label{respresentation THM}
 Let $t\in[0,T]$ and $\xi\in \Ic_t$. Then, there exists a Borel map $\tilde{\xi}: U\times 
C^\ell_{[0,t]}\times (0,1)
\to\R^d$ such that
%
%it is $\P$-almost surely equal to a variable of the form $\underline{\xi}^u(W^u_{[0,t]},Z^u)$ for a measurable function $\underline{\xi}^u:
%C^\ell_{[0,t]}\times (0,1)
%\to\R^d$.
%
such that 
\beqs
\Lc\big(\tilde{\xi}^u( W_{[0,t]},Z), W_{[0,t]},Z\big) & = & \Lc\big({\xi}^u, W^u_{[0,t]},Z^u\big)
\enqs
for every $u\in U$ and for any choice of the random pair $(W,Z)$ (defined on an arbitrary probability space), where $W=(W_t)_{t\in [0,T]}$ is an $\R^\ell$-valued standard Brownian motion and $Z$ is a real random variable having uniform distribution in $(0,1)$ and independent of $W$.
\end{Proposition}

 We first need to prove the following result.
 
\begin{Lemma}\label{respresentation LEM}
For any family $(Y^u)_u$ of random variables uniformly distributed on $(0,1)$ and any family $(\Phi^u)_u$ of Borel maps from $(0,1)$ to some Polish space $S$ such that $u\mapsto \Lc(\Phi^u(Y^u))$ is Borel measurable, there exists a Borel map $\tilde \Phi:~U\times (0,1)\rightarrow S$ such that $\Lc(\Phi^u(Y^u))=\Lc(\tilde \Phi(u,Y))$ for every  $u\in U$ where $Y$ is any random variable uniformly distributed on $(0,1)$ (defined on an arbitrary probability space).
\end{Lemma}
\proofname. 
To prove this claim note that setting
$Q(u,A)= \Lc(\Phi^u(Y^u))(A)$, for $u\in U$ and any Borel set $A\subset U$, we define a transition kernel from $U$ to $S$, by the measurability assumption. It is a classical result of Skorohod (used in the proof of the Skorohod representation theorem) that there exists a Borel function $\tilde\Phi(u,\cdot):(0,1)\to S$ carrying the Lebesgue measure on $(0,1)$ to the measure $Q(u,\cdot)$ and therefore satisfying $\Lc(\tilde \Phi(u,Y))=Q(u,\cdot)
=\Lc(\Phi^u(Y^u))$. 
The function
$\tilde\Phi(u,\cdot)$ is obtained from $Q(u,\cdot)$ in a constructive way which shows that, since $Q$ is a kernel, the function $\tilde\Phi(u,y)$ is in fact Borel measurable in $(u,y)\in U\times (0,1)$: for a detailed proof see for instance the proof of Theorem 3.1.1 in \cite{Zab96}. 
%J. Zabczyk. Chance and decision.
%Stochastic control in discrete time. Quaderni della Scuola Normale Superiore di Pisa.
\qedsymbol

\vspace{2mm}

\noindent \proofname ~of Proposition  \ref{respresentation THM}. 
  Fix a bijection  $\psi: C^\ell_{[0,t]}\times (0,1)\rightarrow (0,1)$ such that $\psi$ and $\psi^{-1}$ are Borel measurable (such a map exists 
since $C^\ell_{[0,t]}\times (0,1)$ is an uncountable Polish space: see e.g.  Corollary 7.16.1 in \cite{BertsekasShreve}). Then set $X^u=\psi(W_{[0,t]}^u,Z^u)$, for $u\in U$. Then $(X^u)_{u\in U}$ is a family of identically distributed random variables with common c.d.f. denoted by $F$ that is continuous as $\psi$ is one to one and $\Lc(W_{[0,t]}^u,Z^u)$ has no atom. In particular, the random variables $Y^u:=F(X^u)$, $u\in U$, are uniformly distributed on $[0,1]$ and $(W_{[0,t]}^u,Z^u)=\psi^{-1}(F^{-1}(Y^u))$ for $u\in U$, where $F^{-1}$ stands for the generalized inverse of $F$. Then we have
\beqs
\Lc\big(\underline{\xi}^u( W^u_{[0,t]},Z^u), W^u_{[0,t]},Z^u\big) & = & \Lc\big(\Phi^u(Y^u) \big)
\enqs
where 
 $\Phi^u(y)\in S:=\R^d\times C^\ell_{[0,t]}\times (0,1)$
 is defined as $\Phi^u(y)=\big(\underline{\xi^u}(\psi^{-1}( F^{-1} (y))),\psi^{-1}( F^{-1} (y))\big)$ for $y\in (0,1)$ and $u\in U$. Define $Y=F(\psi( W_{[0,t]},Z))$. Then, $Y$ is uniformly distributed and there exists some Borel map $\tilde \Phi:~U\times (0,1)\rightarrow S$ such that $\Lc(\Phi^u(Y^u))=\Lc(\tilde \Phi(u,Y))$ for every $u\in U$. If we denote by $\tilde \Phi_1$ the first component of $\tilde \Phi$, the map $\tilde \xi ^u= \tilde \Phi_1(u,F(\psi ( \cdot)))$ is a solution to our initial problem. 
\ep

\begin{Proposition}\label{repseqTHM}
Let $(t_n, \mu_n)_n$ be a sequence of $ [0,T]\times L^2_\lambda(\Pc_2(\R^d))$ and $(t,\mu)\in [0,T] \times L^2_\lambda(\Pc_2(\R^d))$ such that $(t_n, \mu_n)\rightarrow (t, \mu)$ as $n\rightarrow+\infty$.  There exist Borel maps $\underline{\xi}$  and  $({\underline \xi}_n)_n$  from $U\times C^\ell_{[0,t]}\times (0,1)$ to $\R^d$ s.t.  
$(\P_{\underline{\xi}^{u}(W^u_{[0,t]}, Z^u)})_u= \mu$,
%and  a sequence $(\xi^n)_n$ such that $\xi^n\in\Ic_{t_n}$,
$(\P_{\underline{\xi}^{n,u}(W^u_{[0,t_n]}, Z^u)})_u=\mu_n$ for all $n\geq 1$ %and $\P_{(\xi^{n,u},\alpha_{t_n}^{n,u})_u}= \pi^n\in B_{\Pc_2(\R^d\times A)}(\pi,\eps)$ for all $n\geq 1$
and
\beqs
\int_0^T\E\Big[\big|\underline{\xi}^{n,u}(W^u_{[0,t_n]}, Z^u)-\underline{\xi}^{u}(W^u_{[0,t]}, Z^u)\big|^2
%+|\alpha^{n,u}-\alpha^{u}_t|^2
\Big]\lambda(\d u) & \xrightarrow[n\rightarrow+\infty]{} & 0\;.
\enqs 
\end{Proposition}
\proofname.
As in the proof of the Proposition \ref{respresentation THM},  we simply need to prove the following statement:

\vspace{1mm}

For any family $(Y^u)_u$ of random variables uniformly distributed on $(0,1)$ and any Polish space $S$, 
%and any family $(\Phi^u)_u$ of Borel maps from $(0,1)$ to some Polish space $S$ such that $u\mapsto \Lc(\Phi^u(Y^u))$ is Borel measurable, 
there exists Borel maps $\Phi$ and $\Phi_n$, $n\geq 1$, from $U\times (0,1)$ to $S$ such that $\Lc(\Phi(u,Y^u))=\mu^u$, $\Lc( \Phi_n(u,Y^u))=\mu^u_n$ for every  $u\in U$ 
and
\beqs
\int_0^T\E\Big[\big|\Lc(\Phi(u,Y^u))-\Lc(\Phi_n(u,Y^u))\big|^2
%+|\alpha^{n,u}-\alpha^{u}_t|^2
\Big]\lambda(\d u) & \xrightarrow[n\rightarrow+\infty]{} & 0\;.
\enqs 
Let $(\pi^n)_n$ be a sequence $L^2_\lambda(\Pc_2(\R^d\times\R^d))$ such that
$\pi^u_{n,1}=\mu^u$ and $\pi^u_{n,2}=\mu^u_n$ for all $u\in U$, and 
\beqs
\int_0^T\int_{\R^d\times\R^d} |x-y|^2\pi_n^u(\d x,\d y)  & \leq & \int_0^T\Wc_2^2(\mu_n^u,\mu^u)\lambda(\d u)+\frac{1}{n}.  
\enqs
We now disintegrate the measure %\textcolor{red}{Idris: ref?}
$\pi_n$ by writing
\beqs
\pi_n^u(dx,dy) & = & \mu^u(dx)\gamma^u_n(x,\d y)
\enqs
for $n\geq1$. From Lemma \ref{respresentation LEM}, there exists a Borel function $\tilde \Phi:~U\times (0,1)\rightarrow S$ such that $\Lc(\Phi(u,Y^u))=\mu^u$ for $u\in U$. Still using Lemma \ref{respresentation LEM} with $U\times \R^d$ in place of $U$,
there exist a Borel functions  $\psi_n:~U\times S \times (0,1)\rightarrow S$ 
 and $\Lc(\psi(u,x,Y^u))=\gamma_n^u(x,.)$ for $x\in S$ and $u\in U$.  Now take a Borel map $\zeta:~(0,1)\rightarrow (0,1)\times(0,1)$ such that $\zeta=(\zeta_1,\zeta_2)\sim \Uc_{(0,1)\times(0,1)}$ and define the Borel maps $\Phi$ and $\Phi_n$ by
\beqs
\Phi(u,x) \; = \;  \tilde \Phi(u,\zeta_1(x)), & \qquad 
\Phi_n(u,x) \;  = \;  \psi_n(u,\tilde \Phi(\zeta_1(x)),\zeta_2(x))
\enqs
for $x\in (0,1)$ and $u\in U$ and $n\geq 1$. Then $\Phi$ and $\Phi_n$ are solutions to the problem.
%From Skorokhod representation Theorem, there exists random variables $(X,Y)$ and $(X_n,Y_n)$, $n\geq1$, from $(U\times (0,1),\Bc(U)\otimes\Bc(0,1),\lambda\otimes\ell)$ to $(S,\Bc(S))$ such that $\Lc(X_n,Y_n)=\pi_n$, $n\geq1$. % $\Lc(X,Y)=\delta_$   
%\vspace{2mm}
%s
\ep

\begin{Proposition}\label{WANTED}
    Let $\alpha\in \Ac$, $t\in [0,T]$, $(t_n)_n$ a sequence of $[0,T]$, $\xi\in \Ic_t$ and $(\xi^n)_n$ a sequence such that  $\xi^n\in \Ic_{t_n}$ for all $n\geq 1$. Suppose that  
\beqs
\big(t_n,\P_{\xi^{n,\cdot}}\big) & \xrightarrow[n\rightarrow+\infty]{} & \big(t,\P_{\xi^\cdot}\big)
\enqs
in $\R\times L^2_\lambda(\Pc_2(\R^d))$.  Let $X$ and $(X^n)$ be the respective solutions to \reff{stateeq} with initial conditions $\xi$ and $\xi^n$ at time $t$ and $t_n$ and control $\alpha$.
Then we have
\beqs
(\P_{X^{n,u}_{.\vee t_n}})_u & \xrightarrow[n\rightarrow+\infty]{} & (\P_{X^u_{.\vee t}})_u
\enqs
%are Borel measurable as   mappings from $U$ to $\Pc_2(C^\ell_{[0,t]}\times (0,1)\times\R^d\times\R^d)$ 
%
%    and
%    $$
%    \int_U \E\,|\xi_n^u-\xi^u|^2\,\lambda(du)\to 0
%    $$
%    as $n\to\infty$. Then $J(t,\xi_n,\alpha)\to J(t,\xi,\alpha)$ and $\upsilon(t,\xi_n)\to \upsilon(t,\xi)$.
in $L^2_\lambda(\Pc_2(C^d_{[0,T]}))$.
\end{Proposition}
\proofname. We take $W$ a $\R^d$ valued brownian motion and $Z$ an independent $(0,1)$-uniformly distributed random variable. From Proposition \ref{repseqTHM}, there are Borel maps $\underline{\xi}$  and  $({\underline \xi}_n)_n$  from $U\times 
C^\ell_{[0,t]}\times (0,1)
$ to $\R^d$ 
such that 
$(\P_{\underline{\xi}^{u}(W_{[0,t]}, Z)})_u= (\P_{{\xi}^{u}})_u$,
$(\P_{\underline{\xi}^{n,u}(W_{[0,t_n]}, Z)})_u=(\P_{{\xi}^{n,u}})_u$ for all $n\geq 1$ 
and
\beq\label{condIniConvL2}
\int_0^T\E\Big[\big|\underline{\xi}^{n,u}(W_{[0,t_n]}, Z)-\underline{\xi}^{u}(W_{[0,t]}, Z)\big|^2
\Big]\lambda(\d u) & \xrightarrow[n\rightarrow+\infty]{} & 0\;.
\enq
 We define $(\tilde{X}^u)_u$ as the unique solution to the SDE
 \begin{equation}%\label{stateeq}
     \left\{
\begin{array}{rcl}
    d \tilde{X}^u_s & = & b\left(u,\tilde{X}^u_s, \tilde{\alpha}^u_s,  \P_{\tilde{X}^{\cdot}_s} , \P_{\tilde{\alpha}^{\cdot}_s}  \right)  \d s\\
     & & + \sigma\left(u,\tilde{X}^u_s,\tilde{\alpha}^u_s, \P_{\tilde{X}^{\cdot}_s}, \P_{\tilde{\alpha}^{\cdot}_s}  \right)\d {W}_s,
    \;\; s\in [t,T],
    \\
    \tilde{X}_t^u & = & \underline{\xi}^{u}(W_{[0,t]}, Z),    \;\; u\in U,
\end{array}\right.
 \end{equation}
 and $(\tilde{X}^{n,u})_u$ as the unique solution to the SDE
 \begin{equation}%\label{stateeq}
     \left\{
\begin{array}{rcl}
    d \tilde{X}^{n,u}_s & = & b\left(u,\tilde{X}^{n,u}_s, \tilde{\alpha}^u_s, \big( \P_{\tilde{X}^{n,v}_s} \big)_{v}, \big( \P_{\tilde{\alpha}^{n,v}_s} \big)_{v}\right)\d s\\
     & & + \sigma\left(u,\tilde{X}^{n,u}_s,\tilde{\alpha}^u_s, \big( \P_{\tilde{X}^{n,v}_s} \big)_{v},\big( \P_{\tilde{\alpha}^{v}_s} \big)_{v}\right)\d {W}_s,
    \;\; s\in [t,T],
    \\
    \tilde{X}_{t_n}^{n,u} & = & \underline{\xi}^{n,u}(W_{[0,t_n]}, Z),    \;\; u\in U,
\end{array}\right.
 \end{equation}
 where $\tilde \alpha$ is defined by
 \beqs
\tilde{\alpha}_t^u=\tilde{\alpha}(u,t, W_{\cdot\wedge t},Z),
\quad t\in [0,T],\,u\in U.
\enqs
Using Proposition \ref{EQUALAWSOLSDE} with $\tilde{W}^u=W$ and $\tilde{Z}^u=Z$ for $u\in U$, we get
\beqs
\Lc(X^u,W^u,Z^u) & = & \Lc(\tilde X^u, W,Z) 
\enqs
and 
\beqs
\Lc(X^{n,u},W^u,Z^u) & = & \Lc(\tilde X^{n,u},W,Z) 
\enqs
for all $n\geq 1$ and $u\in U$. Now, using \reff{condIniConvL2} we get from classical estimates on diffusion processes that
\beqs
\int_0^T\E\Big[\sup_{s\in[t\vee t_n,T]}\big|\tilde{X}^{n,u}_s-\tilde{X}^{u}_s\big|^2
\Big]\lambda(\d u) & \xrightarrow[n\rightarrow+\infty]{} & 0\;,
\enqs
which gives the result.
\ep 

\bigskip

\noindent {\bf Acknowledgments.}  Marco Fuhrman wishes to thank Huy\^en Pham for the invitation  to Universit\'e Paris Cit\'e where this work began.

\bibliographystyle{plain}
\bibliography{biblio}

\end{document}